\documentclass[onefignum,onetabnum]{siamonline171218}

\usepackage{graphicx}

\usepackage{makecell}
\usepackage[utf8]{inputenc} 
\usepackage[T1]{fontenc}    
\usepackage{hyperref}       
\usepackage{url}            
\usepackage{booktabs}       
\usepackage{amsfonts}       
\usepackage{amsmath}
\usepackage{amssymb}
\usepackage{nicefrac}       
\usepackage{microtype}      
\usepackage{xcolor}         
\usepackage{graphicx}
\usepackage{enumitem}
\usepackage{multirow}

\usepackage{comment}
\numberwithin{equation}{section}

\usepackage{wrapfig}

\DeclareMathOperator*{\argmin}{arg\,min}
\def\1{{\mathbf 1}}

\def\mF{{\mathcal F}}
\def\mP{{\mathcal P}}

\def\mE{{\mathcal E}}

\def\mM{{\mathcal M}}

\def\mT{{\mathcal T}}
\def\mS{{\mathcal S}}
\def\mG{{\mathcal G}}
\def\mN{{\mathcal N}}

\def\R{{\mathbb R}}

\def\eps{{\varepsilon}}

\def\bE{{\mathbb{E}}}
\def\gt{{\rightarrow}}

\usepackage{amsmath,amsfonts,bm}

\def\eqref#1{equation~\ref{#1}}

\def\1{\bm{1}}

\def\eps{{\epsilon}}

\def\mE{{\bm{E}}}
\def\mF{{\bm{F}}}
\def\mG{{\bm{G}}}

\def\mJ{{\bm{J}}}
\def\mK{{\bm{K}}}
\def\mL{{\bm{L}}}
\def\mM{{\bm{M}}}
\def\mN{{\bm{N}}}

\def\mP{{\bm{P}}}

\def\mS{{\bm{S}}}
\def\mT{{\bm{T}}}

\DeclareMathAlphabet{\mathsfit}{\encodingdefault}{\sfdefault}{m}{sl}
\SetMathAlphabet{\mathsfit}{bold}{\encodingdefault}{\sfdefault}{bx}{n}

\usepackage{amsmath,amsfonts,bm}

\def\eqref#1{equation~\ref{#1}}

\def\1{\bm{1}}

\def\eps{{\epsilon}}

\def\mE{{\bm{E}}}
\def\mF{{\bm{F}}}
\def\mG{{\bm{G}}}

\def\mJ{{\bm{J}}}
\def\mK{{\bm{K}}}
\def\mL{{\bm{L}}}
\def\mM{{\bm{M}}}
\def\mN{{\bm{N}}}

\def\mP{{\bm{P}}}

\def\mS{{\bm{S}}}
\def\mT{{\bm{T}}}

\DeclareMathAlphabet{\mathsfit}{\encodingdefault}{\sfdefault}{m}{sl}
\SetMathAlphabet{\mathsfit}{bold}{\encodingdefault}{\sfdefault}{bx}{n}

\def\mF{{\mathcal F}}
\def\mP{{\mathcal P}}

\def\mE{{\mathcal E}}

\def\mM{{\mathcal M}}

\def\mT{{\mathcal T}}
\def\mS{{\mathcal S}}
\def\mG{{\mathcal G}}
\def\mN{{\mathcal N}}

\def\R{{\mathbb R}}

\def\eps{{\varepsilon}}

\def\bE{{\mathbb{E}}}
\def\gt{{\rightarrow}}

\usepackage{hyperref}
\usepackage{url}

\usepackage{lipsum}
\usepackage{amsfonts}
\usepackage{graphicx}
\usepackage{epstopdf}
\usepackage{algorithmic}
\ifpdf
  \DeclareGraphicsExtensions{.eps,.pdf,.png,.jpg}
\else
  \DeclareGraphicsExtensions{.eps}
\fi

\usepackage{enumitem}
\setlist[enumerate]{leftmargin=.5in}
\setlist[itemize]{leftmargin=.5in}

\newsiamremark{remark}{Remark}
\newsiamremark{hypothesis}{Hypothesis}
\crefname{hypothesis}{Hypothesis}{Hypotheses}
\newsiamthm{claim}{Claim}

\headers{Statistical Numerical PDE}{Preprint Version}

\title{Machine Learning For Elliptic PDEs:\\ Fast Rate Generalization Bound, Neural Scaling Law and Minimax Optimality}

\author{Yiping Lu\thanks{Institute for Computational and Mathematical Engineering, Stanford University, Stanford, CA 
  (\email{yplu@stanford.edu}, \url{https://web.stanford.edu/\string~yplu/}).}
\and  Haoxuan Chen\thanks{Department of Computing and Mathematical Sciences, Caltech 
  (\email{haoxuan@caltech.edu}).}
 \and Jianfeng Lu \thanks{Mathematics Department, Duke University (\email{jianfeng@math.duke.edu})}
\and Lexing Ying\thanks{Department of Mathematics and
Institute for Computational and Mathematical Engineering, Stanford University, Stanford, CA (\email{lexing@stanford.edu}})
\and Jose Blanchet \thanks{Department of Management Science \& Engineering, Stanford University, Stanford, CA (\email{jose.blanchet@stanford.edu})}}

\usepackage{amsopn}

\makeatletter
\newcommand*{\addFileDependency}[1]{
  \typeout{(#1)}
  \@addtofilelist{#1}
  \IfFileExists{#1}{}{\typeout{No file #1.}}
}
\makeatother

\newcommand*{\myexternaldocument}[1]{%
    \externaldocument{#1}%
    \addFileDependency{#1.tex}%
    \addFileDependency{#1.aux}%
}

\ifpdf
\hypersetup{
  pdftitle={Statistical Numerical PDE},
  pdfauthor={Lu. et al.}
}
\fi

\myexternaldocument{ex_supplement}

\begin{document}

\maketitle

\begin{abstract}
    In this paper, we study the statistical limits of deep learning techniques for solving elliptic partial differential equations (PDEs) from random samples using the Deep Ritz Method (DRM) and Physics-Informed Neural Networks (PINNs). To simplify the problem, we focus on a prototype elliptic PDE: the Schr\"odinger equation on a hypercube with zero Dirichlet boundary condition, which is applied in quantum-mechanical systems. We establish upper and lower bounds for both methods, which improve upon concurrently developed upper bounds for this problem via a fast rate generalization bound. We discover that the current Deep Ritz Method is sub-optimal and propose a modified version of it. We also prove that PINN and the modified version of DRM can achieve minimax optimal bounds over Sobolev spaces. Empirically, following recent work which has shown that the deep model accuracy will improve with growing training sets according to a power law, we supply computational experiments to show similar-behavior of dimension dependent power law for deep PDE solvers.
\end{abstract}


\section{Introduction}
Partial differential equations (PDEs) play a prominent role in many disciplines of science and engineering. The recent deep learning breakthrough and the rapid development of sensors, computational power, and data storage in the past decade has drawn attention to numerically solving PDEs via machine learning methods \cite{long2018pde,long2019pde,raissi2019physics,han2018solving,sirignano2018dgm,khoo2017solving}, especially in high dimensions where conventional methods become impractical. The set of applications that motivate this interest is wide-ranging, including computational physics  \cite{han2018solving,long2018pde,raissi2019physics}, inverse problem  \cite{zhang2018dynamically,gilton2019neumann,fan2020solving} and quantitative finance  \cite{heaton2017deep,germain2021neural}. The numerical methods generated by the use of deep learning techniques are mesh-less methods, see the discussion in \cite{xu2020finite}. A natural deep learning technique in the problems that are based on a standard feed-forward type of architecture takes advantage (when available) of a variational formulation, whose solution coincides with the solution of the PDE of interest. Despite the success and popularity of adopting neural networks for solving high-dimensional PDEs, the following question still remains poorly answered.

\begin{quotation}
{\emph{For a given PDE and a data-driven approximation architecture, how large the sample size and how complex the model are needed to reach a prescribed performance level?} }
\end{quotation}

In this paper, we aim to establish the numerical analysis of such deep learning based PDE solvers. Inspired by recent works which showed that the empirical performance of a model is remarkably predictable via a power law of the data number, known as the neural scaling law  \cite{kaplan2020scaling,hestness2017deep,sharma2020neural}, we aim to explore the neural scaling law for deep PDE solvers and compare its performance to Fourier approximation. 

Among the various approaches of using deep learning methods for solving PDEs, in this work, we focus on the Deep Ritz method (DRM) \cite{weinan2018deep,khoo2017solving} and the Physics-Informed Neural Networks (PINN) approach \cite{sirignano2018dgm,raissi2019physics}, both of which are based on minimizing neural network parameters according to some loss functional related to the PDEs. 
To provide theoretical guarantees for DRM and PINN, following   \cite{lu2021priori,duan2021convergence,bai2021physics}, we decompose the error into approximation error \cite{yarotsky2017error,suzuki2018adaptivity,shen2021neural} and generalization error \cite{bartlett2005local,xu2020towards,farrell2021deep,schmidt2020nonparametric,suzuki2018adaptivity}. However, instead of the $O(1/\sqrt{n})$ ($n$ is the number of data sampled) slow rate generalization bounds established in prior work  \cite{lu2021priori,shen2021neural,xu2020finite,shin2020error}, we utilize the strongly convex structure of the DRM and PINN objectives and provide an $O(1/n)$ fast rate generalization bound \cite{bartlett2005local,xu2020towards} that leads us to a non-parametric estimation bound. Our theory also suggests an optimal selection of network size with respect to the number of sampled data. Moreover, to illustrate the optimality of our upper bound, we also establish an information-theoretic lower bound which matches our upper bound for PINN and a modified version of DRM.

We also test our theory by numerical experiments. Recent
works \cite{hestness2017deep,kaplan2020scaling,rosenfeld2019constructive,mikami2021scaling} studying a variety of deep learning algorithms all find the same polynomial scaling relation between the testing error and the number of data. As the number of training data $n$ increases, the population loss $\mathcal{L}$ of well-trained and well-tuned models scales with $n$ as a power-law $\mathcal{L}\propto\frac{1}{n^\alpha}$ for some $\alpha$.  \cite{sharma2020neural} also scans over a large range of $\alpha$ and problem dimension $d$ and finds an approximately $\alpha\propto\frac{1}{d}$ scaling law. In Section 4, we conduct numerical experiments to show that this phenomenon still appears for deep PDE solvers and this neural scaling law tests more idiosyncratic features of the theory.

\subsection{Related Works}

\paragraph{Neural Scaling Law} The starting point of our work is the recent observation across speech, vision and text  \cite{hestness2017deep,kaplan2020scaling,rosenfeld2019constructive,rosenfeld2021scaling} that the empirical performance of a model satisfies a power law scales as a power-law with model size and dataset size.  \cite{sharma2020neural} further finds out that the power of the scaling law depends on the intrinsic dimension of the dataset. Theoretical works  \cite{schmidt2020nonparametric,suzuki2018adaptivity,suzuki2019deep,chen2019nonparametric,imaizumi2020advantage,farrell2021deep,jiao2021deep} explore the optimal power law under the non-parametric curve estimation setting via a plug-in neural network. Our work extends this line of research to solving PDEs.

	\paragraph{Deep Network Based PDE Solver.} Solving high dimensional partial differential equations (PDEs) has been a long-standing challenge due to the curse of dimensionality. At the same time, deep learning has shown superior flexibility and adaptivity in approximating high dimensional functions, which leads to state-of-the-art performances in a wide range of tasks ranging from computer vision to natural language processing. Recent years, pioneer works \cite{han2018solving,raissi2019physics,long2018pde,sirignano2018dgm,khoo2017solving} try to utilize the deep neural networks to solve different types of PDEs and achieve impressive results in many tasks \cite{lu202186,li2020fourier}. Based on the natural idea of representing solutions of PDEs by (deep) neural networks, different loss functions for solving PDEs are proposed. \cite{han2018solving,han2020solvingeigen} utilize the Feynman-Kac formulation which turns solving PDE to a stochastic control problem and the weak adversarial network \cite{zang2020weak} solves the weak formulations of PDEs via an adversarial network. In this paper, we focus on the convergence rate of the Deep Ritz Method (DRM)   \cite{weinan2018deep,khoo2017solving} and Physics-Informed neural network (PINN) \cite{raissi2019physics,sirignano2018dgm}. DRM \cite{weinan2018deep,khoo2017solving} utilizes the variational structure of the PDE, which is similar to the Ritz-Galerkin method in classical numerical analysis of PDEs, and trains a neural network to minimize the variational objective. PINN \cite{raissi2019physics,sirignano2018dgm} trains a neural network directly to minimize the residual of the PDE, i.e., using the strong form of the PDE.

\paragraph{Theoretical Guarantees For Machine Learning Based PDE Solvers.} Theoretical convergence results for deep learning based PDE solvers raises wide interest recently. Specifically,  \cite{lu2021priori,grohs2020deep,marwah2021parametric,wojtowytsch2020some,xu2020finite,shin2020error,bai2021physics} investigate the regularity of PDEs approximated by neural network and  \cite{lu2021priori,luo2020two} further provide a generalization analysis.  \cite{nickl2020convergence} introduces a prior over the solution of the PDE and considers an equivalent white noise model \cite{brown1996asymptotic}. \cite{nickl2020convergence} provides the rate of convergence of the posterior. Our paper does not need to introduce the prior on the target function and provides a non-asymptotic guarantee for finite number of data. At the same time, \cite{nickl2020convergence} can only be applied to linear PDEs while our proof technique can be extend to nonlinear ones. All these papers also fail to answer the question that how to determine the network size corresponding to the sampled data number to achieve a desired statistical convergence rate. \cite{hutter2019minimax,manole2021plugin} consider the similar problem for the optimal transport problem, \emph{i.e.} Monge-ampere equation. Nevertheless, the variational problem we considered is different from \cite{hutter2019minimax,manole2021plugin} and leads to technical difference. The most related works to ours are two \textbf{concurrent} papers \cite{duan2021convergence,jiao2021convergence,jiao2021error}. However, our upper bound is faster than \cite{duan2021convergence,jiao2021convergence,jiao2021error}. In this paper, we also show that generalization analysis in \cite{lu2021priori,duan2021convergence,luo2020two} are loose due to the lack of a localization technique \cite{de1978practical,bartlett2005local,koltchinskii2011oracle,xu2020finite}. With observation of the strong convexity of the loss function, we follow the fast rate results for ERM \cite{schmidt2020nonparametric,xu2020towards,farrell2021deep} and provide a near optimal bound for both DRM and PINN.

\subsection{Contribution} In short, we summarize our contribution as follows
\begin{itemize}
    \item In this paper, we first considered the statistical limit of learning a PDE solution from sampled observations. The lower bound shows a non-standard exponent different from non-parametric estimation of a function.
    \item Instead of the $O(1/\sqrt{n})$ slow rate generalization bounds in  \cite{lu2021priori,duan2021convergence,jiao2021convergence,jiao2021deep},  we utilized the strongly convex nature of the variational form and provided a fast rate generalization bound via the localization methods \cite{van1987new,bartlett2005local,koltchinskii2011oracle,srebro2010smoothness,xu2020towards}. We discovered that the current Deep Ritz Methods is sub-optimal and propose a modified version of it. We showed that PINN and the modified version of DRM can achieve nearly min-max optimal convergence rate. Our result is listed in Table \ref{table:rate}. 
    \item We tested the recently discovered neural scaling law \cite{hestness2017deep,kaplan2020scaling,rosenfeld2019constructive,hashimoto2021predicting} for deep PDE solvers numerically. The empirical results verified our theory.
\end{itemize}

\begin{table}
\begin{tabular}{ |c|c|c|c|c| }
\hline\hline
\multicolumn{4}{|c|}{\textbf{Upper Bounds}}                                                             & \multicolumn{1}{l|}{\multirow{2}{*}{\textbf{Lower Bound}}} \\ \cline{1-4}
    \multicolumn{1}{|c|}{Objective Function} & \multicolumn{1}{c|}{Neural Network} & \multicolumn{1}{c|}{Previous Bound} & \multicolumn{1}{c|}{Fourier Basis} &                          \\ \hline
\hline
Deep Ritz&$n^{-\frac{2s-2}{d+2s-2}}\log n$  & \makecell{$n^{-\frac{2s-2}{d+4s-4}}\log n$ \\ \cite{duan2021convergence}}& $n^{-\frac{2s-2}{d+2s-2}}$ & $n^{-\frac{2s-2}{d+2s-4}}$\\ \hline
Modified Deep Ritz& $n^{-\frac{2s-2}{d+2s-2}}\log n$ &/&{\color{red} $n^{-\frac{2s-2}{d+2s-4}}$}&$n^{-\frac{2s-2}{d+2s-4}}$\\
 \hline
  \hline
PINN &  {\color{red}$n^{-\frac{2s-4}{d+2s-4}}\log n$} &  \makecell{$n^{-\frac{2s-4}{d+4s-8}}\log n$\\ \cite{jiao2021convergence}} &  {\color{red} $n^{-\frac{2s-4}{d+2s-4}}$}  & {$n^{-\frac{2s-4}{d+2s-4}}$}\\
\hline
\end{tabular}
\caption{Upper bounds and lower bounds we achieve in this paper and previous work. The upper bound colored in red indicates that the convergence rate matches the min-max lower bound.}
\label{table:rate}
\end{table}

\section{Set-up}

We consider the static Schr\"odinger equation with zero Dirichlet boundary conditions on the domain $\Omega$, which we assume to be the unit hypercube in $\R^d$. In order to precisely introduce the problem, we recall some standard notions. We consider our domain as  $\Omega = [0,1]^{d}$ and use $L^2(\Omega)$ to  denote the space of square integrable functions on $\Omega$ with respect to the Lebesgue measure. We let $L^\infty(\Omega)$ be the space of essentially bounded (with respect to the Lebesgue measure) functions on $\Omega$ and $C(\partial \Omega)$ denotes the space of continuous functions on $\partial \Omega$.

Let $f\in L^2(\Omega)$, $V\in L^\infty(\Omega)$, and , $g\in L^\infty(\Omega)$. Our focus is on the analysis of Deep-Learning-based numerical methods to solve the elliptic equations 
\begin{equation}\label{eq:maineq}
\begin{aligned}
   -\Delta u + V u & = f \quad \text{ in } \Omega,\\ 
     u & =  g \quad \text{ on } \partial \Omega. 
\end{aligned}
\end{equation}

\subsection{Loss Functions for Solving PDEs and Induced Evaluation Metric}
\label{subsection:lossfunc}

In this paper, we mainly focus on analyzing Deep Ritz Methods (DRM) and Physics Informed Neural Network (PINN). In this subsection, we first introduce the objective function and algorithm of the two methods.

\paragraph{Deep Ritz Methods} \cite{weinan2018deep,sirignano2018dgm} Recall that the equation \ref{eq:maineq} is equivalent to following variational form
\begin{equation}\label{eq:variationalform}
    u^\ast = \arg \min_{H_0^1(\Omega)} \mE^{\text{DRM}}(u):=\frac{1}{2} \int_{\Omega} |\nabla u|^2  +   V |u|^2 \ dx- \int_{\Omega} f u dx, 
\end{equation}
where $u$ is minimized over $H_0^1(\Omega)$ with boundary condition given by $g$ on $\partial \Omega$. 

This variational form provides the basis for the DRM type method for solving the static Schr\"odinger equation based on neural network ansatz. More specifically, the energy functional given in \eqref{eq:variationalform} is viewed as the population risk function to train an optimal estimator approximation of the solution to the PDE within a parameterized hypothesis function class $\mF\subset H^1(\Omega)$. In this paper, we also rely on the strong convexity of the DRM objective respect to the $H^1$ norm.
\begin{proposition}\label{prop:drmconvex} We further assume $0< V_{\min} \leq V(x) \leq V_{\max}$, then we have
$$
\frac{2}{\max\{1,V_{\max}\}}  \left(\mE^{\text{DRM}}(u)-\mE^{\text{DRM}}(u^\ast)\right)\le \|u-u^\ast\|_{H^1}^2\le \frac{2}{\max\{1,V_{\min}\}}  \left(\mE^{\text{DRM}}(u)-\mE^{\text{DRM}}(u^\ast)\right)
$$
holds for all $u\in H_0^1(\Omega)$
\end{proposition}

\paragraph{Physics-Informed Neural Network} \cite{raissi2019physics,sirignano2018dgm}. PINN solves \ref{eq:maineq} via minimizing the following objective function
$$
u^\ast = \arg \min_{H_0^1(\Omega)} \mE^{\text{PINN}}(u):=\arg \min_{H_0^1(\Omega)} \int_\Omega |\Delta u(x)-V(x)u(x) + f(x)|^2 dx.
$$
The objective function $\mE^{\text{PINN}}$ can also be viewed as the population risk function and we can train an optimal estimator approximation of the solution to the PDE within a parameterized hypothesis function class $\mF\subset H_0^1(\Omega)$. In this paper, we also rely on the strong convexity of the PINN objective with respect to the $H^2$ norm, for which we need some additional assumptions on the potential. 

\begin{proposition}\label{prop:pinn} For PINN, we further assume $V\in L^\infty(\Omega)$ with $0<C_{\min}<V^2-\Delta V, 0<C_{\min}< V(x) \leq V_{\max}$ and $-\Delta V(x)\le V_{\max}$, then we have for all $u \in H_0^1(\Omega)$
\begin{multline*}
\frac{1}{2\left(1+V_{\max}+V_{\max}^2\right)}  \left(\mE^{\text{PINN}}(u)-\mE^{\text{PINN}}(u^\ast)\right)\le \|u-u^\ast\|_{H^2}^2 \\\le \frac{2}{\max\{1,C_{\min}\}}  \left(\mE^{\text{PINN}}(u)-\mE^{\text{PINN}}(u^\ast)\right).
\end{multline*}
\end{proposition}

\subsection{Estimator Setting}
\label{subsection:setup}

\paragraph{Empirical Loss Minimization} In order to access the $d$-dimensional integrals, DRM \cite{weinan2018deep,khoo2017solving} and PINN\cite{raissi2019physics,sirignano2018dgm} employ a Monte-Carlo method for computing the high dimensional integrals, which  leads to the so-called {\em empirical risk minimization} training for neural networks.  To define the empirical loss, let $\{X_j\}_{j=1}^n$ be an i.i.d.~sequence of random variables distributed according to the uniform distribution in domain $\mP_\Omega$. We also have access to $f_j = f(X_j)+\xi_i, j=1,\cdots,n$ is the noisy observation of the right hand side of the PDE (\ref{eq:maineq}) and $\xi$ is a bounded random variable with mean zero and independent with $X_j$. Define the empirical losses  $\mE_{n,S}$ by setting
\begin{equation}\begin{aligned}
      \mE^{\text{DRM}}_{n}(u)  & = \frac{1}{n} \sum_{j=1}^n\Big[ |\Omega| \cdot  \Big(\frac{1}{2} |\nabla u(X_j)|^2 + \frac{1}{2} V(X_j) |u(X_j)|^2- f_ju(X_j) \Big)\Big],
    \end{aligned}
\end{equation}
\begin{equation}\begin{aligned}
      \mE^{\text{PINN}}_{n}(u)  & = \frac{1}{n} \sum_{j=1}^n\Big[ |\Omega| \cdot  \Big(\Delta u(X_j)-V(X_j)u(X_j)+f_j\Big)^2\Big],
    \end{aligned}
\end{equation}

where $|\Omega|$ represent the Lebesgue measure of the sets. 

Once given an empirical loss $\mE'_{n}$, we apply the empirical loss minimization to seek the estimation $u_{n}$, i.e. $u_{n} = \argmin_{u\in \mF} \mE_{n}(u)$ where $\mF$ is the parametrized hypothesis function space we consider. For example, reproducing kernel Hilbert space\cite{chen2021solving} and tensor training format\cite{richter2021solving}. In this paper, we consider sparse neural network and truncated fourier basis, which can achieves min-max optimal estimation rate for the non-parametric function estimation\cite{tsybakov2008introduction,schmidt2020nonparametric,farrell2021deep,suzuki2018adaptivity,chen2019nonparametric,jiao2021deep,nitanda2020optimal}.  

\paragraph{Sparse Neural Network Function Space} In this paper,  the hypothesis function space $\mathcal{F}$ is expressed by the neural network following  \cite{schmidt2020nonparametric,suzuki2018adaptivity,farrell2021deep}. Let us denote the ReLU$^3$ activation by $\eta_3(x) = \max\{x^3,0\}~(x \in \mathbb{R})$ which is used in \cite{weinan2018deep},  and for a vector $x$, $\eta(x)$ is operated in an element-wise manner.
Define the neural network with height $L$, width $W$, sparsity constraint $S$ and norm constraint $B$ as 
\begin{align}
& \Phi(L,W,S,B) 
:= \{(\mathcal{W}^{(L)}\eta_3(\cdot) + b^{(L)})\circ \cdots (\mathcal{W}^{(2)}\eta_3(\cdot) + b^{(2)})\circ(\mathcal{W}^{(1)}x + b^{(1)}) \ | \  \nonumber\\
&\mathcal{W}^{(L)} \in \mathbb{R}^{1 \times W}, b^{(L)} \in \mathbb{R},\mathcal{W}^{(1)} \in \mathbb{R}^{W \times d}, b^{(1)} \in \mathbb{R}^{W}, \mathcal{W}^{(l)} \in \mathbb{R}^{W \times W},b^{(l)} \in \mathbb{R}^{W} (1 <l<L),\nonumber\\
&\sum_{l=1}^{L}(\|\mathcal{W}^{(l)}\|_0 + \|b^{(l)}\|_0) \leq S, \max_{l}\|\mathcal{W}^{(l)}\|_{\infty,\infty} \vee \|b^{(l)}\|_{\infty} \leq B\},
\end{align}
where $\circ$ denotes the function composition, $\|\cdot\|_0$ is the $\ell_0$-norm of the matrix (the number of non-zero elements of the matrix) and $\|\cdot\|_{\infty,\infty}$ is the $\ell_\infty$-norm of the matrix (maximum of the absolute values of the elements). 

\paragraph{Truncated Fourier Basis Estimator} We also considered the Truncated  Fourier basis as our estimator. Suppose the domain we interested $\Omega \subseteq [0,1]^{d}$. For any $z \in \mathbb{N}^{d}$, we consider the corresponding Fourier basis function $\phi_{z}(x):= e^{2\pi i \langle z,x \rangle} \ (x \in \Omega)$. Any function $f \in L^2(\Omega)$ can be represented as weighted sum of the Fourier basis
$f(x) := \sum_{z \in \mathbb{N}^{d}}f_z \phi_{z}(x)$
where $f_z := \int_{\Omega}f(x)\overline{\phi_{z}(x)}dx \ (\forall \ z \in \mathbb{N}^{d})$ is the Fourier coefficient. This inspired us to use the Fourier Basis whose index lies in a truncated set $Z_\xi=\{z\in\mathcal{Z}|\|z\|_\infty\le\xi\}$ to represent the function class $\mF$ as $\mF_\xi =\{\sum_{\|z\|_\infty\le \xi} a_z\phi_z|a_z\in \mathbb{R}, \|z\|_\infty\le \xi\}$.

\section{Lower Bound}
\label{section:lowerbound}
In this section, we aim to consider the statistical limit of learning the solution of a PDE. As discussed in Propositions \ref{prop:drmconvex} and \ref{prop:pinn}, we directly consider the $H^1$ norm for DRM and $H^2$ norm for PINN as the evaluation metric. The lower bound shown as follows.
\begin{theorem}[Lower bound]
\label{lowerboundthm}
We denote $u^\ast(f)$ to be the solution of the PDE \ref{eq:maineq} and we can access randomly sampled data $\{X_i,f_i\}_{i=1,\cdots,n}$ as described in Section \ref{subsection:setup}. We further assume $u^\ast(f)\in H^{s}$ for a given $s \in \mathbb{Z}^+$, we have the following lower bounds.
\paragraph{DRM Lower Bound.} For all estimators $\psi:\left(\mathbb{R}^d\right)^{\otimes n}\times\mathbb{R}^{\otimes n}\rightarrow H^{s}(\Omega)$, we have
\begin{equation}
    \begin{aligned}
    &\inf_{\psi}\sup_{u^{\ast} \in H^{s}(\Omega)} \mathbb{E}\|\psi(\{X_i,f_i\}_{i=1,\cdots,n})-u^\ast(f)\|_{H^1}^2\gtrsim n^{- \frac{2s-2}{d+2s-4}}.
    \end{aligned}
\end{equation}

\paragraph{PINN Lower Bound.} For all estimators $\psi:\left(\mathbb{R}^d\right)^{\otimes n}\times\mathbb{R}^{\otimes n}\rightarrow H^{s}(\Omega)$, we have
\begin{equation}
    \begin{aligned}
    &\inf_{\psi}\sup_{u^{\ast} \in H^{s}(\Omega)} \mathbb{E}\|\psi(\{X_i,f_i\}_{i=1,\cdots,n})-u^\ast(f)\|_{H^2}^2 \gtrsim n^{- \frac{2s-4}{d+2s-4}}.
    \end{aligned}
\end{equation}
\end{theorem}
\begin{proof} We construct the following bump function to construct the multiple hypothesis test used for proving the lower bound.  Consider a simple $C^\infty$ bump function supported on $[0,1]^d$
 
$$
g(x)=\prod_{i=1}^d \xi(x_i), x=(x_1,\cdots,x_d),
$$
where $\xi:\mathbb{R}\rightarrow\mathbb{R}$ is a non-zero funtion in $C^\infty(\mathbb{R})$ with support contained in $[0,1]$ and satisfies $\xi(x)\not=0, \frac{d}{dx}\xi(x)\not=0$. Then $\nabla g(x)\not=0$ and the support of function $g$ is $[0,1]^d$.

Next, we take $m=[n^{\frac{1}{2s-4+d}}]$ and consider a regular gird $x^{(j)}, j\in[m]^d$. According to the Varshamov-Gilbert lemma, there exist $2^{m^d/8}$ $(0,1)$-sequences $\tau^{(1)}, \cdots, \tau^{(2^{m^d/8})}\in\{0,1\}^{m^d}$ such that $\|\tau^{(k)}-\tau^{(k')}\|^2\ge\frac{m^d}{8}$ for all $0<k\not=k'\le 2^{m^d/8}$. Then we construct the multiple hypothesis as 
$$
u_k(x)=\sum_{j\in[m]^d}\tau_{j}^{(k)} \frac{\omega}{m^{s+\frac{d}{2}}}g(m(x-x^{(j)})), k =1,2,\cdots,2^{m^d/8},
$$
where $\omega$ is a constant to be determined later. It is easy to find out that $u_k\in C^{s}$.

Then we reduce solving the PDE to a multiple hypothesis testing problem, which considers all mappings from $n$ sampled data to the constructed hypothesis $\Psi:\left(\mathbb{R}^d\right)^{\otimes n}\times\mathbb{R}^{\otimes n}\rightarrow \mathcal{V}:=\{u_i|i=1,2,\cdots,2^{m^d/8}\}$. Then we apply the local Fano method and check that we can obtain a constant lower bound of $\mathcal{P}(\hat V\not = V)$ for any estimator $\hat V$. From the local Fano method, we know that
\begin{align*}
    I(V;X)\le \frac{1}{|\mathcal{V}|^2}\sum_{z}\sum_{v\not=v'} D_{KL}(P_v||P_v'),
\end{align*}
where $P_k$ denotes the joint distribution of the sampled data $(x,y)$. In specific,  $x$ follows a uniform distribution on $[0,1]^d$ and $y=f(X)+\epsilon$, where $\epsilon$ is independently sampled from a standard Gaussian distribution $N(0,1)$. Then we have
$$KL(P_k||P_{k'})=\mathbb{E}\log(\frac{dP_k}{dP_{k'}})=\|\Delta u_k+Vu_k\|_{L_2}^2\leq\frac{C\omega}{m^{2s -4}}.$$ 
Using Fano inequality, if we select $m\propto[n^{\frac{1}{2s-4+d}}]$ then we have the following lower bound when $\omega$ is taken to be sufficiently large:
$$
\mathcal{P}(\hat V\not = V)\ge 1-\frac{I(V;X)+\log 2}{\log(|\mathcal{V}|)}\ge 1-\frac{\frac{8C\omega}{m^{2s -4}}}{m^d\log 2}\ge 1/2.
$$

At the same time, we can estimate the separation of the hypotheses in two different norms:
\begin{itemize}
    \item Deep Ritz Method:
    $$
\int_{[0,1]^d} \|\nabla u_k-\nabla u_{k'}\|^2dx=\frac{\kappa^2}{m^{2s-2+d}}
\sum_{j\in [m]^d} \|\tau_j^{(k)}-\tau_j^{(k')}\|_1\int_{\mathbb{R}^d}\|\nabla g(x)\|^2dx\gtrsim \frac{1}{m^{2s-2}}.
$$
\item  Physic Informed Neural Network:
$$
\int_{[0,1]^d} \|\Delta u_k-\Delta u_{k'}\|^2dx=\frac{\kappa^2}{m^{2s-4+d}}
\sum_{j\in [m]^d} \|\tau_j^{(k)}-\tau_j^{(k')}\|_1\int_{\mathbb{R}^d}\Delta g(x)^2dx\gtrsim \frac{1}{m^{2s-4}}.
$$
\end{itemize}

Plugging in  $m\propto[n^{\frac{1}{2s-4+d}}]$, we know that with constant probability we have

\begin{equation}
    \begin{aligned}
    \inf_{\psi}\sup_{u^\ast \in H^{s}(\Omega)} \mathbb{E}\|\psi(\{X_i,Y_i\}_{i=1,\cdots,n})-u^\ast(f)\|_{H^1}^2 \gtrsim n^{- \frac{2s-2}{d+2s-4}},
    \end{aligned}
\end{equation}

\begin{equation}
    \begin{aligned}
    \inf_{\psi}\sup_{u^\ast \in H^{s}(\Omega)} \mathbb{E}\|\psi(\{X_i,Y_i\}_{i=1,\cdots,n})-u^\ast(f)\|_{H^2}^2 \gtrsim n^{- \frac{2s-4}{d+2s-4}}.
    \end{aligned}
\end{equation}
\end{proof}
Given that $n^{-\frac{2(\beta-k)}{d+2\beta}}$ is the minimax rate of estimation of the $k$-th derivative of a
$\beta$-smooth density in $L_2$ \cite{liu2012convergence,prakasa1996nonparametric,muller1979optimal}, the lower bound obtained here is the rate of estimating the right hand side function $f$ in terms of the $H^{-1}$ norm.  Given the $H^{-1}$ norm error estimate on $f$, we can achieve estimate of $u$, which provides an alternative way to understand our upper bound. The lower bound is non-standard, for the $2s-2$ in the numerator is different from the $2s-4$ in the denominator.

\section{Upper Bound}
\label{section:upper}
To theoretically understand the empirical success of Physics Informed Neural Networks and the Deep Ritz solver, in this section, we aim to prove that the excess risk $\Delta \mE_{n} := \mE (u_{n}) - \mE(u^\ast)$ of a well-trained neural networks on the PINN/DRM loss function will follow a precise power-law scaling relations with the size of the training dataset. Similar to \cite{xu2020finite,lu2021priori,duan2021convergence,jiao2021convergence,jiao2021error}, we decompose the excess risk into approximation error and generalization error. Different from the concurrent bound \cite{duan2021convergence,jiao2021convergence}, we provided a fast rate $O(1/n)$ by utilizing the strong convexity of the objective function established in Section \ref{subsection:lossfunc} and achieved a faster and near optimal upper bound. 

\subsection{Proof Sketch}

\paragraph{Error Decomposition.} We first decompose the excess risk $\Delta \mE_{n} := \mE (u_{n}) - \mE(u^\ast)$ of a well-trained neural network on the PINN/DRM loss function into approximation error and generalization error, similar to \cite{xu2020finite,lu2021priori,duan2021convergence,jiao2021convergence,jiao2021error}. The regularity results used in the decomposition are proved in Appendix \ref{appendix:regular}. Explicitly, for any $u_{\mF} \in \mF(\Omega)$, we can decompose the excess risk as

\begin{equation}
\label{eq:decomp_appendix}
\begin{aligned}
\Delta \mE^{(n)}(\hat u) &= \mE(\hat u) -\mE(u^{\star})\\
&= \big[\mE(\hat u) - \mE_{n}(\hat u)\big] +\big[ \mE_{n}(\hat u) - \mE_{n}(u_{\mF})\big] 
+ \big[\mE_{n}(u_{\mF}) - \mE(u_{\mF}) \big]+\big[ \mE(u_{\mF}) - \mE(u^{\star})\big]\\
& \leq \underbrace{\big[\mE(\hat u) - \mE_{n}(\hat u)\big]+\big[\mE_{n}(u_{\mF}) - \mE(u_{\mF})\big]}_{\text{Generalization Error}} +\underbrace{\big[ \mE(u_{\mF}) - \mE(u^{\star})\big]}_{\text{Approximation Error}},
\end{aligned}
\end{equation}
where the expectation is on uniformly sampled data, $\mF(\Omega)$ is the space of parametrized estimators we used like truncated Fourier series or sparse neural networks,  $\hat u$ is the minimizer of the empirical loss $\mE_n$ in $\mF(\Omega)$ and $u^\ast$ is the minimizer of the population loss $\mE$ (i.e, ground truth solution). The inequality in the third line follows from the fact that $\hat u$ is the minimizer of the empirical loss $\mE_{n}$ in the space $\mF(\Omega)$, which implies $\mE_{n}(\hat u) \leq \mE_{n}(u_{\mF})$. We call the first term generalization error as it's measuring the difference between $\mE_{n}$ and $\mE$. We call the second term approximation error as it seeks for a parametrized estimator $u_{\mF}$ that approximates the ground truth solution $u^\ast$ well in $\mF(\Omega)$. The upper bounds on generalization and approximation error that we achieved in this paper are listed in Table \ref{table:appendix}. 

Let $n$ denote the number of sampled datapoints. For the generalization error, different from the concurrent upper bound $O(\frac{1}{\sqrt{n}})$  \cite{duan2021convergence,jiao2021convergence}, we provide a faster and near optimal upper bound $O(\frac{1}{n})$ by utilizing the strong convexity of the objective function established in Appendix \ref{appendix:regular}. Via using the Peeling Lemma (Lemma \ref{lem: peeling}, for completeness, we also provide a proof), we show that the generalization error can be bounded by the fixed point of the local Rademacher complexity
$$
\phi(r) = R_n(\{\mathcal{I}(u) \ | \ \|u-u^\ast\|_A^2\le r\}),
$$
where $R_n$ is the Rademacher complexity, $\mathcal{I}(u)= \Delta u+Vu, \|\cdot\|_A=\|\cdot\|_{H^2}$ for PINN and $\mathcal{I}(u)=\|\nabla u\|^2+Vu, \|\cdot\|_A=\|\cdot\|_{H^1}$ for DRM. Once we show that $\phi(r)$ is of magnitude $O(\sqrt{\frac{r}{n}})$, we can achieve the $O(\frac{1}{n})$ convergence rate via solving the fix point equation $\phi(r)=O(\sqrt{\frac{r}{n}})=r\Rightarrow r=O(\frac{1}{n})$. Using the solution of the fixed point equation of the local Rademacher complexity to bound the generalization error is a standard result in empirical process \cite{bartlett2005local,srebro2010smoothness,koltchinskii2011oracle,xu2020towards,farrell2021deep}. The difference is that we used the $H^1$/$H^2$ norm to define the localized set, while the previous papers used the $\ell_2$ distance. The way to obtain the fast rate generalization bound is using the Peeling Lemma.

We present the error decomposition results as a meta theorem, which is shown in Theorem \ref{meta:pinn} for PINN (proof in Appendix \ref{appendix:PINNmeta}), Theorem \ref{meta:drm} for DRM and Theorem \ref{theorem:MDRM} for MDRM (proof in Appendix \ref{appendix:MDRMmeta}), respectively. To make the final rate depend on the data number only, we need bounds of the approximation error in Section \ref{section:approx} and bounds of the local Rademacher complexity in Appendix \ref{appendix:generalization}.

\paragraph{Approximation Error.} The proof of the approximation results of truncated Fourier series is easy and intuitive. For completeness, we provide it in Section \ref{section:approxfourier}. The proof of the approximation results of neural networks follows from the fact that a B-spline approximation can be formulated as a ReLU3 neural network efficiently. Our proof basically follows \cite{duan2021convergence,jiao2021convergence}, while the only difference is the activation function. Our proof is also very similar to \cite{yarotsky2017error,suzuki2018adaptivity}, but the depth of our network is of constant magnitude instead of $O(\frac{1}{\log\epsilon})$ magnitude, where $\eps$ denotes the desired approximation error. Such improvement of depth results from the fact that ReLU3 activations can approximate B-splines more easily than the ReLU activations, which is useful in our generalization analysis. Although the proof of the approximation results of neural networks in the Sobolev space is standard, we still list it in Appendix \ref{section:approxNN}.

\paragraph{Generalization Error.} As we discussed above, the generalization error can be bounded by the fix point of the local Rademacher complexity, \emph{i.e.} the solution of $\phi(r)=r$. Once we have a $O(\sqrt{\frac{r}{n}})$ bound of $\phi(r)$, we can achieve the $O(\frac{1}{n})$ fast rate generalization bound we want. It remains to upper bound the the local Rademacher complexity $\phi(r)$.

For the upper bound on the local Rademacher complexity of truncated Fourier series estimators, our proof technique is similar to that of the kernel estimators, whose Rademacher complexity can be bounded by the trace of the Gram matrix (\emph{i.e.} the effective number of basis). One interesting thing we showed is that the final upper bound of the Rademacher complexity localized by $H^1$ norm is $\sqrt{\frac{\xi^{d-2}r}{n}}$. The term $\xi^{d-2}$ in the numerator is smaller than $\xi^{d}$, which is the exact number of Fourier basis. This improvement results from the $H^1$ norm localization. The detailed proof is given in Lemma \ref{lem:local rademacher of fourier}, Lemma \ref{lem:local rademacher of fourier gradient} and Lemma \ref{lem:local rademacher of fourier laplacian}. 

For the upper bound on the local Rademacher complexity bound for neural network, we follow \cite{schmidt2020nonparametric,suzuki2018adaptivity,farrell2021deep} to use a Dudley integral theorem and a covering number argument. The covering number arguments are shown in Theorem \ref{dnn_covering_num}, Theorem \ref{dnn_grad_covering_num} and Theorem \ref{dnn_laplacian_covering_num}. The final local Rademacher complexity bounds are given in Lemma \ref{lemma:localDRM} and Lemma \ref{lemma:localPINN}. The difference is that the complexity of gradient of ReLU3 activation function makes the covering number depend exponentially on the neural network's depth. However, the  improvement of neural network's depth to constant magnitude mentioned above in the approximation results saves this problem. One drawback of our proof is that the $H^1$ norm localization wouldn't improve the bound for Rademacher complexity and leads to sub-optimal upper bounds. We hypothesize that our bound is tight for sparse neural network and put seeking a right complexity measure of neural network for solving PDEs as a future work.

\begin{table}[h]
\begin{tabular}{|l|l||l|l||l|}
\hline\hline
\multicolumn{1}{|c|}{\textbf{Objective Function}}        & \multicolumn{1}{c||}{\textbf{Estimator}} & \multicolumn{1}{c|}{\textbf{Approximation}} & \multicolumn{1}{c||}{\textbf{Generalization}} & \multicolumn{1}{c|}{\textbf{Complexity Measure}} \\ \hline\hline
\multirow{2}{*}{\textbf{PINNs}} & Neural Network                 &                $N^{-\frac{2s-4}{d}}$                           &          $\frac{N}{n}$                           & $N$: Number of parameters               \\ \cline{2-5} 
                                                & Fourier Seriers                &             $\xi^{-2(s-2)}$                             &                         $\frac{\xi^d}{n}$            & $\xi$:maximum frequency                 \\ \hline\hline
\multirow{2}{*}{\textbf{DRM}}              & Neural Network                 &         $N^{-\frac{2s-2}{d}}$                                 &               $\frac{N}{n}$                      & $N$: Number of parameters               \\ \cline{2-5} 
                                                & Fourier Seriers                &                 $\xi^{-2(s-1)}$                         &                            $\frac{\xi^{d}}{n}$         & $\xi$:maximum frequency                 \\ \hline\hline
\multirow{2}{*}{\textbf{MDRM}}     & Neural Network                 &                  $N^{-\frac{2s-2}{d}}$                        &    $\frac{N}{n}$                                 & $N$: Number of parameters               \\ \cline{2-5} 
                                                & Fourier Seriers                &   $\xi^{-2(s-1)}$                                       &            $\frac{\xi^{d-2}}{n}$                         & $\xi$:maximum frequency                 \\ \hline
\end{tabular}
\caption{Approximation and generalization results we achieved in this paper.}
\label{table:appendix}
\end{table}

\subsection{Bounding the Approximation Error}
\label{section:approx}
\subsubsection{Approximation using Truncated Fourier Basis}
\label{section:approxfourier}
\begin{lemma} \label{lem: approximation Fourier}
Given $\alpha > 0$ and a fixed integer $\xi \in \mathbb{Z}^{+}$. For any  function $f \in H^{\alpha}(\Omega)$ , we let $f_{\xi} = \sum_{\|z\|_{\infty} \leq \xi}f_{z}\phi_{z}$ be the best approximation of $f$ in the space $F_{\xi}(\Omega)$. Then for any $0<\beta\leq\alpha$, we have the following inequality:
\begin{align*}
\|f-f_{\xi}\|_{H^{\beta}(\Omega)}^2 \leq     \xi^{-2(\alpha-\beta)} \|f\|_{H^\alpha}^2.
\end{align*}
\end{lemma}
\begin{proof} For $f \in H^{\alpha}(\Omega)$, we know the Fourier coefficient satisfies
\begin{align*}
\sum_{\|z\|_{\infty} \geq \xi}|f_z|^{2}\|z\|^{2\alpha}  \lesssim \|f\|_{H^\alpha}^2.
\end{align*}
We directly construct $f_\xi = \sum_{\|z\|_\infty\le \xi}f_z\phi_z$ to be the truncated Fourier series of the function $f$, then we have
\begin{align*}
    \|f-f_\xi\|_{H^{\beta}(\Omega)}^2 \lesssim \sum_{\|z\|_\infty\ge \xi}|f_z|^{2}\|z\|^{2\beta} \leq \xi^{-2(\alpha-\beta)}\sum_{\|z\|_\infty\ge \xi}|f_z|^{2}\|z\|^{2\alpha}\le\xi^{-2(\alpha-\beta)} \|f\|_{H^\alpha}^2.
\end{align*}
\end{proof}

\subsubsection{Approximation using Neural Network}
\label{section:approxNN}

In this section, we aim to provide approximation bound for deep neural network. Our proof of the approximation upper bound is based on the observation that the B-spline approximation\cite{de1978practical,schumaker2007spline} can be formulated as a ReLU3 neural network efficiently\cite{suzuki2018adaptivity,guhring2020error,duan2021convergence,jiao2021convergence}. Although the proof of the approximation of the neural network to the Sobolev spaces is a standard approach, we still demonstrate the proof sketch here.

\begin{definition}{(Univariate and Multivariate B-splines)} \label{def: b_splines}
Fix an arbitrary integer $l \in \mathbb{Z}^{+}$. Consider a corresponding uniform partition $\pi_{l}$ of $[0,1]$:
\begin{align*}
\pi_{l}: 0 = t_{0}^{(l)} < t_{1}^{(l)} < \cdots < t_{l-1}^{(l)} < t_{l}^{(l)} = 1,    
\end{align*}
where $t_{i}^{(l)} = \frac{i}{l} \ (\forall \ 0 \leq i \leq l)$. Now for any $k \in \mathbb{Z}^{+}$, we can define an extended partition $\pi_{l,k}$ as:
\begin{align*}
\pi_{l,k}: t_{-k+1}^{(l)} = \cdots t_{-1}^{(l)}= 0 = t_{0}^{(l)} < t_{1}^{(l)} < \cdots < t_{l-1}^{(l)} < t_{l}^{(l)} = 1 = t_{l+1}^{(l)} = \cdots = t_{l+k-1}^{(l)}     
\end{align*}
Based on the extended partition $\pi_{l,k}$, the univariate B-splines of order $k$ with respect to partition $\pi_{l}$ are defined by:
\begin{equation}
\label{univ_b_spline}
N_{l,i}^{(k)}(x) := (-1)^{k}(t_{i+k}^{(l)} - t_{i}^{(l)}) \cdot \Big[t_{i}^{(l)}, \cdots, t_{i+k}^{(l)}\Big]\max\{(x-t),0\}^{k-1}, \ x \in [0,1], \ i \in I_{l,k}
\end{equation}
where $I_{l,k} = \{-k+1,-k+2,\cdots,l-1\}$ and $\Big[t_{i}^{(l)}, \cdots, t_{i+k}^{(l)}]$ denotes the divided difference operator.\\
Equivalently, for any $x \in [0,1]$, we can rewrite the univariate B-splines $N_{l,i}^{(k)}(x)$ in an explicit form:
\begin{equation}
\label{spline_explicit}
N_{l,i}^{(k)}(x) =
\begin{cases}
\frac{l^{k-1}}{(k-1)!}\sum_{j=0}^{k}(-1)^{j}{k \choose j}\max\Big\{x-\frac{i+j}{l},0\Big\}^{k-1}, \ (0 \leq i \leq l-k+1)\\
\sum_{j=0}^{k-1}a_{ij} \max\Big\{x-\frac{j}{l},0\Big\}^{k-1} + \sum_{n=1}^{k-2}b_{in}x^{n} + b_{i0}, \ (-k+1 \leq i \leq 0)\\
\sum_{j=l-k+1}^{l}c_{ij}\max\Big\{x-\frac{j}{l},0\Big\}^{k-1}, \ (l-k+1 \leq i \leq l-1)
\end{cases}
\end{equation}
where $\{a_{ij} \ | \ -k+1 \leq i \leq 0, \ 0 \leq j \leq k-1\}$, $\{b_{in} \ | \ -k+1 \leq i \leq 0, \ 1 \leq n \leq k-2\}$ and $\{c_{ij} \ | \ l-k+1 \leq i \leq l-1, \ l-k+1 \leq j \leq l-1\}$ are some fixed constants.\\
For any index vector $\boldsymbol{i} = (i_1,i_2,\cdots,i_d) \in I_{l,k}^{d}$, we can define a corresponding multivariate B-spline as a product of univariate B-splines:
\begin{equation}
\label{mulv_b_spline}
N_{l,\boldsymbol{i}}^{(k)}(\boldsymbol{x}) := \Pi_{j=1}^{d}N_{l,i_{j}}^{(k)}(x_{j}).
\end{equation}
\end{definition}

\begin{definition}{(Interpolation Operator\cite{schumaker2007spline})}
Take some domain $\Omega \subset [0,1]^{d}$ and two arbitrary integers $k,l \in \mathbb{Z}^{+}$. Consider the extended partition $\pi_{l,k}$ and the corresponding set of multivariate B-splines $\{N_{l,\boldsymbol{i}}^{(k)}(x)\}_{\boldsymbol{i} \in I_{l,k}^{d}}$ defined in Definition \ref{def: b_splines}. For any $\boldsymbol{i} \in I_{l,k}^{d}$, we define the domain $\Omega_{\boldsymbol{i}} := \{\boldsymbol{x} \in \Omega: x_{j} \in [t_{i_{j}}, t_{i_{j}+k}], \ 1 \leq j \leq d\}$. There exists a set of linear functionals $\{\lambda_{\boldsymbol{i}}\}_{\boldsymbol{i} \in I_{k,l}^{d}}$, where $\lambda_{\boldsymbol{i}}: L^1(\Omega) \rightarrow \mathbb{R} \ (\forall \ \boldsymbol{i} \in I_{k,l}^{d})$, such that for any $\boldsymbol{i} \in I_{k,l}^{d}$ and $p \in [1,\infty]$, we have: 
\begin{equation}
\label{linear_func_constriant}
\lambda_{\boldsymbol{i}}(N_{l,\boldsymbol{j}}^{(k)}) = \delta_{\boldsymbol{i},\boldsymbol{j}} \text{ and } |\lambda_{\boldsymbol{i}}(f)| \leq 9^{d(k-1)}(2k+1)^{d}\Big(\frac{k}{l}\Big)^{-\frac{d}{p}}\|f\|_{L^p(\Omega_{\boldsymbol{i}})}, \ \forall \ f \in L^p(\Omega).
\end{equation}
The corresponding interpolation operator $Q_{k,l}$ is defined as:
\begin{align*}
Q_{k,l}f := \sum_{\boldsymbol{i} \in I_{k,l}^{d}}\lambda_{i}(f)N_{l,\boldsymbol{i}}^{(k)}, \ \forall \ f \in L^1(\Omega).    
\end{align*}
\end{definition}

\begin{theorem}{[\cite{schumaker2007spline}]}
\label{approx_intep_op}
Fix $f \in W^{s}(\Omega)$ with $\Omega \subseteq [0,1]^{d}, s \in \mathbb{Z}^{+}$ and $p \in [1,\infty)$. Then for any $k,l,r \in \mathbb{Z}^{+}$ with $k 
\geq s$ and $0 \leq r \leq s$, we have that there exists some constant $C = C(k,s,r,p,d)$, such that:
\begin{align*}
\|f-Q_{k,l}f\|_{H^{r}(\Omega)} \leq C\Big(\frac{1}{l}\Big)^{s-r}\|f\|_{H^{s}(\Omega)}.    
\end{align*}
\end{theorem}

\begin{theorem}{(Approximation result of Deep Neural Network)}
\label{thm: approximation NN}
Fix some dimension $d \in \mathbb{Z}^{+}$, some domain $\Omega \subseteq [0,1]^{d}$. We pick some $l= N^{\frac{1}{d}} \geq 2$, for any $s,r \in \mathbb{Z}^{+}$ with $0 \leq r \leq s$ and any function $u^{\ast} \in H^{s}(\Omega)$, there exists some sparse Deep Neural Network $u_{\text{DNN}} \in \Phi(L,W,S,B)$ with $L = O(1), W=O(N), S=O(N), B=O(N)$, such that:
\begin{equation}
\|u_{DNN}-u^{\ast}\|_{H^{r}(\Omega)} \lesssim N^{-\frac{s-r}{d}}\|u^{\ast}\|_{H^{s}(\Omega)}.  
\end{equation}
\begin{proof}
We firstly show that the given function $u^{\ast}$ can be approximated well by some linear combination of multivariate splines, which is denoted by $u_{\text{sp}}$. Note that $N$ is assumed to be sufficiently large. Hence, we may pick $l=\lceil N^{\frac{1}{d}} \rceil =\Theta(N^{\frac{1}{d}}) \in \mathbb{Z}^{+}$ to be the partition size of the B-splines. Moreover, by picking $k=4$ and $p=2$ in Theorem \ref{approx_intep_op}, we have that the linear combination $u_{\text{sp}} := Q_{4,l}u^\ast = \sum_{\boldsymbol{i} \in I_{4,l}^{d}}\lambda_{i}(u^\ast)N_{l,\boldsymbol{i}}^{(4)}$ satisfies:
\begin{align*}
\|u^{\ast} - u_{\text{sp}}\|_{H^{r}(\Omega)}=\|u^\ast-Q_{4,l}u^\ast\|_{H^{r}(\Omega)} \leq C\Big(\frac{1}{l}\Big)^{s-r}\|u^\ast\|_{H^{s}(\Omega)} = CN^{-\frac{s-r}{d}}\|u^\ast\|_{H^s(\Omega)}.       
\end{align*}
We will then show that the linear combination $u_{\text{sp}} = \sum_{\boldsymbol{i} \in I_{4,l}^{d}}\lambda_{i}(f)N_{l,\boldsymbol{i}}^{(4)}$ can be implemented by some Deep Neural Network $u_{\text{DNN}} \in \Phi(L,W,S,B)$ with $L=O(1),W=O(N),S=O(N)$ and $B=O(\log N)$. Firstly, note that for $x \geq 0$, both $x$ and $x^2$ can be expressed in terms of the ReLU3 activation function $\eta_{3}$ with no error:
\begin{align*}
x &= -\frac{1}{12}[\eta_{3}(x+3) - 5\eta_{3}(x+2) +7\eta_{3}(x+1) - 3\eta_{3}(x) + 6]\\
x^2 &= -\frac{1}{6}[\eta_{3}(x+2) - 4\eta_{3}(x+1) +3\eta_{3}(x) - 4]
\end{align*}
Applying the explicit formula listed in equation \ref{spline_explicit} implies that for any $-3 \leq i \leq l-1$, the univariate B-spline function $N_{l,i}^{(4)}(x) \ (x \in [0,1])$ can be implemented by some ReLU3 Deep Neural Network $v_{\text{DNN}}$ with both scalar input and scalar output. We have that for $v_{\text{DNN}}$, the depth $L_{v}$ is $2$ and the maximum width $W_{v}$ is upper bounded by $11$.\\
Secondly, for any $x,y \geq 0$, we have that the product operation $x \cdot y$ can be expressed in terms of the ReLU3 activation function $\eta_{3}$ with no error:
\begin{align*}
x \cdot y &= \frac{1}{2}[(x+y)^2 - x^2 - y^2]\\
&= -\frac{1}{12}\Big[\eta_{3}(x+y+2) -4\eta_{3}(x+y+1) + 3\eta_3(x+y)\\
&-\eta_3(x+2) + 4\eta_3(x+1) -3\eta_3(x)-\eta_3(y+2) + 4\eta_3(y+1) -3\eta_3(y)+4 \Big]
\end{align*}
In \cite{schumaker2007spline}, it has been proved that the B-splines are always non-negative, i.e $N_{l,i}^{(4)}(x) \geq 0, \ \forall \ x \in [0,1]$. Therefore, by multiplying the non-negative univariate B-splines, we can implement any multivariate B-spline $N_{l,\boldsymbol{i}}^{(4)} = \Pi_{j=1}^{d}N_{l,i_j}^{(4)}(x_j)$ with some ReLU3 Deep Neural Network $p_{\text{DNN}}$. We have that for $p_{\text{DNN}}$, the depth $L_{p} = \lceil \log_{2}d \rceil + 2$ and the maximum width $W_{p} = \max\{11d,\frac{9}{2}d\}$.\\
Hence, we can further claim that $u^{\ast} = \sum_{\boldsymbol{i} \in I_{4,l}^{d}}\lambda_{i}(u^\ast)N_{l,\boldsymbol{i}}^{(4)}$, which is a linear combination of the multivariate B-splines $N_{l,\boldsymbol{i}}^{(4)}$, can be implemented by some ReLU3 Deep Neural Network $u_{\text{DNN}}$. It remains to check that $u_{\text{DNN}} \in \Phi(L,W,S,B)$ with $L = O(1), W=O(N), S=O(N)$ and $B=O(N)$. Note that
we can ensure that the hidden layers of $u_{\text{DNN}}$ are of the same dimension $W$ by adding inactive neurons.\\
For the depth $L$ of $u_{\text{DNN}}$, we have that $L$ is equal to $L_{p} + 1$, where $L_{p}$ denotes the depth of the ReLU3 Deep Neural Network $p_{\text{DNN}}$. Thus, we have $L = L_{p} + 1 = \lceil \log_2 d \rceil + 3$, which implies that $L = O(1)$.\\
For the width $W$ of $u_{\text{DNN}}$, we have that $W \leq |I_{k,l}^{d}|W_{p}$, where $W_{p}$ denotes the width of the ReLU3 Deep Neural Network $p_{\text{DNN}}$. This implies:
\begin{align*}
W \leq |I_{k,l}^{d}| \times 11d = 11d(l+k)^d = 11d(l+4)^d = O(l^d) \Rightarrow W = O(N)    
\end{align*}
For the sparsity constraint $S$ of $u_{\text{DNN}}$, starting from the third layer, the number of activated neurons is half of the number of activated neurons at previous layer.  This yields the following upper bound on $S$: 
\begin{align*}
S \leq 2(W + W + \sum_{j=0}^{L-2}\frac{W}{2^j}) \leq 8W \Rightarrow S = O(W) = O(N)
\end{align*}
For the norm constraint $B$ of $u_{\text{DNN}}$, we have the following upper bound on $B$ from equation \ref{spline_explicit} and equation \ref{linear_func_constriant}:
\begin{align*}
B = O(\max\{l^{k-1}, \sup_{\boldsymbol{i} \in I_{k,l}^d}\lambda_{\boldsymbol{i}}(u^\ast)\}) = O(\max\{l^{3},l^{d}\}) = O(N) 
\end{align*}
Now we have shown that parameters $L,W,S,B$ of the Deep Neural Network $u_{\text{DNN}}$ are of the desired magnitude, which completes our proof.
\end{proof}

\end{theorem}

\subsection{Bounding the Local Rademacher Complexity}

\subsubsection{Local Rademacher Complexity of Truncated Fourier Basis}
In this subsection, we aim to bound the local Rademacher complexity of the Truncated Fourier Basis estimator. The proof is standard and we put the proof in the appendix.
\begin{lemma}({Local Rademacher Complexity of Localized Truncated Fourier Series)} \label{lem:local rademacher of fourier}
For a fixed $\xi \in \mathbb{Z}^{+}$, we consider a localized class of functions $\mF_{\rho,\xi}(\Omega) = \Big\{f \in F_{\xi}(\Omega) \ \Big| \ \|f\|_{H^1(\Omega)}^2 \leq \rho \Big\}$, where $\rho >0$ is fixed. Then we have the following upper bound on the local Rademacher complexity:
\begin{equation}
R_{n}(\mF_{\rho,\xi}(\Omega)) = \bE_{X}\left[\bE_{\sigma}\Big[\sup_{f \in \mF_{\rho,\xi}(\Omega)}\frac{1}{n}\sum_{i=1}^{n}\sigma_{i}f(X_{i}) \ \Big| \ X_1, \cdots, X_n \Big]\right] \lesssim \sqrt{\frac{\rho}{n}}\xi^{\frac{d-2}{2}}.
\end{equation}
\end{lemma}

\begin{lemma}{(Local Rademacher Complexity of Localized Truncated Fourier Series' Gradient)}\label{lem:local rademacher of fourier gradient}
For a fixed $\xi \in \mathbb{Z}^{+}$, we consider a localized class of functions $\mG_{\rho,\xi}(\Omega) = \{\|\nabla f\| \ | \ f \in F_{\rho,\xi}(\Omega)\}$, where $\rho >0$ is fixed. Then for any sample $\{X_{i}\}_{i=1}^{n} \subset \Omega$, we have the following upper bound on the local Rademacher complexity:
\begin{equation}
R_{n}(\mG_{\rho,\xi}(\Omega)) = \bE_{X}\left[\bE_{\sigma}\Big[\sup_{f \in \mF_{\rho,\xi}(\Omega)}\frac{1}{n}\sum_{i=1}^{n}\sigma_{i}\|\nabla f(X_{i})\| \ \Big| \ X_1, \cdots, X_n \Big]\right] \lesssim \sqrt{\frac{\rho}{n}}\xi^{\frac{d}{2}}.
\end{equation}
\end{lemma}

\begin{lemma}{(Local Rademacher Complexity of Localized Truncated Fourier Series' Laplacian)}\label{lem:local rademacher of fourier laplacian}
For a fixed $\xi \in \mathbb{Z}^{+}$, we consider a localized class of functions $\mJ_{\rho,\xi}(\Omega) := \Big\{f \in F_{\xi}(\Omega) \ \Big| \ \|f\|_{H^2(\Omega)}^2 \leq \rho \Big\}$, where $\rho >0$ is fixed. Correspondingly, we define a localized class of Laplacians $\mK_{\rho,\xi}(\Omega) := \{\Delta f \ | \ f \in J_{\rho,\xi}(\Omega)\}$. Then for any sample $\{X_{i}\}_{i=1}^{n} \subset \Omega$, we have the following upper bound on the local Rademacher complexity:
\begin{equation}
R_{n}(\mK_{\rho,\xi}(\Omega)) = \bE_{X}\left[\bE_{\sigma}\Big[\sup_{f \in \mF_{\rho,\xi}(\Omega)}\frac{1}{n}\sum_{i=1}^{n}\sigma_{i}\Delta f(X_{i}) \ \Big| \ X_1, \cdots, X_n \Big]\right] \lesssim \sqrt{\frac{\rho}{n}}\xi^{\frac{d}{2}}.
\end{equation}
\end{lemma}

\subsubsection{Local Rademacher Complexity of Deep Neural Networks} In this subsection, we aim to bound the local Rademacher complexity of the Neural Network estimator. Informally, we showed that the local Rademacher complexity is at the scale of $\sqrt{\frac{N}{n}}$, where $N$ is the number of neuron of a neural network. For simplicity, we put the proof in the appendix.

\begin{lemma}[Local Rademacher Complexity Bound for Deep Ritz Method]
\label{lemma:localDRM}
Consider a Deep Neural Network space $\mF(\Omega) = \Phi(L,W,S,B)$ with $L=O(1), W=O(N),S=O(N)$ and $B=O(N)$, where $N \in \mathbb{Z}^{+}$ is fixed to be sufficiently large. Moreover, assume that the gradients and function value of $\mF(\Omega), V$ and $f$ are uniformly bounded 
\begin{equation}\label{assp:boundedness_drm_local_rad}
    \max \Big\{\sup_{u \in \mF(\Omega)}\|u\|_{L^{\infty}(\Omega)}, \sup_{u \in \mF(\Omega)}\|\nabla u\|_{L^{\infty}(\Omega)}, \|u^{\ast}\|_{L^{\infty}(\Omega)}, \|\nabla u^{\ast}\|_{L^{\infty}(\Omega)}, V_{max}, \|f\|_{L^{\infty}(\Omega)} \Big \} \leq C.
\end{equation} 
For any $\rho >0$, we consider a localized set $L_{\rho}$ defined by:
$$
\mL_\rho(\Omega):=\{u:u \in \mF(\Omega),\|u-u^\ast\|_{H^1}^2\le \rho\}.
$$
Then for any $\rho \gtrsim n^{-2}$, the Rademacher complexity of a localized function space $\mS_{\rho}(\Omega) := \Big\{h :=|\Omega| \cdot \left[ \frac{1}{2}\Big(\|\nabla u\|^2-\|\nabla u^{\ast}\|^2\Big) + \frac{1}{2}V(|u|^2-|u^{\ast}|^2)-f(u-u^{\ast})\right] \ \ \Big | \ u \in L_{\rho}(\Omega)\Big \}$ can be upper bounded by a sub-root function 
$$
\phi(\rho):= O\left(\sqrt{\frac{S3^L\rho}{n}\log\left(BWn\right)}\right).
$$
\emph{i.e.} we have
\begin{equation}
\phi(4\rho) \leq 2\phi(\rho) \text{ and } R_{n}(\mS_{\rho}(\Omega)) \leq \phi(\rho). \ 
\end{equation}
holds for all $\rho \gtrsim n^{-2}$.
\end{lemma}

\begin{lemma}[Local Rademacher Complexity Bound for Physics Informed Neural Network]
\label{lemma:localPINN}
Consider a Deep Neural Network space $\mF(\Omega) = \Phi(L,W,S,B)$ with $L=O(1), W=O(N),S=O(N)$ and $B=O(N)$, where $N \in \mathbb{Z}^{+}$ is fixed to be sufficiently large. Moreover, assume that the gradients and function value of $\mF(\Omega), V$ and $f$ are uniformly bounded 
\begin{equation}\label{assp:boundedness_pinn_local_rad}
    \max \Big\{\sup_{u \in \mF(\Omega)}\|u\|_{L^{\infty}(\Omega)}, \sup_{u \in \mF(\Omega)}\|\Delta u\|_{L^{\infty}(\Omega)}, \|u^{\ast}\|_{L^{\infty}(\Omega)}, \|\Delta u^{\ast}\|_{L^{\infty}(\Omega)}, V_{max}, \|f\|_{L^{\infty}(\Omega)} \Big \} \leq C.
\end{equation} 
For any $\rho >0$, we consider a localized set $M_{\rho}$ defined by:
$$
\mM_\rho(\Omega):=\{u:u \in \mF(\Omega),\|u-u^\ast\|_{H^2}^2\le \rho\}.
$$
Then for any $\rho \gtrsim n^{-2}$, the Rademacher complexity of a localized function space $\mT_{\rho}(\Omega) := \Big\{h :=|\Omega| \cdot \left[ (\Delta u -Vu +f)^2-(\Delta u^\ast -Vu^\ast +f)^2\right] \ \ \Big | \ u \in M_{\rho}(\Omega)\Big \}$ can be upper bounded by a sub-root function 
$$
\phi(\rho):= O\left(\sqrt{\frac{S3^L\rho}{n}\log\left(BWn\right)}\right).
$$
\emph{i.e.} we have
\begin{equation}
\phi(4\rho) \leq 2\phi(\rho) \text{ and } R_{n}(\mT_{\rho}(\Omega)) \leq \phi(\rho). \ 
\end{equation}
holds for all $\rho \gtrsim n^{-2}$.
\end{lemma}

\subsection{Final Upper Bound}

\paragraph{Deep Ritz Methods.}

In this subsection, we provide the proof of upper bounds for DRM. We first provide a meta-theorem to illustrate the approximation and generalization decomposition with a $O(1/n)$ fast rate generalization bound\cite{bartlett2005local,xu2020finite}. Then we use truncated fourier basis estimator and neural network estimator as example to obtain the final rate.

\begin{theorem}[Meta-theorem for Upper Bounds of Deep Ritz Methods]
\label{meta:drm}
Let $u^\ast \in H^{s}(\Omega)$ denote the true solution to the PDE model with Dirichlet boundary condition:
\begin{equation}
\begin{aligned}
   -\Delta u + V u & = f \text{ on } \Omega,\\ 
     u & =  0 \text{ on } \partial \Omega, 
\end{aligned}
\end{equation}
where $f\in L^2(\Omega)$ and $V\in L^\infty(\Omega)$ with $0< V_{\min} \leq V(x) \leq V_{\max}>0$.
For a fixed function space $\mF(\Omega)$, consider the empirical loss induced by the Deep Ritz Method:
\begin{equation}
\begin{aligned}
      \mE_{n}(u)  & = \frac{1}{n} \sum_{j=1}^n\Big[ |\Omega| \cdot  \Big(\frac{1}{2} |\nabla u(X_j)|^2 + \frac{1}{2} V(X_j) |u(X_j)|^2- f(X_j)u(X_j) \Big)\Big],
\end{aligned}
\end{equation}
where $\{X_{j}\}_{j=1}^{n}$ are datapoints uniformly sampled from the domain $\Omega$. Then the Deep Ritz estimator associated with function space $\mF(\Omega)$ is defined as the minimizer of $\mE_{n}(u)$ over the function space $\mF(\Omega)$:
\begin{align*}
\hat u_{\text{DRM}} = \min_{u \in \mF(\Omega)}\mE_{n}(u)    
\end{align*}
Moreover, we assume that there exists some constant $C > 0$ such that all function $u$ in the function space $\mF(\Omega)$, the real solution $u^{\ast}$ and $f,V$ satisfy the following two conditions.
\begin{itemize}
    \item The gradients and function value are uniformly bounded 
    \begin{equation}\label{assp:boundedness_drm}
    \max \Big\{\sup_{u \in \mF(\Omega)}\|u\|_{L^{\infty}(\Omega)}, \sup_{u \in \mF(\Omega)}\|\nabla u\|_{L^{\infty}(\Omega)}, \|u^{\ast}\|_{L^{\infty}(\Omega)}, \|\nabla u^{\ast}\|_{L^{\infty}(\Omega)}, V_{max}, \|f\|_{L^{\infty}(\Omega)} \Big \} \leq C.
\end{equation} 
\item All the functions in  the function space $\mF(\Omega)$ satisfies the boundary condition
\begin{equation*}
    u=0 \text{ on } \partial \Omega.
\end{equation*}
\end{itemize}

At the the same time, for any $\rho > 0$, we assume the Rademacher complexity of a localized function space $\mS_{\rho}(\Omega) := \Big\{h :=|\Omega| \cdot \left[ \frac{1}{2}\Big(|\nabla u|^2-|\nabla u^{\ast}|^2\Big) + \frac{1}{2}V(|u|^2-|u^{\ast}|^2)-f(u-u^{\ast})\right] \ \ \Big | \ \|u-u^\ast\|_{H^1}^2\leq \rho \Big \}$ can be upper bounded by a sub-root function $\phi = \phi(\rho): [0, \infty) \rightarrow [0,\infty)$, \emph{i.e.}
\begin{equation}\label{peelingcond_drm}
\phi(4\rho) \leq 2\phi(\rho) \text{ and } R_{n}(\mS_{\rho}(\Omega)) \leq \phi(\rho) \ (\forall \ \rho > 0).
\end{equation}
For all constant $t>0$. We denote $r^*$ to be the solution of the fix point equation of local Rademacher complexity $r = \phi(r)$. There exists a constant $C_p$ such that for probability $1-C_p\exp(-t)$, we have the following upper bound for the Deep Ritz Estimator
$$
\|\hat u_{\text{DRM}} -u^\ast\|_{H^1}^2 \lesssim \inf_{u_{\mF} \in \mF(\Omega)}\Big(\mE(u_{\mF}) - \mE(u^{\star})\Big)+ \max\Big\{r^*,\frac{t}{n}\Big\}.
$$

\end{theorem}
\begin{proof}

To upper bound the excess risk $\Delta \mE^{(n)}:=\mE(\hat u_{\text{DRM}})-\mE(u^\ast) $, following\cite{xu2020finite,lu2021priori,duan2021convergence}, we decompose the excess risk into approximation error and generalization error with probability $1-e^{-t}$:
\begin{equation}
\label{eq:decomp_drm}
\begin{aligned}
\Delta \mE^{(n)}(\hat u_{\text{DRM}}) = \mE(\hat u_{\text{DRM}}) -\mE(u^{\star}) &= \big[\mE(\hat u_{\text{DRM}}) - \mE_{n}(\hat u_{\text{DRM}})\big] +\big[ \mE_{n}(\hat u_{\text{DRM}}) - \mE_{n}(u_{\mF})\big] \\
&+ \big[\mE_{n}(u_{\mF}) - \mE(u_{\mF}) \big]+\big[ \mE(u_{\mF}) - \mE(u^{\star})\big]\\
& \leq \big[\mE(\hat u_{\text{DRM}}) - \mE_{n}(\hat u_{\text{DRM}})\big]+\big[\mE_{n}(u_{\mF}) - \mE(u_{\mF})\big] +\big[ \mE(u_{\mF}) - \mE(u^{\star})\big]\\
&\leq \big[\mE(\hat u_{\text{DRM}})-\mE(u^\ast)+\mE_{n}(u^{\ast})- \mE_{n}(\hat u_{\text{DRM}})\big]\\
& +\frac{3}{2}\big[ \mE(u_{\mF}) - \mE(u^{\star})\big]+\frac{t}{2n},
\end{aligned}
\end{equation}
where the expectation is on all sampled data. The inequality of the third line is because the  $u$ is the minimizer of the empirical loss $\mE_{n}$ in the solution set $\mF(\Omega)$, so we have $\mE_{n}(u) \leq \mE_{n}(u_{\mF})$. The last inequality is based on the Bernstein inequality. The variance of $h = |\Omega| \cdot \left[ \frac{1}{2}\Big(|\nabla u|^2-|\nabla u^{\ast}|^2\Big) + \frac{1}{2}V(|u|^2-|u^{\ast}|^2)-f(u-u^{\ast})\right]$ can be bounded by $\big[\mE(u_{\mF}) - \mE(u^{\star})\big]$ due to the strong convexity of the variation objective (\ref{cond: talagrand strong convex_drm}). According to the Bernstein inequality, we know with probability $1-e^{-t}$ we have
\begin{align*}
\mE_{n}(u_{\mF})-\mE_{n}(u^\ast)-\mE(u_{\mF})+\mE(u^\ast) \le \sqrt{\frac{t\big[\mE(u_{\mF}) - \mE(u^{\star})\big]}{n}} \leq \frac{1}{2}\big[\mE(u_{\mF}) - \mE(u^{\star})\big] +\frac{t}{2n}.
\end{align*}
Note that \ref{eq:decomp_drm} holds for all function lies in the function space $\mF$. Thus, we can take $u_{\mF}:=\arg\min_{u_{0} \in \mF(\Omega)}\Big(\mE(u_{0}) - \mE(u^{\star})\Big)$ and finally get
\begin{align*}
\Delta \mE^{(n)} &\leq \underbrace{\mE(\hat u_{\text{DRM}})-\mE(u^\ast)+\mE_{n}(u^{\ast})- \mE_{n}(u)}_{\Delta \mE_{\text{gen}}} + \frac{3}{2}\underbrace{\inf_{u_{\mF} \in \mF(\Omega)}\Big(\mE(u_{\mF}) - \mE(u^{\star})\Big)}_{\Delta \mE_{\text{app}} } + \frac{t}{2n}.
\end{align*}
This inequality decompose the excess risk to the generalization error $\Delta \mE_{\text{gen}} := \mE(\hat u_{\text{DRM}})-\mE(u^\ast)+\mE_{n}(u^{\ast})- \mE_{n}(\hat u_{\text{DRM}})$ and the approximation error $\Delta \mE_{\text{app}} = \inf_{u_{\mF} \in \mF(\Omega)}\Big(\mE(u_{\mF}) - \mE(u^{\star})\Big)$. \\
We'll focus on providing fast rate upper bounds of the generalization error for the two estimators using the localization technique\cite{bartlett2005local,xu2020finite}. To achieve the fast generalization bound, we focus on the following normalized empirical process

\begin{align*}
\tilde{\mS}_{r}(\Omega) := \big\{\tilde{h}(x) := \frac{\mathbb{E}[h]-h(x)}{\bE[h] + r} \ | \ h \in \mS(\Omega)\big\} \ (r > 0).
\end{align*}

First, we try to bound the expectation of the normalized empirical process. Applying the Symmetrization Lemma \ref{lem:radcomp}, we can first bound the expectation as

\begin{align*}
\sup_{\tilde{h} \in \tilde{S}_{r}(\Omega)}\mathbb{E}_{x'}\left[\frac{1}{n}\sum_{i=1}^{n}\tilde{h}(x_i')\right] \leq\mathbb{E}_{x'}\left[\sup_{h \in S(\Omega)}\Big|\frac{1}{n}\sum_{i=1}^{n}\frac{h(x_i')-\bE[h] }{\bE[h] + r}\Big|\right]
\leq 2R_{n}(\hat{\mS}_{r}(\Omega)).
\end{align*}
where the function class $\hat{\mS}_{r}(\Omega)$ is defined as:
\begin{align*}
\hat{\mS}_{r}(\Omega) := \big\{\hat{h}(x) := \frac{h(x)}{\bE[h] + r} \ | \ h \in \mS(\Omega)\big\},    
\end{align*}

where $ \mS(\Omega)=\Big\{h :=|\Omega| \cdot \left[ \frac{1}{2}\Big(|\nabla u|^2-|\nabla u^{\ast}|^2\Big) + \frac{1}{2}V(|u|^2-|u^{\ast}|^2)-f(u-u^{\ast})\right]\Big \}.$ Then Applying the Peeling Lemma to any function $h \in \mS(\Omega)$ helps us upper bound the local Rademacher complexity $R_{n}(\hat{\mS}_{r}(\Omega))$ with the function $\phi$ defined in equation \ref{peelingcond_drm}:
$$
R_{n}(\hat{\mS}_{r}(\Omega)) = \bE_{\sigma}\left[\bE_{x}\Big[\sup_{h \in \mS(\Omega)}\frac{\frac{1}{n}\sum_{i=1}^{n}\sigma_{i}h(x_i)}{\bE[h] + r}\Big]\right] \leq \frac{4\phi(r)}{r}. 
$$
Combining all inequalities derived above yields:
\begin{equation}
\sup_{\tilde{h} \in \tilde{S}_{r}(\Omega)}\mathbb{E}_{x'}\left[\frac{1}{n}\sum_{i=1}^{n}\tilde{h}(x_i')\right]\leq 2R_{n}(\hat{\mS}_{r}(\Omega)) \leq  \frac{8\phi(r)}{r} \ (r > 0).
\end{equation}

Secondly we'll apply the Talagrand concentration inequality, which requires us to verify the condition needed. We will first check that the expectation value $\bE[h]$ is always non-negative for any $h \in \mS(\Omega)$:
\begin{align*}
\bE[h]&= \frac{1}{|\Omega|}\int_{\Omega}|\Omega| \cdot (\frac{1}{2} |\nabla u(x)|^2 + \frac{1}{2} V(x) |u(x)|^2- f(x)u(x))dx\\
&-\frac{1}{|\Omega|}\int_{\Omega}|\Omega| \cdot (\frac{1}{2} |\nabla u^{\star}(x)|^2 + \frac{1}{2} V(x) |u^{\star}(x)|^2- f(x)u^{\star}(x))dx\\
&= \mE(u) -\mE(u^\star) \geq 0 \Rightarrow \bE[h] \geq 0.
\end{align*}
We will proceed to verify that any $\tilde{h}=\frac{\bE[h]-h}{\bE[h]+r} \in \tilde{\mS}_{r}(\Omega)$ is of bounded inf-norm. We need to prove that any $h \in \mS(\Omega)$ is of bounded inf-norm beforehand. Using boundedness condition listed in equation \ref{assp:boundedness_drm} implies:
\begin{align*}
\|h\|_{\infty} &= |\Omega|\|\frac{1}{2}\Big(|\nabla u|^2-|\nabla u^{\ast}|^2\Big) + \frac{1}{2}V(|u|^2-|u^{\ast}|^2)-f(u-u^{\ast})\|_{\infty}\\
&\leq \frac{|\Omega|}{2}\Big(\|\nabla u \|_{\infty}^2 + \|\nabla u^{\ast}\|_{\infty}^2 \Big) + \frac{|\Omega|}{2}V_{\text{max}}\Big(\|u\|_{\infty}^2+ \|u^{\ast}\|_{\infty}^2\Big) + |\Omega|\|f\|_{\infty}\Big(\|u\|_{\infty}+\|u^{\ast}\|_{\infty}\Big)\\
&\leq \frac{|\Omega|}{2} \times 2C^2 + \frac{|\Omega|}{2}V_{\text{max}} \times 2C^2 + 2|\Omega|C^2 = |\Omega|(V_{\text{max}} + 3)C^2
\end{align*}
By taking $M := |\Omega|(V_{\text{max}} + 3)C^2$, we then have $\|h\|_{\infty} \leq M$ for all $h \in \mS(\Omega)$. Note that the denominator can be lower bounded by $|\bE[h]+r| \geq r > 0$. Combining these two inequalities help us upper bound the inf-norm $\|\tilde{h}\|_{\infty} = \sup_{x \in \Omega}|\tilde{h}(x)|$ as follows:
\begin{align*}
\|\tilde{h}\|_{\infty} = \frac{\|\bE[h]-h\|_{\infty}}{|\bE[h] + r|} \leq \frac{2\|h\|_{\infty}}{r} \leq \frac{2M}{r} =: \beta.    
\end{align*}

We will then check the normalized functions $\frac{\mathbb{E}[h]-h(x)}{\bE[h] + r}$ in $\tilde{S}_{r}(\Omega)$ have bounded second moment, which is satisfied because of the regularity results of the PDE. We aim to show that there exist some constants $\alpha,
\alpha' > 0$, such that for any $h \in \mS(\Omega)$, the following inequality holds:
\begin{equation}\label{cond: talagrand strong convex_drm}
\alpha \mathbb{E}[h^2] \leq \|u-u^\ast\|_{H^1(\Omega)}^2 \leq \alpha'\mathbb{E}[h].     
\end{equation}
The RHS of the inequality follows from strong convexity of the DRM objective function proved in Theorem \ref{thm: PDE regularity_drm}:
\begin{align*}
\mathbb{E}[h] = \mE(u) - \mE(u^\ast) \geq \frac{\min\{1, V_{\text{min}}\}}{4}\|u-u^\ast\|_{H^1(\Omega)}^2
\end{align*}
The LHS of the inequality follows from boundedness condition listed in equation \ref{assp:boundedness_drm} and the QM-AM inequality:
\begin{align*}
\mathbb{E}[h^2] &= \int_{\Omega}\left[ \frac{1}{2}\Big(|\nabla u|^2-|\nabla u^{\ast}|^2\Big) + \frac{1}{2}V(|u|^2-|u^{\ast}|^2)-f(u-u^{\ast})\right]^2 dx\\
&\leq \frac{3}{4}\int_{\Omega}\Big(|\nabla u|^2-|\nabla u^{\ast}|^2\Big)^2dx + \frac{3}{4}\int_{\Omega}V^2(|u|^2-|u^{\ast}|^2)^2dx + 3\int_{\Omega}f^2(u-u^{\ast})^2dx\\
&\leq \frac{3}{4}\int_{\Omega}\Big||\nabla u| - |\nabla u^{\ast}|\Big|^2(|\nabla u| + |\nabla u^{\ast}|)^2dx + \frac{3}{4}V_{\text{max}}^2\int_{\Omega}\Big||u| - |u^{\ast}|\Big|^2(|u| + |u^{\ast}|)^2dx\\
&+3C^2\int_{\Omega}(u-u^{\ast})^2dx \leq 3C^2\int_{\Omega}|\nabla u - \nabla u^{\ast}|^2dx + 3C^2(1+V_{\text{max}}^2)\int_{\Omega}|u-u^{\ast}|^2dx\\
&\leq 3C^2(1+V_{\text{max}}^2)\|u-u^{\ast}\|_{H^1(\Omega)}^2
\end{align*}
By picking $\alpha' = \frac{4}{\min\{1, V_{\text{min}}\}}$ and $\alpha = \frac{1}{3C^2(1+V_{\text{max}}^2)}$, we have finished proving inequality \ref{cond: talagrand strong convex_drm}. Then we can can upper bound the expectation  $\bE[\tilde{h}^2]$ as:

\begin{align*}
\mathbb{E}[\tilde{h}^2] = \frac{\mathbb{E}[(h-\mathbb{E}[h])^2]}{|\bE[h]+r|^2} = \frac{\bE[h^2] - \bE[h]^2}{|\bE[h]+r|^2} \leq \frac{\bE[h^2]}{|\bE[h]+r|^2}.    
\end{align*}
Using the fact that $\bE[h] \geq 0$ and inequality \ref{cond: talagrand strong convex_drm}, we can lower bound the denominator $|\bE[h]+r|^2$ as follows:
\begin{align*}
|\bE[h]+r|^2 \geq 2\bE[h]r \geq \frac{2r\alpha}{\alpha'}\mathbb{E}[h^2].    
\end{align*}
Therefore, we can deduce that:
\begin{align*}
\mathbb{E}[\tilde{h}^2]  \leq \frac{\mathbb{E}[h^2]}{|\bE[h]+r|^2} \leq \frac{\bE[h^2]}{\frac{2r\alpha}{\alpha'}\bE[h^2]} = \frac{\alpha'}{2r\alpha} =: \sigma^2.     
\end{align*}
Hence, any function in the localized class $\tilde{\mS}_{r}(\Omega)$ is of bounded second moment.

It is easy to check that for any $\tilde{h} \in \tilde{\mS}_{r}(\Omega)$, we have
\begin{align*}
\bE[\tilde{h}] =  \frac{\bE[h]-\bE[h]}{\bE[h]+r} = 0,
\end{align*}
\emph{i.e.} any function in the localized class $\tilde{\mS}_{r}(\Omega)$ is of zero mean.

Now we have verified that any function $\tilde{h} \in \tilde{\mS}_{r}(\Omega)$ satisfies all the required conditions. By taking $\mu$ to be the uniform distribution on the domain $\Omega$ and applying Talagrand's Concentration inequality given in Lemma \ref{lem:Talagrand ineq}, we have:
\begin{align*}
\mathbb{P}_{x}\left[\sup_{\tilde{h} \in \tilde{\mS}_{r}(\Omega)}\frac{1}{n}\sum_{i=1}^{n}\tilde{h}(x_i) \geq 2\sup_{\tilde{h} \in \tilde{\mS}_{r}(\Omega)}\mathbb{E}_{x'}\Big[\frac{1}{n}\sum_{i=1}^{n}\tilde{h}(x_i')\Big]+ \sqrt{\frac{2t\sigma^2}{n}}+\frac{2t\beta}{n}\right] \leq e^{-t}.
\end{align*}

By using the upper bound deduced above and plugging in the expressions of $\beta$ and $\sigma$, we can rewrite Talagrand's Concentration Inequality in the following way. With probability at least $1-e^{-t}$, the inequality below holds:
\begin{align*}
\frac{1}{n}\sum_{i=1}^{n}\tilde{h}(x_i) \leq \sup_{\tilde{h} \in \tilde{\mS}_{r}(\Omega)}\frac{1}{n}\sum_{i=1}^{n}\tilde{h}(x_i) &\leq 2\sup_{\tilde{h} \in \tilde{\mS}_{r}(\Omega)}\mathbb{E}_{x'}\Big[\frac{1}{n}\sum_{i=1}^{n}\tilde{h}(x_i')\Big]+ \sqrt{\frac{2t\sigma^2}{n}}+\frac{2t\beta}{n}\\
&\leq \frac{16\phi(r)}{r} + \sqrt{\frac{t\alpha'}{n\alpha r}}+\frac{4Mt}{nr} =: \psi(r).    
\end{align*}
Let's pick the critical radius $r_0$ to be:
\begin{equation}
\label{thresholdradius_drm}
\begin{aligned}
r_0 = \max\{2^{14}r^{\ast}, \frac{24Mt}{n}, \frac{36\alpha' t}{\alpha n}\}. 
\end{aligned}    
\end{equation}
Note that concavity of the function $\phi$ implies that $\phi(r) \leq r$ for any $r \geq r^{\ast}$. Combining this with the first inequality listed in \ref{peelingcond_drm} yields:
\begin{align*}
\frac{16\phi(r)}{r} &\leq \frac{2^{11} \phi(\frac{r_0}{2^{14}})}{2^{14}\frac{r_0}{2^{14}}} = \frac{1}{8} \times \frac{\phi(\frac{r_0}{2^{14}})}{\frac{r_)}{2^{14}}} \leq \frac{1}{8}.
\end{align*}
On the other hand, applying equation \ref{thresholdradius_drm} yields:
\begin{align*}
\sqrt{\frac{\alpha' t}{n \alpha r_0}} &\leq \sqrt{\frac{\alpha' t}{n \alpha}\frac{\alpha n}{36\alpha' t}} = \frac{1}{6},\\
\frac{4Mt}{nr_0} &\leq \frac{4Mt}{n} \times \frac{n}{24Mt} = \frac{1}{6}.
\end{align*}
Summing the three inequalities above implies:
\begin{align*}
\psi(r_0) = \frac{16\phi(r_0)}{r_0} + \sqrt{\frac{t\alpha'}{n\alpha r_0}}+\frac{4Mt}{nr_0} \leq \frac{1}{8}+\frac{1}{6}+\frac{1}{6} < \frac{1}{2}.
\end{align*}
By picking $r=r_0$, we can further deduce that for any function $u \in \mF(\Omega)$, the following inequality holds with probability $1-e^{-t}$: 
\begin{align*}
\frac{\mE(u) - \mE(u^\ast)- \mE_{n}(u) + \mE_{n}(u^\ast)}{\mE(u) - \mE(u^\ast)+r_0} = \frac{1}{n}\sum_{i=1}^{n}\tilde{h}(x_i) &\leq \psi(r_0) < \frac{1}{2}.
\end{align*}
Multiplying the denominator on both sides indicates:
\begin{align*}
\Delta \mE_{\text{gen}} = \mE(u) - \mE(u^\ast)- \mE_{n}(u) + \mE_{n}(u^\ast) \leq \frac{1}{2}\Big[\mE(u) - \mE(u^\ast)\Big] + \frac{1}{2}r_0= \frac{1}{2}\Delta \mE^{(n)} + \frac{1}{2}r_0.
\end{align*}
Substituting the upper bound above into the decomposition $\Delta \mE^{(n)} \leq \Delta E_{\text{gen}} + \frac{3}{2}\Delta E_{\text{app}}+\frac{t}{2n}$ yields that with probability $1-e^{-t}$, we have:
\begin{align*}
\Delta \mE^{(n)} \leq \Delta \mE_{\text{gen}} + \frac{3}{2}\Delta \mE_{\text{app}} + \frac{t}{2n} \leq \frac{1}{2}\Delta \mE^{(n)} + \frac{1}{2}r_0 + \frac{3}{2}\Delta \mE_{\text{app}}+\frac{t}{2n}.    
\end{align*}
Simplifying the inequality above yields that with probability $1-e^{-t}$, we have:
\begin{align*}
\Delta \mE^{(n)} &\leq r_0 + 3\Delta \mE_{\text{app}} + \frac{t}{n} = 3\inf_{u_{\mF} \in \mF(\Omega)}\Big(\mE(u_{\mF}) - \mE(u^{\star})\Big) + \max\{2^{14}r^{\ast}, 24M\frac{t}{n}, \frac{36\alpha'}{\alpha}\frac{t}{n}\} + \frac{t}{n}\\
&\lesssim \inf_{u_{\mF} \in \mF(\Omega)}\Big(\mE(u_{\mF}) - \mE(u^{\star})\Big)+ \max\Big\{r^*,\frac{t}{n}\Big\}
\end{align*}
Moreover, using strong convexity of the DRM objective function proved in Theorem \ref{thm: PDE regularity_drm} implies:
\begin{align*}
\Delta \mE^{(n)} = \mE(\hat u_{\text{DRM}})-\mE(u^\ast) \geq \frac{\min\{1, V_{\text{min}}\}}{4}\|\hat u_{\text{DRM}}-u^\ast\|_{H^1(\Omega)}^2     
\end{align*}
Combining the two bounds above yields that with probability $1-e^{-t}$, we have:
\begin{align*}
\|\hat u_{\text{DRM}}-u^\ast\|_{H^1(\Omega)}^2  \lesssim  \inf_{u_{\mF} \in \mF(\Omega)}\Big(\mE(u_{\mF}) - \mE(u^{\star})\Big)+ \max\Big\{r^*,\frac{t}{n}\Big\}  
\end{align*}
\end{proof}

\begin{theorem}(Final Upper Bound of DRM with Deep Neural Network Estimator)
With proper assumptions, consider the sparse Deep Neural Network function space $\Phi(L,W,S,B)$ with parameters $L =O(1), \ W=O(n^{\frac{d}{d+2s-2}}), \ S=O(n^{\frac{d}{d+2s-2}}), \ B=O(1)$, then the Deep ritz estimator $\hat u_{\text{DRM}}^{\text{DNN}} = \min_{u \in \Phi(L,W,S,B)}\mE_n^{\text{DRM}}(u)$ satisfies the following upper bound with high probability:
\begin{align*}
\|\hat u_{\text{DRM}}^{\text{DNN}} -u^\ast\|_{H^1}^2 \lesssim n^{-\frac{2s-2}{d+2s-2}}\log n.   
\end{align*} 
\end{theorem}

\begin{proof}

On the one hand, by taking $s=1$ and $p=2$ in Theorem \ref{thm: approximation NN} proved above, we have that there exists some Deep Neural Network $u_{\text{DNN}} \in \Phi(L,W,S,B)$ with $L=O(1), W=O(N), S=O(N), B=O(N)$, such that. 
\begin{align*}
\|u_{\text{DNN}}-u^{\ast}\|_{H^1(\Omega)}^2 \leq  N^{-\frac{2s-2}{d}}\|u^{\ast}\|_{H^s(\Omega)}.     
\end{align*}
Applying strong convexity of the DRM objective function proved in Section \ref{subsection:lossfunc} further implies:
\begin{align*}
\Delta \mE_{\text{app}} \lesssim \|u_{\text{DNN}}-u^{\ast}\|_{H^1(\Omega)}^2 \leq N^{-\frac{2s-2}{d}}.   
\end{align*}
On the other hand, from Lemma \ref{lemma:localDRM} proved above, we know that the function $\phi(\rho)$ that upper bounds the local Rademacher complexity of the Deep Neural Network space is of the same magnitude as $\sqrt{\frac{S3^L\rho}{n}\log\left(BWn\right)}$. By plugging in the magnitudes of $L,W,S,B$, we can determine the critical radius $r^\ast$:
$$\sqrt{\frac{r^\ast 3^{L}S}{n}\log(BWn)} \simeq \sqrt{\frac{r^\ast N}{n}(2\log N +\log n)} \simeq r^\ast \Rightarrow r^\ast \simeq \frac{N(\log N + \log n)}{n}.$$
Combining the two bounds above with Theorem \ref{meta:drm} yields that with high probability, we have:
$$\|\hat u_{\text{DRM}}^{\text{DNN}} -u^\ast\|_{H^1}^2 \lesssim \Delta \mE_{\text{app}} + \hat{r} \lesssim N^{-\frac{2(s-1)}{d}} + \frac{N(\log N + \log n)}{n}.$$
By equating the two terms above, we can solve for the optimal $N$ that yields the desired bound:
$$N^{-\frac{2(s-1)}{d}} \simeq \frac{N}{n} \Rightarrow N \simeq n^{\frac{d}{d+2s-2}}.$$
Plugging in the optimal $N$ gives us the magnitudes of the four parameters $L =O(1), \ W=O(n^{\frac{d}{d+2s-2}}), \ S=O(n^{\frac{d}{d+2s-2}}), \ B=O(n^{\frac{d}{d+2s-2}})$, as well as the final rate:
$$\|\hat u_{\text{DRM}}^{\text{DNN}} -u^\ast\|_{H^1}^2 \lesssim N^{-\frac{2(s-1)}{d}} + \frac{N\log N}{n} \lesssim n^{-\frac{2(s-1)}{d+2(s-1)}}\log n.$$
\end{proof}

\begin{theorem}(Final Upper Bound of DRM with Truncated Fourier Series Estimator)With proper assumptions, consider the Deep Ritz objective with a plug in Fourier Series estimator $\hat u_{\text{DRM}}^{\text{Fourier}} = \min_{u \in \mF_{\xi}(\Omega)}\mE_{n}^{\text{DRM}}(u)$ with $\xi = \Theta(n^{\frac{1}{d+2s-2}})$, then with high probability we have
\begin{align*}
\|\hat u_{\text{DRM}}^{\text{Fourier}} -u^\ast\|_{H^1}^2 \lesssim n^{-\frac{2s-2}{d+2s-2}}.
\end{align*}
\end{theorem}

\begin{proof}
On the one hand, from Lemma \ref{lem:local rademacher of fourier} and Lemma \ref{lem:local rademacher of fourier gradient} proved above, we know that the function $\phi(\rho)$ that upper bounds the local Rademacher complexity for Truncated Fourier Series can be dominated by the term

\begin{equation}
    \begin{aligned}
    R_{n}(\mS_{\rho}(\Omega))&\lesssim R_n \left( \Big\{u-u^\ast: u\in \mF_{\rho,\xi}(\Omega),\|u -u^\ast\|_{L^2(\Omega)} \le \sqrt{\rho}\Big\} \right)\\
&+R_n \left(\Big\{\|\nabla u - \nabla u^\ast\|: u \in \mF_{\rho,\xi}(\Omega),\Big\|\nabla u - \nabla u^\ast\Big\|_{L^2(\Omega)} \leq 
\sqrt{\rho} \Big\}\right)\\
&\lesssim \bE_{X}\left[\bE_{\sigma}\Big[\sup_{f \in \mF_{\rho,\xi}(\Omega)}\frac{1}{n}\sum_{i=1}^{n}\sigma_{i} (u(X_{i})- u^\ast(X_{i})) \Big| \|u-\Pi_\xi u^\ast\|_{H^1(\Omega)}^2\le \rho \Big]\right]\\
&+ \bE_{X}\left[\bE_{\sigma}\Big[\sup_{f \in \mF_{\rho,\xi}(\Omega)}\frac{1}{n}\sum_{i=1}^{n}\sigma_{i}\|\nabla u(X_{i})-\nabla u^\ast(X_{i})\| \ \Big| \|u-\Pi_\xi u^\ast\|_{H^1(\Omega)}^2\le \rho \Big]\right]\\
    &+\bE_{X}\left[\bE_{\sigma}\Big[\frac{1}{n}\sum_{i=1}^{n}\sigma_{i}\|\nabla \Pi_{>\xi} u^\ast(X_{i})\| \Big]\right]+\bE_{X}\left[\bE_{\sigma}\Big[\frac{1}{n}\sum_{i=1}^{n}\sigma_{i} \Pi_{>\xi} u^\ast(X_{i}) \Big]\right]\\
&\lesssim \bE_{X}\left[\bE_{\sigma}\Big[\sup_{f \in \mF_{\rho,\xi}(\Omega)}\frac{1}{n}\sum_{i=1}^{n}\sigma_{i} f(X_{i}) \ \Big| \|f\|_{H^1(\Omega)}^2\le \rho \Big]\right]\\
&+\bE_{X}\left[\bE_{\sigma}\Big[\sup_{f \in \mF_{\rho,\xi}(\Omega)}\frac{1}{n}\sum_{i=1}^{n}\sigma_{i}\|\nabla f(X_{i})\| \ \Big| \|f\|_{H^1(\Omega)}^2\le \rho \Big]\right]\\
    &+\bE_{X}\left[\bE_{\sigma}\Big[\frac{1}{n}\sum_{i=1}^{n}\sigma_{i}\|\nabla \Pi_\xi u^\ast(X_{i})\| \Big]\right]+\bE_{X}\left[\bE_{\sigma}\Big[\frac{1}{n}\sum_{i=1}^{n}\sigma_{i} \Pi_{>\xi} u^\ast(X_{i}) \Big]\right]\\
    &\le \sqrt{\frac{\rho}{n}}\xi^{\frac{d}{2}} +\sqrt{\frac{\|\Pi_\xi f\|_{H^1}^2}{n}}\le \sqrt{\frac{\rho}{n}}\xi^{\frac{d}{2}} + \sqrt{\frac{\xi^{-2(s-1)}}{n}}\lesssim \sqrt{\frac{\rho}{n}}\xi^{\frac{d}{2}} + \frac{1}{n}+\xi^{-2(s-1)},
    \end{aligned}
\end{equation}

where $\Pi_{\xi}u:=\sum_{\|z\|_\infty\le\xi} u_z\phi_z(x)$ is the projection to the Fourier basis whose frequency is smaller than $\xi$ and $\Pi_{\xi}u:=\sum_{\|z\|_\infty\le\xi} u_z\phi_z(x)$ is the projection to the Fourier basis whose frequency is larger than $\xi$. Then, the critical radius $\hat{r}$  can be determined as follows:
\begin{align*}
\sqrt{\frac{\rho}{n}}\xi^{\frac{d}{2}} + \frac{1}{n}+\xi^{-2(s-1)}\simeq \rho \Rightarrow \hat{r} \simeq \frac{\xi^{d}}{n}+\frac{1}{n}+\xi^{-2(s-1)},
\end{align*}
On the other hand, by taking $\alpha = s$ and $\beta =1$ in Lemma \ref{lem: approximation Fourier} and applying strong convexity of the DRM objective function proved in Theorem \ref{thm: PDE regularity_drm}, we can upper bound the approximation error $\Delta \mE_{\text{app}}$ as below:
\begin{align*}
\Delta \mE_{\text{app}} \lesssim \xi^{-2(s-1)},    
\end{align*}
By equating the two terms above, we can solve for $\xi$ that yields the desired bound:
\begin{align*}
\frac{\xi^{d}}{n} \simeq \xi^{-2(s-1)} \Rightarrow \xi \simeq n^{\frac{1}{d+2s-2}},   
\end{align*}
Plugging in the expression of $\xi$ gives the final upper bound:
\begin{align*}
\mathbb{E}_{x \sim \mu}[\Delta \mE_{n}] \lesssim \hat{r} + \Delta \mE_{\text{app}} \lesssim \frac{\xi^{d}}{n} + \xi^{-2(s-1)} \simeq n^{-\frac{2s-2}{d+2s-2}}.    
\end{align*}

\end{proof}

\paragraph{Physics Informed Neural Network.}
Then we aim to provide the upper bound for the Physics Informed Neural Network with a Similar meta-theorem followed by upper bounds of DNN and truncated Fourier Estimators. For simplicity, we drop the proof of the meta-theorem to appendix for all the proof follows the similar idea of the DRM one.

\begin{theorem}[Meta-theorem for Upper Bounds of Physics Informed Neural Network]
\label{meta:pinn}
Let $u^\ast \in H^{s}(\Omega)$ denote the true solution to the PDE model with Dirichlet boundary condition:
\begin{equation}
\begin{aligned}
   -\Delta u + V u & = f \text{ on } \Omega,\\ 
     u & =  0 \text{ on } \partial \Omega, 
\end{aligned}
\end{equation}
where $f\in L^2(\Omega)$ and $V\in L^\infty(\Omega)$ with $0< V_{\min} \leq V(x) \leq V_{\max}>0$.
For a fixed function space $\mF(\Omega)$, consider the empirical loss induced by the Physics Informed Neural Network:
\begin{equation}
\begin{aligned}
      \mE_{n}(u)  & = \frac{1}{n} \sum_{j=1}^n\Big[ |\Omega| \cdot  \Big(\Delta u(X_j)-V(X_j)u(X_j)+f(X_j)\Big)^2\Big],
\end{aligned}
\end{equation}
where $\{X_{j}\}_{j=1}^{n}$ are datapoints uniformly sampled from the domain $\Omega$. Then the Physics Informed Neural Network estimator associated with function space $\mF(\Omega)$ is defined as the minimizer of $\mE_{n}(u)$ over the function space $\mF(\Omega)$:
\begin{align*}
\hat u_{\text{PINN}} = \min_{u \in \mF(\Omega)}\mE_{n}(u)    
\end{align*}
Moreover, we assume that there exists some constant $C > 0$ such that all function $u$ in the function space $\mF(\Omega)$, the real solution $u^{\ast}$ and $f,V$ satisfy the following two conditions.
\begin{itemize}
    \item The gradients and function value are uniformly bounded 
    \begin{equation}\label{assp:boundedness_pinn}
    \begin{aligned}
    \max \Big\{&\sup_{u \in \mF(\Omega)}\|u\|_{L^{\infty}(\Omega)}, \sup_{u \in \mF(\Omega)}\|\nabla u\|_{L^{\infty}(\Omega)}, \sup_{u \in \mF(\Omega)}\|\Delta u\|_{L^{\infty}(\Omega)},\\
    &\|u^{\ast}\|_{L^{\infty}(\Omega)}, \|\nabla u^{\ast}\|_{L^{\infty}(\Omega)},\|\Delta u^{\ast}\|_{L^{\infty}(\Omega)}, V_{max}, \|f\|_{L^{\infty}(\Omega)} \Big \} \leq C.
    \end{aligned}
    \end{equation} 
\item All the functions in  the function space $\mF(\Omega)$ satisfies the boundary condition
\begin{equation*}
    u=0 \text{ on } \partial \Omega.
\end{equation*}
\end{itemize}

At the the same time, for any $\rho > 0$, we assume the Rademacher complexity of a localized function space $\mT_{\rho}(\Omega) := \Big\{h :=|\Omega| \cdot \left[ (\Delta u -Vu +f)^2-(\Delta u^\ast -Vu^\ast +f)^2\right] \ \ \Big | \ \|u-u^\ast\|_{H^2}^2\leq \rho \Big \}$ can be upper bounded by a sub-root function $\phi = \phi(\rho): [0, \infty) \rightarrow [0,\infty)$, \emph{i.e.}
\begin{equation}\label{peelingcond_pinn}
\phi(4\rho) \leq 2\phi(\rho) \text{ and } R_{n}(\mT_{\rho}(\Omega)) \leq \phi(\rho) \ (\forall \ \rho > 0).
\end{equation}
For all constant $t>0$. We denote $r^*$ to be the solution of the fix point equation of local Rademacher complexity $r = \phi(r)$. There exists a constant $C_p$ such that for probability $1-C_p\exp(-t)$, we have the following upper bound for the Physics Informed Neural Network Estimator
$$
\|\hat u_{\text{PINN}} -u^\ast\|_{H^2}^2 \lesssim \inf_{u_{\mF} \in \mF(\Omega)}\Big(\mE(u_{\mF}) - \mE(u^{\star})\Big)+ \max\Big\{r^*,\frac{t}{n}\Big\}.
$$

\end{theorem}

Then we aim to calculate the final upper bound for DNN and truncated Fourier series based PINN estimator.

\begin{theorem}(Informal Upper Bound of PINN with Deep Neural Network Estimator) With proper assumptions, consider the sparse Deep Neural Network function space $\Phi(L,W,S,B)$ with parameters $L =O(1), \ W=O(n^{\frac{d}{d+2s-4}}), \ S=O(n^{\frac{d}{d+2s-4}}), \ B=O(1)$, then the Physics Informed estimator $\hat u_{\text{PINN}}^{\text{DNN}} = \min_{u \in \Phi(L,W,S,B)}\mE_{n}^{\text{PINN}}(u)$ satisfies the following upper bound with high probability:
\begin{align*}
\|\hat u_{\text{PINN}}^{\text{DNN}} -u^\ast\|_{H^2}^2 \lesssim n^{-\frac{2s-4}{d+2s-4}}\log n.   
\end{align*} 
\end{theorem}

\begin{proof}

On the one hand, by taking $s=2$ and $p=2$ in Theorem \ref{thm: approximation NN} proved above, we have that there exists some Deep Neural Network $u_{\text{DNN}} \in \Phi(L,W,S,B)$ with $L=O(1), W=O(N), S=O(N), B=O(1)$, such that. 
\begin{align*}
\|u_{\text{DNN}}-u^{\ast}\|_{H^2(\Omega)}^2  \leq  N^{-\frac{2s-4}{d}}\|u\|_{H^s(\Omega)}.     
\end{align*}
Applying strong convexity of the DRM objective function proved in Section \ref{subsection:lossfunc} further implies:
\begin{align*}
\Delta \mE_{\text{app}} \lesssim \|u_{\text{DNN}}-u^{\ast}\|_{H^2(\Omega)}^2 \leq N^{-\frac{2s-4}{d}}.   
\end{align*}
On the other hand, from lemma \ref{lemma:localPINN} proved above, we know that the function $\phi(\rho)$ that upper bounds the local Rademacher complexity of the Deep Neural Networks $u_{\text{DNN}}$ is dominated by the term $\sqrt{\frac{\rho 3^{L}S}{n}\log(W(B \vee 1)n)}$. By plugging in the magnitudes of $L,W,S,B$, we can determine the critical radius $\hat{r}$:
$$\sqrt{\frac{\rho 3^{L}S}{n}\log(W(B \vee 1)n)} \simeq \sqrt{\frac{\rho N}{n}(\log N +\log n)} \simeq \rho \Rightarrow \hat{r} \simeq \frac{N(\log N + \log n)}{n}.$$
Combining the two bounds above gives us:
$$\mathbb{E}_{x \sim \mu}[\Delta \mE_{n}] \lesssim \Delta \mE_{\text{app}} + \hat{r} \lesssim N^{-\frac{2(s-2)}{d}} + \frac{N(\log N + \log n)}{n}.$$
By equating the two terms above, we can solve for the optimal $N$ that yields the desired bound:
$$N^{-\frac{2(s-2)}{d}} \simeq \frac{N}{n} \Rightarrow N \simeq n^{\frac{d}{d+2s-4}}.$$
Plugging in the optimal $N$ gives us the magnitudes of the four parameters $L =O(1), \ W=O(n^{\frac{d}{d+2s-4}}), \ S=O(n^{\frac{d}{d+2s-4}}), \ B=O(1)$, as well as the final rate:
$$\mathbb{E}_{x \sim \mu}[\Delta \mE_{n}]  \lesssim N^{-\frac{2(s-2)}{d}} + \frac{N\log N}{n} \lesssim n^{-\frac{2(s-2)}{d+2(s-2)}}\log n.$$
\end{proof}

\begin{theorem}(Informal Upper Bound of PINN with Truncated Fourier Series Estimator) With proper assumptions, consider the Physics Informed Neural Network objective with a plug-in Fourier Series estimator $\hat u_{\text{PINN}}^{\text{Fourier}} = \min_{u \in \mF_{\xi}(\Omega)}\mE_{n}^{\text{PINN}}(u)$ with $\xi = \Theta(n^{\frac{1}{d+2s-4}})$, then with high probability we have
\begin{align*}
\|\hat u_{\text{PINN}}^{\text{Fourier}} -u^\ast\|_{H^2}^2 \lesssim n^{-\frac{2s-4}{d+2s-4}}.
\end{align*}
\end{theorem}
\begin{proof}
On the one hand, from Lemma \ref{lem:local rademacher of fourier} and Lemma \ref{lem:local rademacher of fourier gradient} proved above, we know that the function $\phi(\rho)$ that upper bounds the local Rademacher complexity for Truncated Fourier Series is dominated by 

\begin{equation}
    \begin{aligned}
    R_{n}(\mS_{\rho}(\Omega))&\lesssim R_n \left(\Big\{|\Delta u - \Delta u^\ast|: u \in \mF_{\rho,\xi}(\Omega),|\Delta u - \Delta u^\ast|_{L^2(\Omega)} \leq 
\sqrt{\rho} \Big\}\right)\\
&\lesssim \bE_{X}\left[\bE_{\sigma}\Big[\sup_{f \in \mF_{\rho,\xi}(\Omega)}\frac{1}{n}\sum_{i=1}^{n}\sigma_{i}|\Delta u(X_{i})-\Delta u^\ast(X_{i})| \ \Big| \|u-\Pi_\xi u^\ast\|_{H^2(\Omega)}^2\le \rho \Big]\right]\\
    &+\bE_{X}\left[\bE_{\sigma}\Big[\frac{1}{n}\sum_{i=1}^{n}\sigma_{i}|\Delta \Pi_{>\xi} u^\ast(X_{i})| \Big]\right]\\
&\lesssim \bE_{X}\left[\bE_{\sigma}\Big[\sup_{f \in \mF_{\rho,\xi}(\Omega)}\frac{1}{n}\sum_{i=1}^{n}\sigma_{i}|\Delta f(X_{i})| \ \Big| \|f\|_{H^2(\Omega)}^2\le \rho \Big]\right]\\
    &+\bE_{X}\left[\bE_{\sigma}\Big[\frac{1}{n}\sum_{i=1}^{n}\sigma_{i}\|\nabla \Pi_\xi u^\ast(X_{i})\| \Big]\right]\\
    &\lesssim \bE_{X}\left[\bE_{\sigma}\Big[\sup_{f \in \mF_{\rho,\xi}(\Omega)}\frac{1}{n}\sum_{i=1}^{n}\sigma_{i}|\Delta f(X_{i})| \ \Big| \|f\|_{H^2(\Omega)}^2\le \rho \Big]\right]\\
    &+\bE_{X}\left[\bE_{\sigma}\Big[\frac{1}{n}\sum_{i=1}^{n}\sigma_{i}\|\Delta \Pi_\xi u^\ast(X_{i})\| \Big]\right]\\
    &\le \sqrt{\frac{\rho}{n}}\xi^{\frac{d}{2}} +\sqrt{\frac{\|\Pi_\xi f\|_{H^2}^2}{n}}\le \sqrt{\frac{\rho}{n}}\xi^{\frac{d}{2}} + \sqrt{\frac{\xi^{-2(s-1)}}{n}}\lesssim \sqrt{\frac{\rho}{n}}\xi^{\frac{d}{2}} + \frac{1}{n}+\xi^{-2(s-2)}
    \end{aligned}
\end{equation}
Thus, the localization radius $\hat{r}$  can be determined as follows:
\begin{align*}
\sqrt{\frac{\rho}{n}}\xi^{\frac{d}{2}}+ \frac{1}{n}+\xi^{-2(s-2)} \simeq \rho \Rightarrow \hat{r} \simeq \frac{\xi^{d}}{n}+ \frac{1}{n}+\xi^{-2(s-2)},
\end{align*} On the other hand, by taking $\alpha = s$ and $\beta =1$ in Lemma \ref{lem: approximation Fourier} and applying strong convexity of the DRM objective function proved in Theorem \ref{thm: PDE regularity_drm}, we can upper bound the approximation error $\Delta \mE_{\text{app}}$ as below:
\begin{align*}
\Delta \mE_{\text{app}} \lesssim \xi^{-2(s-2)},    
\end{align*}
By equating the two terms above, we can solve for $\xi$ that yields the desired bound:
\begin{align*}
\frac{\xi^{d}}{n} \simeq \xi^{-2(s-2)} \Rightarrow \xi \simeq n^{\frac{1}{d+2s-4}},   
\end{align*}
Plugging in the expression of $\xi$ gives the final upper bound:
\begin{align*}
\mathbb{E}_{x \sim \mu}[\Delta \mE_{n}] \lesssim \hat{r} + \Delta \mE_{\text{app}} \lesssim \frac{\xi^{d}}{n} + \xi^{-2(s-2)} \simeq n^{-\frac{2s-4}{d+2s-4}}.    
\end{align*}
\end{proof}

\paragraph{Remark.} 
\begin{itemize}
    \item There is a common belief that Machine learning based PDE solvers can break the curse of dimensionality \cite{weinan2018deep,grohs2018proof,lanthaler2021error}. However we obtained an $n^{-\frac{2s-2}{2s-4+d}}$ convergence rate which can become super slow in high dimension. Our analysis showed that it is essential to constrain the function space to break the curse of dimensionality. \cite{lu2021priori} considered the DRM in Barron spaces. \cite{ongie2019function} showed that functions in the Barron space enjoy a smoothness $s$ at the same magnitude as $d$ , which will also leads to convergence rate independent of the dimension using our upper bound.  Neural network can also approximate mixed sparse grid spaces \cite{montanelli2019new,suzuki2018adaptivity}, function on manifold \cite{nitanda2020optimal,chen2019nonparametric}  without curse of dimensionality. Combined with these approximation bounds, we can also achieve a bound that breaks the curse of dimensionality using Theorem \ref{meta:pinn} and \ref{meta:drm}. In this paper, we aim to consider the statistical power of the loss function in common function spaces and leave the curse of dimensionality as a separate topic. 
    
    \item  Our bound is faster than the concurrent bound \cite{duan2021convergence,jiao2021convergence} for we provided a fast rate $O(1/n)$ by utilizing the strong convexity of the objective function and improves the convergence rate from $n^{-\frac{2s-2}{d+4s-4}}$ to $n^{-\frac{2s-2}{d+2s-2}}$ for Deep Ritz and from $n^{-\frac{2s-4}{d+4s-8}}$ to $n^{-\frac{2s-4}{d+2s-4}}$ for PINN. Comparing to the lower bound provided in Section\ref{section:lowerbound}, we show that our bound for PINN is near optimal and we'll let our bound for DRM become near optimal in the next section. 
    
    \item For upper bound of DRM, due to a technical issue, we assume the observation we access is clean, \emph{i.e} $f_i=f(X_i)$. We conjecture that add noising on observation will not effect the rate and leave this to future work.
    
\end{itemize}

\section{Modified Deep Ritz Methods}

Comparing the lower bound in Section \ref{section:lowerbound} and the upper bound in Section  \ref{section:upper}, we find out that the Physics Informed Neural Network achieved min-max optimality while the Deep Ritz Method does not. In this section,  we proposed a modified version of deep Ritz which can be statistically optimal.

As discussed in Appendix \ref{appendix:subopt}, the reason behind the suboptimality of DRM comes from the high complexity introduced via the uniform concentration bound of the gradient term in the variational form. At the same time, we further observed that the $\int \|\nabla u\|^2 dx$ does not require any query from the right hand side function $f$, which means that we can easily make another splitted sample to approximate the $\int \|\nabla u\|^2 dx$ term more precisely.

\begin{equation}\begin{aligned}
      \mE_{N,n}(u)  & = \frac{1}{N} \sum_{j=1}^N\Big[ |\Omega| \cdot  \frac{1}{2} \|\nabla u(X_j^\prime)\|^2 \Big]+\frac{1}{n} \sum_{j=1}^n\Big[ |\Omega| \cdot  \Big( \frac{1}{2} V(X_j) |u(X_j)|^2- f_ju(X_j) \Big)\Big]
    \end{aligned}
\end{equation}

Once we sampled more data for approximating $\int |\nabla u|^2 dx$, we can achieve an near optimal bound for the Truncated Fourier Estimator when $
\frac{N}{n}\gtrsim n^{\frac{2}{d+2s-4}}
$. The proof is based on a similar meta-theorem as following.
\begin{theorem}[Meta-theorem for Upper Bounds of Modified Deep Ritz Method]
\label{theorem:MDRM}
Let $u^\ast \in H^{s}(\Omega)$ denote the true solution to the PDE model with Dirichlet boundary condition:
\begin{equation}
\begin{aligned}
   -\Delta u + V u & = f \text{ on } \Omega,\\ 
     u & =  0 \text{ on } \partial \Omega, 
\end{aligned}
\end{equation}
where $f\in L^2(\Omega)$ and $V\in L^\infty(\Omega)$ with $0< V_{\min} \leq V(x) \leq V_{\max}>0$.
For a fixed function space $\mF(\Omega)$, consider the empirical loss induced by the Modified Deep Ritz Method $(N \geq n)$:
\begin{equation}
\begin{aligned}
      \mE_{N,n}(u)  & = \frac{1}{N} \sum_{i=1}^N\Big[ |\Omega| \cdot  \frac{1}{2} |\nabla u(X_i^\prime)|^2 \Big]+\frac{1}{n} \sum_{j=1}^n\Big[ |\Omega| \cdot  \Big( \frac{1}{2} V(X_j) |u(X_j)|^2- f(X_j)u(X_j) \Big)\Big],
\end{aligned}
\end{equation}
where $\{X_{i}'\}_{i=1}^{N}$ and $\{X_{j}\}_{j=1}^{n}$ are datapoints uniformly and independently sampled from the domain $\Omega$. Then the Modified Deep Ritz estimator associated with function space $\mF(\Omega)$ is defined as the minimizer of $\mE_{N,n}(u)$ over the function space $\mF(\Omega)$:
\begin{align*}
\hat u_{\text{MDRM}} = \min_{u \in \mF(\Omega)}\mE_{N,n}(u)    
\end{align*}
Moreover, we assume that there exists some constant $C > 0$ such that all function $u$ in the function space $\mF(\Omega)$, the real solution $u^{\ast}$ and $f,V$ satisfy the following two conditions.
\begin{itemize}
    \item The gradients and function value are uniformly bounded 
    \begin{equation}\label{assp:boundedness_modified}
    \max \Big\{\sup_{u \in \mF(\Omega)}\|u\|_{L^{\infty}(\Omega)}, \sup_{u \in \mF(\Omega)}\|\nabla u\|_{L^{\infty}(\Omega)}, \|u^{\ast}\|_{L^{\infty}(\Omega)}, \|\nabla u^{\ast}\|_{L^{\infty}(\Omega)}, V_{max}, \|f\|_{L^{\infty}(\Omega)} \Big \} \leq C.
\end{equation} 
\item All the functions in  the function space $\mF(\Omega)$ satisfy the boundary condition
\begin{equation*}
    u=0 \text{ on } \partial \Omega.
\end{equation*}
\end{itemize}

At the the same time, for any $\rho > 0$, we assume the Rademacher complexity of two localized function spaces 

$$
\mS_{\rho}(\Omega) := \Big\{(h_1,h_2)\big|h1 :=|\Omega| \cdot \left[ \frac{1}{2}\Big(|\nabla u|^2-|\nabla u^{\ast}|^2\Big)\right], \ h2 :=|\Omega| \cdot \left[\frac{1}{2}V(|u|^2-|u^{\ast}|^2)-f(u-u^{\ast})\right], \|u-u^\ast\|_{H^1}^2\leq \rho \Big \}
$$

can be upper bounded by a sub-root function $\phi = \phi(\rho): [0, \infty) \rightarrow [0,\infty)$, \emph{i.e.}
\begin{equation}\label{peelingcond_mdrm}
\phi(4\rho) \leq 2\phi(\rho) \text{ and } R_{N,n}(\mS_{\rho}(\Omega)) \leq \phi(\rho) \ (\forall \ \rho > 0),
\end{equation}

where $R_{N,n}(\mS):=R_{N}(\{h_1|(h_1,h_2)\in\mS\})+R_{n}(\{h_2|(h_1,h_2)\in\mS\})$. For all constant $t>0$. We denote $r^{\ast}$ to be the solution of the fix point equation of local Rademacher complexity $r = \phi(r)$. 
There exist a constant $C_p$ such that for probability $1-C_p\exp(-t)$, we have the following upper bound for the Modified Deep Ritz Estimator
$$
\|\hat u_{\text{MDRM}} -u^\ast\|_{H^1}^2 \lesssim\inf_{u_{\mF} \in \mF(\Omega)}\Big(\mE(u_{\mF}) - \mE(u^{\star})\Big)+ \max\Big\{r^*,\frac{t}{n}\Big\}.
$$
\end{theorem}

\begin{theorem}(Informal Upper Bound of DRM with Truncated Fourier Series Estimator)With proper assumptions, consider the Deep Ritz objective with a plug in Fourier Series estimator $\hat u_{\text{DRM}}^{\text{Fourier}} = \min_{u \in \mF_{\xi}(\Omega)}\mE_{n}^{\text{DRM}}(u)$ with $\xi = \Theta(n^{\frac{1}{d+2s-4}})$ and $
\frac{N}{n}\gtrsim n^{\frac{2}{d+2s-4}}
$, then we have
\begin{align*}
\|\hat u_{\text{DRM}}^{\text{Fourier}} -u^\ast\|_{H^1}^2 \lesssim n^{-\frac{2s-2}{d+2s-4}}.
\end{align*}
\end{theorem}

\begin{proof}
On the one hand, from Lemma \ref{lem:local rademacher of fourier} and Lemma \ref{lem:local rademacher of fourier gradient} proved above, we know that the function $\phi(\rho)$ that upper bounds the local Rademacher complexity is dominated by the term $\sqrt{\frac{\rho}{n}}\xi^{\frac{d}{2}}$ for Truncated Fourier Series in $\mF_{\xi}(\Omega)$. Following the same proof as shown for DRM upper bound, the localization radius $\hat{r}$  can be determined as follows:
\begin{align*}
\sqrt{\frac{\rho}{n}}\xi^{\frac{d-2}{2}}+\sqrt{\frac{\rho}{N}}\xi^{\frac{d}{2}}+\frac{1}{n}+\frac{1}{N}+\xi^{-2(s-1)} \simeq \rho.
\end{align*}

For we have assumed $\frac{\xi^{d}}{N} <\frac{\xi^{d-2}}{n}$, the solution of the fixed point equation is $\hat{r} \simeq \frac{\xi^{d}}{n}$. On the other hand, by taking $\alpha = s$ and $\beta =1$ in Lemma \ref{lem: approximation Fourier} and applying strong convexity of the DRM objective function proved in Theorem \ref{thm: PDE regularity_drm}, we can upper bound the approximation error $\Delta \mE_{\text{app}}$ as below:
\begin{align*}
\Delta \mE_{\text{app}} \lesssim \xi^{-2(s-1)},    
\end{align*}
By equating the two terms above, we can solve for $\xi$ that yields the desired bound:
\begin{align*}
\frac{\xi^{d-2}}{n}+\xi^{-2(s-1)} \simeq \xi^{-2(s-1)} \Rightarrow \xi \simeq n^{\frac{1}{d+2s-4}}+\xi^{-2(s-1)},   
\end{align*}

Plugging in the expression of $\xi$ gives the final upper bound:
\begin{align*}
\mathbb{E}_{x \sim \mu}[\Delta \mE_{n}] \lesssim \hat{r} + \Delta \mE_{\text{app}} \lesssim \frac{\xi^{d-2}}{n} + \xi^{-2(s-2)} \simeq n^{-\frac{2s-2}{d+2s-4}}.    
\end{align*}

\end{proof}

\paragraph{Remark.} We still cannot achieve optimal rate for neural network even with modified DRM methods. The reason is because the number of neuron is not a good complexity measure for the gradient of the function thus the bound for $\psi(r) = R_n(\{\mathcal{I}(u)|\|u-u^\ast\|_{H^1}^2\le r\})$ is not enough for achieving optimal convergence rate. However, following \cite{schmidt2020nonparametric,suzuki2018adaptivity,imaizumi2020advantage,chen2019nonparametric,farrell2021deep} using deep networks for estimating functions, we optimize the best neural network with constrained sparsity in our paper. Here we conjecture that there exists a computable complexity measure can makes DRM statistical optimal and leave finding the right complexity of the neural network's gradient to be future work.

\section{Experiments}

In this section, we conduct several numerical experiments to verify our theory.  We follow the neural network and hyper-parameter setting in \cite{chen2020comprehensive}. Due to the page limit, we only put the experiments for Deep Ritz Methods here.

\subsection{The Modified Deep Ritz Methods}
\label{subsec:modifyexp}

In this section, we conduct experiments which substantiate our theoretical results for modified Deep Ritz methods. For simplicity, we take $V(x)=1$ in our experiment. We conduct experiment in 2-dimension and select the solution of the PDE as $u^\ast=\sum_z \|z\|^{-s}\phi_z(x)\in H^{s}$.  We show the log-log plot of the $H^1$ loss against the number of sampled data for $s=4$ in Figure \ref{fig:mdrmfourier}.  We use an OLS estimator to fit the log-log plot and put the estimated slope and corresponding $\text{R}^2$ score in Figure \ref{fig:mdrmfourier}. As our theory predicts, the modified Deep Ritz Method converges faster than the original one. All the derivation of the two estimators is listed in Appendix \ref{appendix:subopt}.

\begin{wrapfigure}{r}{0.5\textwidth}
    \centering
    \vspace{-20pt}
    \includegraphics[width=2.3in]{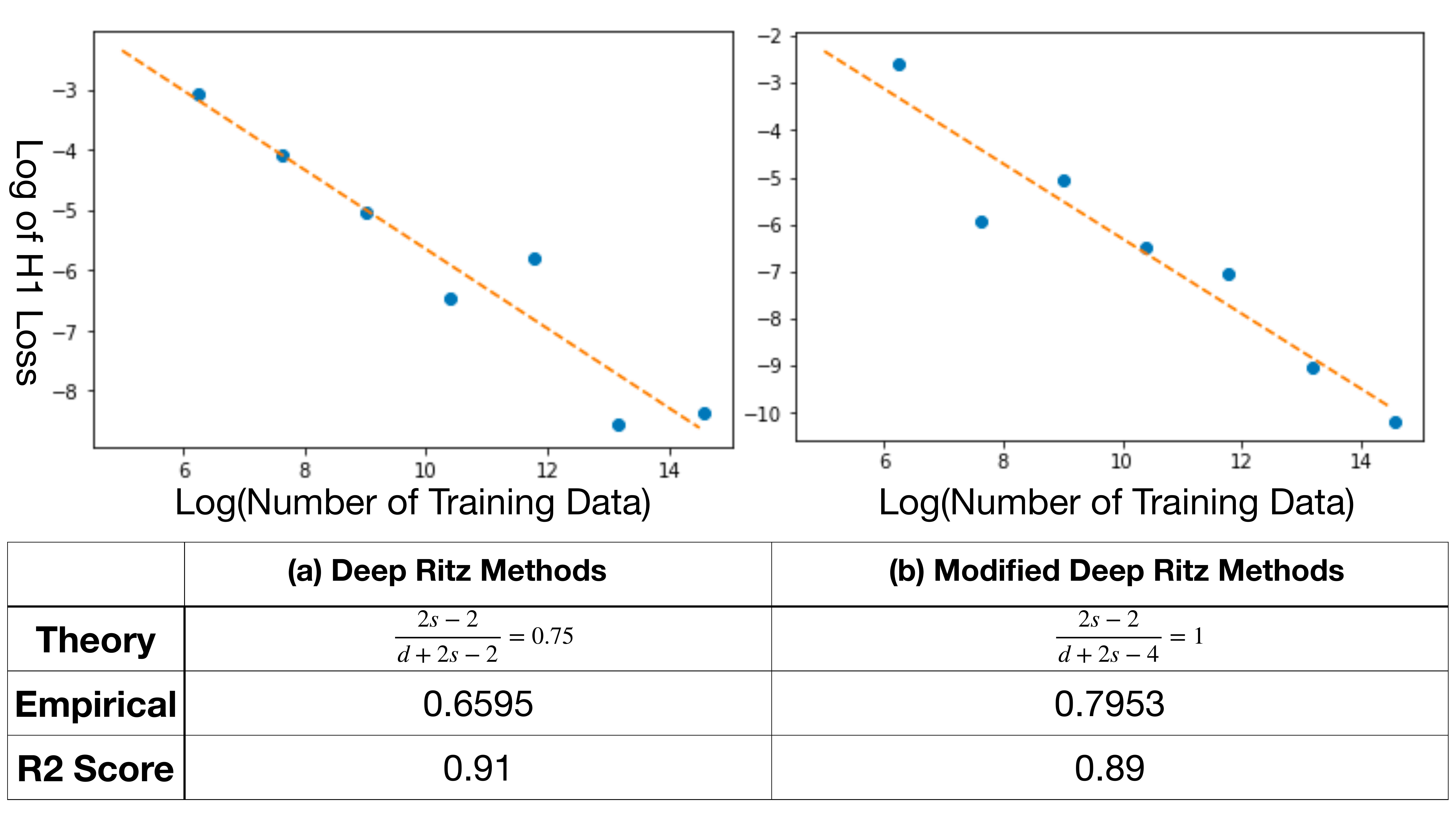}
    \vspace{-13pt}
    \caption{The Log-Log plot and estimated convergence slope for Modified DRM and DRM using Fourier basis, showing the median error over 5 replicates.}
    \label{fig:mdrmfourier}
\end{wrapfigure}

\subsection{Dimension Dependent Scaling Law.} 
We conduct experiments to illustrate that the population loss of well-trained and well-tuned Deep Ritz method will scale with the $d$-dimensional training data number $N$ as a power-law $\mathcal{L}\propto\frac{1}{N^\alpha}$. We also scan over a range of $d$ and $\alpha$ and verify an approximately $\alpha\propto\frac{1}{d}$ scaling law as our theory suggests. We use the same test function in Section \ref{subsec:modifyexp} as the solution of our PDE. For simplicity, we take $V(x)=1$ in our experiment. We train the deep Ritz method on 20, 80, 320, 1280, 10240 sampled data points for 5,6,7,8,9,10 dimensional problems and we plot our results on the log-log scale. Results are shown in Figure \ref{fig:my_label}.  We discover the $L\propto n^{\frac{1}{d+2}}$ scaling law in practical situations.

\begin{figure}[ht]
    \centering
    \includegraphics[width=5in]{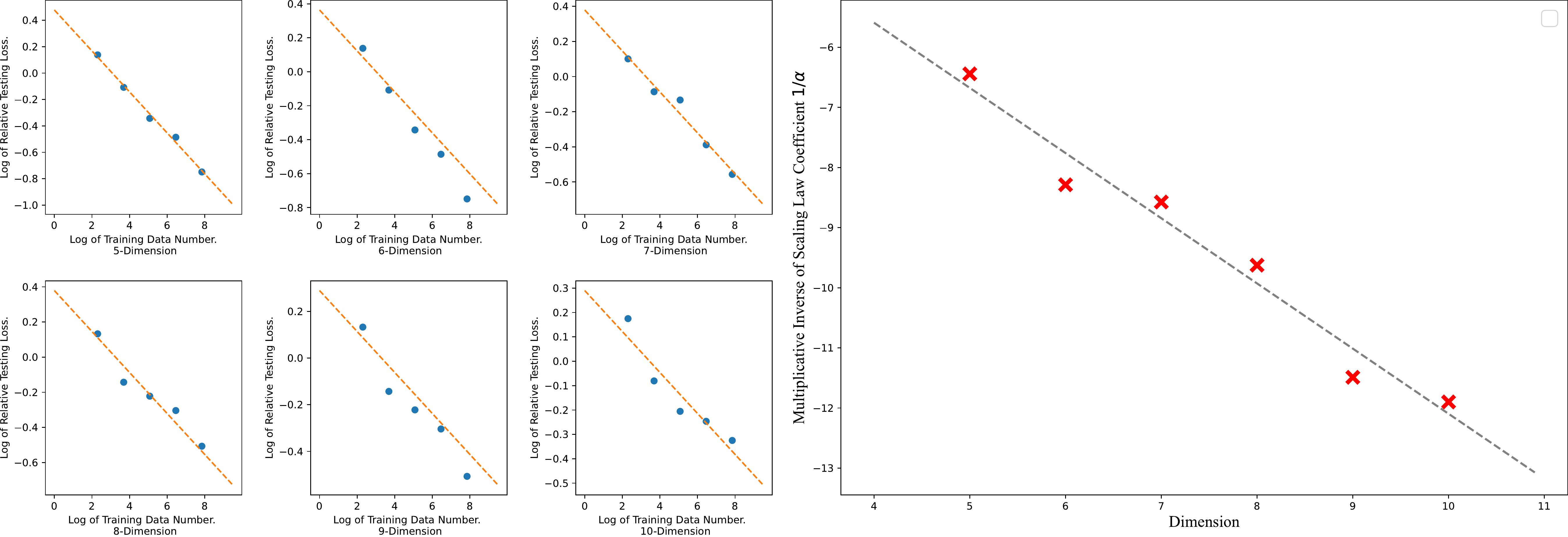}
    \caption{We verify the dimension dependent scaling law empirically. The multiplicative inverse of the scaling law coefficient is highly linear with the dimension $d$, showing the mean error over 2 replicates.}
    \label{fig:my_label}
\end{figure}

\subsection{Adaptation To The Simpler Functions.} \cite{sharma2020neural} showed that the neural scaling law will adapt to the structure that the target function enjoys. This adaptivity enables the neural network to break the cure of the dimensionality for simple functions in high dimension.  \cite{suzuki2019deep,chen2019efficient} also observed this theoretically. For solving PDEs, we also observed this adaptivity in practice. Here we tested the following two hypothesis

\begin{wrapfigure}{r}{0.5\textwidth}
    \centering
    \vspace{-20pt}
    \includegraphics[width=2.3in]{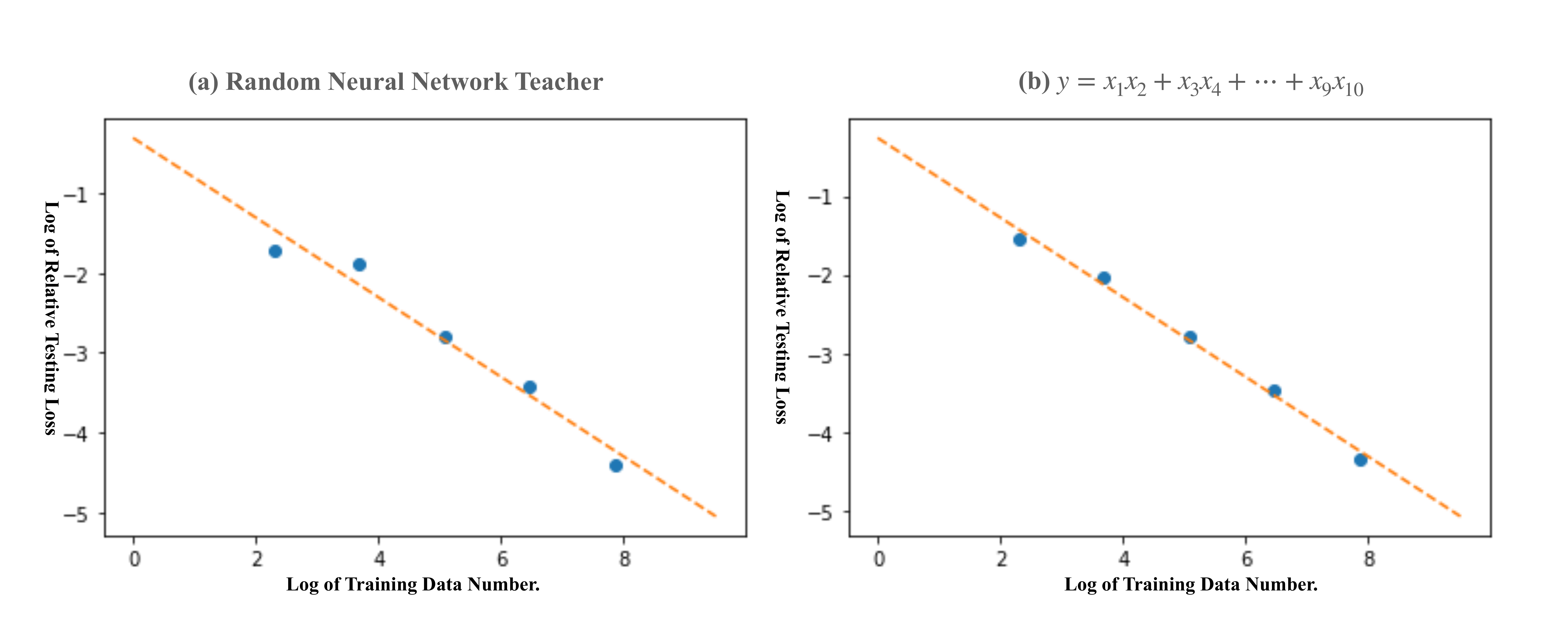}
    \vspace{-13pt}
    \caption{Neural network have the ability to adapt to simple functions and achieves convergence without curse of dimensionality, showing the median error over 5 replicates.}
    \label{fig:simplefunction}
\end{wrapfigure}

\begin{itemize}
    
    \item \textbf{Random Neural Network Teacher.} Following  \cite{sharma2020neural}, we also tested random neural network using He initialization \cite{he2015delving} as the ground turth solution $u^\ast$.  \cite{de2018random} showed that random deep neural networks are biased towards simple functions and in practice we observed a scaling law at the parametric rate. Specifically, we obtained a linear estimate with slope $\alpha=-0.50679429$ and a $\text{R}^2$ score $=0.96$  in the log-log plot. See Figure \ref{fig:simplefunction}(a).
    \item \textbf{Simple Polynomials.} Neural network can approximate simple polynomials exponentially fast \cite{wang2018exponential}. Thus, we select the ground truth solution to be the following simple polynomial in 10 dimensional spaces $u^\ast(x)=x_1x_2+\cdots+x_9x_{10}.$ In this example, we obtained a linear estimate with slope $\alpha=-0.49755418$ and a $\text{R}^2$ score $=0.99$ in the log-log plot. See Figure \ref{fig:simplefunction}(b).
\end{itemize}

\section{Conclusion and Discussion}

\paragraph{Conclusion}

In this paper, We considered the statistical min-max optimality of solving a PDE from random samples.  We improved the previous bounds \cite{xu2020finite,lu2021priori,duan2021convergence,jiao2021convergence} by providing the first fast rate generalization bound for learning PDE solutions via the strongly convex nature of the two objective functions. We achieved the optimal rate via the PINN and a modified Deep Ritz method. We verified our theory via numerical experiments and explored the dimension dependent scaling laws of Deep PDE solvers.

\paragraph{Discussion and Future Work} Here we discuss several drawbacks of our theory 

\begin{itemize}
    \item We restricted our target function and estimators in $W^{1,\infty}$ instead of $H^1$ due to boundedness assumption made in the local Rademacher complexity arguments. However, typical functional used in physics is always unbounded, such as the Newtonian potential $\frac{1}{\|x-y\|^{d-2}}$, which limits the application of our theory. 
    \item This paper did not discuss any optimization aspect of the deep PDE solvers and always assumed achieves global optimum. However, it is important to investigate whether the optimization error \cite{suzuki2020benefit,chizat2021convergence} will finally dominate. 
    \item  Instead of solving a single PDE,  recent works\cite{long2018pde,long2019pde,li2020fourier,lanthaler2021error,bhattacharya2020model,fan2020solving,feliu2020meta} considered the so-called "operator learning", which aims to learn a family of PDE/inverse problems using a single network. It is interesting to investigate the generalization bound and neural scaling law there.
    \item We find out that the sparsity of the neural network is not a good complexity measure of neural network's gradient. We conjecture that there exists an oracle complexity measure, whose approximation and generalization bounds can lead Modified DRM to achieve the optimal convergence rate.
\end{itemize}

\section*{Acknowledgments}
Yiping Lu is supported by the Stanford Interdisciplinary Graduate Fellowship (SIGF).  Jianfeng Lu is supported in part by National Science Foundation via grants DMS-2012286 and CCF1934964. Lexing Ying is supported by National Science Foundation under award DMS-2011699.  Jose Blanchet is supported in part by the Air Force Office of Scientific Research under award number FA9550-20-1-0397 and NSF grants 1915967, 1820942, 1838576. Yiping Lu also thanks Taiji Suzuki, Atsushi Nitanda, Yifan Chen, Junbin Huang, Wenlong Ji, Greg Yang, Yufan Chen, Zong Shang, Denny Wu, Jikai Hou, Jun Hu, Fang Yao and Bin Dong for helpful comments and feedback.

\bibliographystyle{siamplain}
\bibliography{references}

\newpage

\appendix

\section{Proof of the Upper Bounds}

\subsection{Notations}

In this section, we provide all the notations we need in the proof.  Let $\Omega \subset \mathbb{R}^{d}$ be some open set. We denote $C(\Omega)$ the space of continuous functions on $\Omega$ and
$C^{k}(\Omega)$  the space of all functions that are $k$ times continuously differentiable on $\Omega$ ($\forall k \in \mathbb{Z}^{+}$). For any $n \in \mathbb{N}_{0}$ ($\mathbb{N}_{0} := \mathbb{Z}^{+} \cup \{0\}$ is the set of all non-negative integers) and $1 \leq p \leq \infty$, we define the Sobolev space $W^{n,p}(\Omega)$ by
\begin{align*}
W^{n,p}(\Omega) := \{f \in L^{p}(\Omega): \ D^{\alpha}f \in L^{p}(\Omega), \ \forall \alpha \in \mathbb{N}_{0}^{d} \text{ with } |\alpha| \leq n\}.    
\end{align*}
In particular, when $p=2$, we define $H^{n}(\Omega) := W^{n,2}(\Omega)$ for any $n \in \mathbb{N}_{0}$. Moreover, for any $f \in W^{n,p}(\Omega)$ with $1 \leq p < \infty$, we define the Sobolev norm by:
\begin{align*}
\|f\|_{W^{n,p}(\Omega)} := \Big(\sum_{0 \leq |\alpha| \leq n}\|D^{\alpha}f\|^{p}_{L^p(\Omega)}\Big)^{\frac{1}{p}}.
\end{align*}
In particular, when $p = \infty$, we have:
\begin{align*}
\|f\|_{W^{n,\infty}(\Omega)} := \max_{0 \leq |\alpha| \leq n}\|D^{\alpha}f\|_{L^{\infty}(\Omega)}.
\end{align*}
Consider the Fourier expansion $f := \sum_{z \in \mathbb{N}^{d}}f_z \phi_{z}(x)$ of the function $f \in W^{n,p}(\Omega)$. We can equivalently express the Sobolev norm as:
$$
\|f\|_{W^{n,p}(\Omega)} =\Big(\sum_{z}\|z\|^{np}|f_z|^p\Big)^{1/p},
$$
where $f_z=\int_{\Omega}f(x)\overline{\phi_{z}(x)}dx = \int_{\Omega}f(x)e^{-2\pi i \langle z,x \rangle}dx \ (x \in \Omega)$ is the $z-$th Fourier coefficient of $f$.\\
Moreover, we use $W^{1,p}_{0}(\Omega)$ to denote the closure of $C_{c}^{1}(\Omega)$ in $W^{1,p}(\Omega)$. In particular, when $p=2$, we define $H_{0}^{1}(\Omega) := W_{0}^{1,2}(\Omega)$.\\
Furthermore, we use $\|\cdot\|$ to present the vector 2 norm and, given a data sample $\{X_{i}\}_{i=1}^{n} \subset \Omega$, $\|\cdot\|_{n,p}=\left({\bE_{n}\cdot^p}\right)^{1/p}$ denote the empirical $p$ norm, where $\bE_{n}: L^2(\Omega) \rightarrow \mathbb{R}$ is the corresponding empirical average operator defined as $\bE_{n}f := \frac{1}{n}\sum_{i=1}^{n}f(X_i), \ \forall \ f \in L^2(\Omega). $ Given two quantities $X$ and $Y$, we write $X \lesssim Y$ when the inequality $X \leq CY$ holds, where $C$ is some constant. For two functions $f$ and $g$ mapping from $\mathbb{R}^{+}$ to $\mathbb{R}$, we write $f = O(g)$ when there exist two constants $C'$ and $x_0$ independent of $f$ and $g$, such that the inequality $f(x) \leq C'g(x)$ holds for any $x \geq x_0$. We use $X \simeq Y$ to denote $X\lesssim Y$ and $Y\lesssim X$.

\subsection{Regularity Result For the PDE model.}
\label{appendix:regular}
\paragraph{Regularity Results of the DRM Objective Function}
\begin{theorem}
\label{thm: PDE regularity_drm}
We consider the static Schr\"odinger equation on the unit hypercube on $\R^d$ with the zero Direchlet boundary condition:

\begin{equation}
\label{eq:schrneumann_drm}
\begin{aligned}
   -\Delta u + V u & = f \text{ on } \Omega,\\ 
     u & =  0 \text{ on } \partial \Omega. 
\end{aligned}
\end{equation}
where $f\in L^2(\Omega)$ and $V\in L^\infty(\Omega)$ with $0< V_{\min} \leq V(x) \leq V_{\max}>0$. There exists a unique weak solution $u^{\ast}_{S}$ to the equivalent variational problem \cite{evans1998partial}:
   \begin{equation}\label{eq:variationalformbp_drm}
    u^{\ast}_{S} =  \argmin_{u\in H^1_0(\Omega)} \mE^{\text{DRM}}_S(u) :=  \argmin_{u\in H^1_0(\Omega)}  \Big\{\frac{1}{2} \int_{\Omega}\Big[\|\nabla u\|^2  +   V |u|^2\Big] \ dx- \int_{\Omega} f u dx \Big\}.
\end{equation}
Then for any $u \in H^{1}(\Omega)$, we have:
\begin{equation}
\label{eq:stronglyconvex_drm}
    \frac{\min(1,V_{\min})}{2} \|u-u^{\ast}_{S}\|^2_{H^1(\Omega)}\le \mE^{\text{DRM}}_S(u)-\mE^{\text{DRM}}_S(u^{\ast}_{S}) \\ \le \frac{\max(1,V_{\max})}{2}  \|u-u^{\ast}_{S}\|^2_{H^1(\Omega)}.    
\end{equation}
\end{theorem}

\begin{proof}
To show that $u^\ast_S$ satisfies estimate \ref{eq:stronglyconvex_drm}, we first claim that for any $u\in H^1(\Omega)$, 
\begin{equation}\label{eq:energyexcessS_drm}
\mE_S^{\text{DRM}} (u) - \mE_S^{\text{DRM}}(u^\ast_S) = 
 \frac{1}{2} \int_\Omega \|\nabla u - \nabla u^\ast_S\|^2 dx + \frac{1}{2} \int_{\Omega} V( u^\ast_S - u)^2\ dx.
\end{equation}
In fact, by plugging in the first equation of \ref{eq:schrneumann_drm}, one has that 
$$\begin{aligned}
    \mE_S^{\text{DRM}}(u^\ast_S) & = \frac{1}{2} \int_\Omega \|\nabla u^\ast_S\|^2 dx + \frac{1}{2} \int_\Omega V |u^\ast_S|^2 dx - \int_\Omega fu_{S}^\ast dx  \\
    & = \frac{1}{2} \int_\Omega \|\nabla u^\ast_S\|^2 dx + \frac{1}{2} \int_\Omega V |u^\ast_S|^2 dx  + \int_\Omega (\Delta u^\ast_S - V u^\ast_S) u_{S}^\ast dx  \\
    & =  \frac{1}{2} \int_\Omega \|\nabla u^\ast_S\|^2 dx+\int_\Omega (\Delta u^\ast_S)u_{S}^\ast dx  - \frac{1}{2} \int_\Omega V |u^\ast_S|^2 dx . 
\end{aligned}
$$
Furthermore, applying Green's formula to the true solution $u^{\ast}_{S}$ yields:
$$\begin{aligned}
    \mE_S^{\text{DRM}}(u^\ast_S) & = \frac{1}{2} \int_\Omega \|\nabla u^\ast_S\|^2 dx+\int_\Omega (\Delta u^\ast_S)u_{S}^\ast dx  - \frac{1}{2} \int_\Omega V |u^\ast_S|^2 dx  \\
    & = \int_{\partial \Omega}\frac{\partial u_{S}^{\ast}}{\partial n}u_{S}^{\ast} dx-\frac{1}{2} \int_\Omega \|\nabla u^\ast_S\|^2dx- \frac{1}{2} \int_\Omega V |u^\ast_S|^2 dx \\
    &= -\frac{1}{2} \int_\Omega \|\nabla u^\ast_S\|^2dx- \frac{1}{2} \int_\Omega V |u^\ast_S|^2 dx,
\end{aligned}
$$
where the last identity above follows from the second equality in \ref{eq:schrneumann_drm}. Now for any $u \in H^{1}(\Omega)$, applying Green's formula to $u$ and the true solution $u^{\ast}_{S}$ implies:
\begin{align*}
\mE_S^{\text{DRM}}(u) - \mE_S^{\text{DRM}}(u^\ast_S) 
& = \frac{1}{2} \int_{\Omega} \|\nabla u\|^2dx  + \frac{1}{2}\int_{\Omega} V |u|^2  dx- \int_{\Omega} f u dx + \frac{1}{2} \int_\Omega \|\nabla u^\ast_S\|^2dx  + \frac{1}{2}\int_{\Omega}V |u^\ast_S|^2 dx  \\
& = \frac{1}{2} \int_{\Omega} \|\nabla u\|^2dx  + \frac{1}{2}\int_{\Omega} V |u|^2  dx + \int_\Omega  (\Delta u^\ast_S - V u^\ast_S) u dx +  \frac{1}{2} \int_\Omega \|\nabla u^\ast_S\|^2dx  + \frac{1}{2}\int_{\Omega}V |u^\ast_S|^2 dx\\
& = \frac{1}{2} \int_{\Omega} \|\nabla u\|^2dx + \int_\Omega (\Delta u^\ast_S) u dx +\frac{1}{2} \int_\Omega \|\nabla u^\ast_S\|^2dx + \frac{1}{2} \int_{\Omega} V \big(u^\ast_S - u\big)^2 dx\\
& = \frac{1}{2} \int_{\Omega} \|\nabla u\|^2dx + \int_{\partial \Omega}\frac{\partial u_{S}^{\ast}}{\partial n}u dx -\int_{\Omega} \nabla u_{S}^{\ast} \cdot \nabla u dx + \frac{1}{2} \int_\Omega \|\nabla u^\ast_S\|^2dx + \frac{1}{2} \int_{\Omega} V \big(u^\ast_S - u\big)^2 dx\\
&=  \frac{1}{2} \int_\Omega \|\nabla u - \nabla u^\ast_S\|^2 dx + \frac{1}{2} \int_{\Omega} V( u^\ast_S - u)^2\ dx,
\end{align*}
where the last identity above again follows from the second equality in \ref{eq:schrneumann_drm}. This completes our proof of identity \ref{eq:energyexcessS_drm}. Using the assumptions on the potential function $V$ then implies:
\begin{align*}
\mE_S^{\text{DRM}}(u) - \mE_S^{\text{DRM}}(u^\ast_S) &\leq \frac{\max(1,V_{\max})}{2}\Big[\int_{\Omega}\|\nabla u - \nabla u^\ast_S\|^2 dx + \int_{\Omega}( u^\ast_S - u)^2\ dx \Big]\\&
= \frac{\max(1,V_{\max})}{2}  \|u-u^{\ast}_{S}\|^2_{H^1(\Omega)},\\
\mE_S^{\text{DRM}}(u) - \mE_S^{\text{DRM}}(u^\ast_S) &\geq \frac{\max(1,V_{\min})}{2}\Big[\int_{\Omega}\|\nabla u - \nabla u^\ast_S\|^2 dx + \int_{\Omega}( u^\ast_S - u)^2\ dx \Big] \\
&= \frac{\max(1,V_{\min})}{2}  \|u-u^{\ast}_{S}\|^2_{H^1(\Omega)}.
\end{align*}
This completes our proof of \ref{thm: PDE regularity_drm}.
\end{proof}

\paragraph{Regularity Results of the PINN Objective Function}
\begin{theorem}
\label{thm: PDE regularity_pinn}
We consider the static Schr\"odinger equation on the unit hypercube on $\R^d$ with the Neumann boundary condition:

\begin{equation}\label{eq:schrneumann_pinn}
\begin{aligned}
   -\Delta u + V u & = f \text{ on } \Omega,\\ 
     u & =  0 \text{ on } \partial \Omega. 
\end{aligned}
\end{equation}
where $f\in L^2(\Omega)$ and $V\in L^\infty(\Omega)$ with $V-\frac{1}{2}\Delta V> C_{\min}, 0<C_{\min}<V(x) \leq V_{\max}$ and $-\Delta V(x)\le V_{\max}$. Then there exists a unique solution $u^{\ast}_{S} \in H_0^1(\Omega)$ to the following minimization problem \cite{brezis2010functional}:
   \begin{equation}\label{eq:variationalformbp}
    u^{\ast}_{S} =  \argmin_{u\in H_0^1(\Omega)} \mE^{\text{PINN}}_S(u) :=  \argmin_{u\in H_0^1(\Omega)}  \Big\{\int_{\Omega} |\Delta u -Vu + f|^2 dx \Big\}.
\end{equation}
Then for any $u \in H^{1}_{0}(\Omega)$, we have:
\begin{equation}
\label{eq:stronglyconvex}
   \min\{1,C_{\min}\}\|u-u^{\ast}_{S}\|^2_{H^2(\Omega)}\le \mE^{\text{PINN}}_S(u)-\mE^{\text{PINN}}_S(u^{\ast}_{S}) \\ \le  2(1+V_{\max}+V_{\max}^2)\|u-u^{\ast}_{S}\|^2_{H^2(\Omega)}.
\end{equation}
\end{theorem}

\begin{proof}
For any $u \in H_0^{1}(\Omega)$, we let $\tilde{u}=u-u^{\ast}$, then we have $\tilde{u} \in H_0^1(\Omega)$. 
\begin{equation}
    \begin{aligned}
    \mE^{\text{PINN}}_S(u)-\mE^{\text{PINN}}_S(u^{\ast}_{S}) &=\int_{\Omega}|\Delta u -Vu -\Delta u^{\ast} + Vu^{\ast}|^2 dx = \int_{\Omega}|\Delta \tilde{u} - V\tilde{u}|^2dx\\
    &=\int_\Omega (\Delta \tilde{u})^2 dx +\int_{\Omega}V^2\tilde{u}^2dx-2\int_{\Omega}V\tilde{u}\Delta \tilde{u}dx.
    \end{aligned}
\end{equation}
Using Green's formula, we have:
\begin{align*}
\int_{\Omega}V\tilde{u}\Delta \tilde{u}dx + \int_{\Omega}\nabla(V\tilde{u}) \cdot \nabla \tilde{u} dx= \int_{\partial \Omega}\frac{\partial \tilde{u}}{\partial n}V\tilde{u}ds = 0,
\end{align*}
where the last equality above follows from the fact that $\tilde{u} \in H_0^1(\Omega)$. This further implies:
\begin{align*}
\mE^{\text{PINN}}_S(u)-\mE^{\text{PINN}}_S(u^{\ast}_{S}) &=\int_\Omega (\Delta \tilde{u})^2 dx+ \int_{\Omega}V^2\tilde{u}^2dx+2\int_{\Omega}\nabla(V\tilde{u}) \cdot \nabla \tilde{u} dx\\
&=\int_\Omega (\Delta \tilde{u})^2dx+\int_\Omega V^2\tilde{u}^2 dx+2\int_\Omega V\|\nabla \tilde{u}\|^2 dx+2\int_\Omega\tilde{u}\nabla V\cdot\nabla \tilde{u}dx.
\end{align*}
Using Green's formula again, we have:
\begin{align*}
2\int_\Omega\tilde{u}\nabla V\cdot\nabla \tilde{u}dx = \int_{\Omega}\nabla(u^2)\cdot \nabla V dx = \int_{\partial \Omega}\frac{\partial V}{\partial n}\tilde{u}^2ds - \int_{\Omega}\tilde{u}^2 \Delta V dx=- \int_{\Omega}\tilde{u}^2 \Delta V dx.        
\end{align*}
Then we can further deduce that:
\begin{align*}
\mE^{\text{PINN}}_S(u)-\mE^{\text{PINN}}_S(u^{\ast}_{S}) = \int_\Omega (\Delta \tilde{u})^2dx+\int_\Omega (V^2-\Delta V)\tilde{u}^2 dx+2\int_\Omega V\|\nabla \tilde{u}\|^2 dx.   
\end{align*}
For we have assumed $V\in L^\infty(\Omega)$ with $0<C_{\min}<V^2-\Delta V, 0<C_{\min}< V(x) \leq V_{\max}$ and $-\Delta V(x)\le V_{\max}$, thus we have
\begin{equation}
   \min\{1,C_{\min}\}\|u-u^{\ast}_{S}\|^2_{H^2(\Omega)}\le \mE^{\text{PINN}}_S(u)-\mE^{\text{PINN}}_S(u^{\ast}_{S}) \\ \le  2(1+V_{\max}+V_{\max}^2)\|u-u^{\ast}_{S}\|^2_{H^2(\Omega)}.
\end{equation}
\end{proof}

\subsection{Auxiliary definitions and lemmata On Generalization Error}
\label{appendix:generalization}

To bound the generalization error, we use the localized Rademacher complexity \cite{bartlett2005local}. Recall that the  Rademacher complexity of a function class $\mG$ is defined by 
$$
R_n (\mG) = \bE_{Z} \bE_{\sigma} \Big[\sup_{g\in \mG} \Big| \frac{1}{n}\sum_{j=1}^n \sigma_j g(Z_j) \Big| \; \Big| \; Z_1, \cdots, Z_n\Big],
$$

where $Z_i$ are i.i.d samples according to the data distributions and $\sigma_j$ are i.i.d Rademacher random variables which take the value $1$ with probability $\frac{1}{2}$ and value $-1$ with probability $\frac{1}{2}$.

The following important symmetrization lemma makes the connection between the uniform law of large numbers and the Rademacher complexity.

\begin{lemma}[Symmetrization Lemma]\label{lem:radcomp}
    Let $\mF$ be a set of functions. Then  $$
     \bE \sup_{u\in \mF} \Big|\frac{1}{n}  \sum_{j=1}^n u(X_j) - \bE_{X\sim \mP_\Omega} u(X)  \Big|
     \leq 2  R_n(\mF).
    $$
\end{lemma}

\begin{lemma}[{Ledoux-Talagrand  contraction \cite[Theorem 4.12]{ledoux2013probability}}]\label{lem:Talagrand contraction}
  Assume that $\phi:\R\gt\R$ is $L$-Lipschitz with $\phi(0)=0$. Let $\{\sigma_i\}_{i=1}^n$ be independent
Rademacher random variables. Then for any $T\subset \R^n$
$$
\bE_\sigma\Big[\sup_{(t_1,\cdots,t_n)\in T} \sum_{i=1}^n \sigma_i \phi(t_i)\Big] \leq 2L\cdot   \bE_\sigma \Big[\sup_{(t_1,\cdots,t_n)\in T} \sum_{i=1}^n \sigma_i t_i \Big].
$$
\end{lemma}

Let $(E,\rho)$ be a metric space with metric $\rho$. A {\em $\delta$-cover} of a set $A\subset E$ with respect to $\rho$ is  a collection of points $\{x_1,\cdots, x_n\}\subset A$ such that for every $x\in A$, there exists $i\in \{1,\cdots,n\}$ such that $\rho(x,x_i)\leq \delta$. The $\delta$-covering number 
 $\mN(\delta,A,\rho)$ is the cardinality of the smallest $\delta$-cover of the set $A$ with respect to the metric $\rho$. Equivalently, the $\delta$-covering number 
 $\mN(\delta,A,\rho)$   is the minimal number of balls $B_{\rho}(x,\delta)$ of radius $\delta$ needed to cover the set $A$.

\begin{theorem}[Dudley's Integral theorem] \label{thm:dudley} Let $\mF$ be a function class such that $\sup_{f\in \mF}\|f\|_{n,2}\leq M$. Then the Rademacher complexity $R_n(\mF)$ satisfies that
$$
R_n(\mF) \leq \inf_{0\leq \delta\leq M} \Big\{4\delta + \frac{12}{\sqrt{n}}\int_\delta^M \sqrt{\log \mN(\eps,\mF, \|\cdot\|_{n,2})} \,d\eps\Big\}.   
$$
\end{theorem} 

\begin{lemma}[Talagrand Concentration Inequality] \label{lem:Talagrand ineq} Consider a function class $\mathcal{F}$ defined on a probability measure $\mu$ such that for all $f\in\mathcal{F}$, we have $\|f\|_{\infty}\le\beta,\mathbb{E}_{\mu}[f]=0,\mathbb{E}_{\mu}[f^2]\le\sigma^2$. Then for any $t>0$, we can have the following concentration results.
$$
\mathbb{P}_{z_1,\cdots,z_n\sim \mu}\left[\sup_{f\in\mathcal{F}}\frac{1}{n}\sum_{i=1}^nf(z_i)\ge2\sup_{f\in\mathcal{F}}\mathbb{E}_{z_1',\cdots,z_n'\sim \mu}\frac{1}{n}\sum_{i=1}^nf(z_i')+\sqrt{\frac{2t\sigma^2}{n}}+\frac{2t\beta}{n}\right]\le e^{-t}.
$$
\end{lemma}
\begin{lemma}[Peeling lemma \cite{bartlett2005local}]
\label{lem: peeling} Consider some measurable function class $\mathcal{F}$. Assume that there exists a sub-root function $\phi(r)$ satisfying
\begin{equation}
    R_n(\{f\in\mathcal{F} \ | \ \bE[f] \le r\})\le \phi(r) \  (\forall \ r>0).
\end{equation}
Then we have
$$
\mathbb{E}_{\sigma_i,z_n}\left[\sup_{f\in\mathcal{F}}\frac{\frac{1}{n}\sum_{i=1}^n \sigma_if(z_i)}{\bE[f]+r}\right]\le \frac{4\phi(r)}{r}.
$$
\end{lemma}
\begin{proof}
Denote $\mathcal{F}(r) = \{f\in\mathcal{F} \ | \ \bE[f] \le r\}$ to be the localized set with radius $r$. Then for a fixed set of datapoints $\{z_{i}\}_{i=1}^{n}$ and a fixed set of Rademacher random variables $\{\sigma_{i}\}_{i=1}^{n}$, we have:
\begin{align*}
\mathbb{E}_{\sigma_i,z_n}\left[\sup_{f \in \mathcal{F}}\frac{\frac{1}{n}\sum_{i=1}^n\sigma_i f(z_i)}{\bE[ f]+r}\right] &\le \mathbb{E}_{\sigma_i,z_n}\left[\sup_{f\in\mathcal{F}(r)}\frac{\frac{1}{n}\sum_{i=1}^n \sigma_i f(z_i)}{r}\right]+\sum_{j=0}^\infty \mathbb{E}_{\sigma_i,z_n}\left[\sup_{f\in\mathcal{F}(r4^{j+1})\backslash \mathcal{F}(r4^{j})}\frac{\frac{1}{n}\sum_{i=1}^n \sigma_i f(z_i)}{r4^j+r}\right] \\
&\le \frac{R_n(\mathcal{F}(r))}{r}+\sum_{j=0}^\infty\frac{R_n(\mathcal{F}(r4^{j+1}))}{r4^j+r}\le \frac{\phi(r)}{r}+\sum_{j=0}^\infty\frac{\phi(r4^{j+1})}{r4^j+r}\\
&\le\frac{\phi(r)}{r}+\sum_{j=0}^\infty\frac{2^{j+1}\phi({r})}{r4^j+r} \le\frac{4\phi(r)}{r}.
\end{align*}
\end{proof}
We also modify the peeling lemma above, as we aim to apply it to derive the upper bound for the Modified Deep Ritz Method (MDRM).
\begin{lemma}[Peeling Lemma For MDRM]\label{lem:genpeel}
Given some measurable function class $\mathcal{F}$ and two continuous mappings $g,h:\mathcal{F} \rightarrow \mathbb{R}$, we define a set $\mathcal{F}$ of vector functions by:
    $$
    \mathcal{F} := \{(g \circ f, h \circ f) \ | \ f \in \mF\}.
    $$
For any $r >0$, the localized set $\mathcal{F}_{r}$ is defined by:
    $$
    \mathcal{F}_{r} = \{(g_f,h_f) \in \mathcal{F} \ | \ \bE_{x}[g_f(x)] + \bE_{y}[h_f(y)] \leq r\}.
    $$
Moreover, the modified Rademacher Complexity of $\mathcal{F}_{r}$ is defined by:
    $$
    R_{n,m}(\mathcal{F}_r):=R_n\Big(\{g_f|(g_f,h_f)\in\mathcal{F}_r\}\Big)+R_m(\{h_f|(g_f,h_f)\in\mathcal{F}_r\}\Big).
    $$
    Assume that there exists some function $\phi: [0,\infty) \rightarrow [0,\infty)$ and some $r^{\star} > 0$, such that for any $r > r^{\star}$, we have:
    $$
    \phi(4r) \leq 2\phi(r) \text{ and } R_{n,m}(\mathcal{F}_r) \leq \phi(r).
    $$
    Then for any $r > r^{\star}$, we have:
    $$
    \bE_{\sigma,\tau}\Big[\bE_{x,y}[\sup_{f \in \mF}\frac{\frac{1}{n}\sum_{i=1}^{n}\sigma_{i}g_f(x_i) + \frac{1}{m}\sum_{j=1}^{m}\tau_{j}h_f(y_j)}{\bE_{x}[g_f(x)] + \bE_{y}[h_f(y)] + r}]\Big] \leq \frac{4\phi(r)}{r}.
    $$
\end{lemma}

\begin{proof}
 The proof is the same as the original peeling lemma, thus we omit the detailed proof here.
\end{proof}

\subsubsection{Local Rademacher Complexity of Truncated Fourier Basis}
\begin{definition}{(Fourier Series)}
Given a domain $\Omega \subseteq [0,1]^{d}$. For any $z \in \mathbb{N}^{d}$, we consider the corresponding Fourier basis function $\phi_{z}(x):= e^{2\pi i \langle z,x \rangle} \ (x \in \Omega)$. With respect to the Fourier basis, any function $f \in L^2(\Omega)$ can be decomposed as the following sum:
\begin{equation}
f(x) := \sum_{z \in \mathbb{N}^{d}}f_z \phi_{z}(x).
\end{equation}
where for any $z \in \mathbb{N}^{d}$, the Fourier coefficient $f_z = \int_{\Omega}f(x)\overline{\phi_{z}(x)}dx$.
\end{definition}
\begin{definition}{(Truncated Fourier Series)} For a fixed positive integer $\xi \in \mathbb{Z}^{+}$, we define the space $F_{\xi}(\Omega)$ of truncated Fourier series as follows:
\begin{equation}
F_{\xi}(\Omega) := \Big\{f = \sum_{z \in \mathbb{N}^{d}}f_z \phi_{z} \ \Big| \ f_z = 0, \  \forall \ \|z\|_{\infty} > \xi \Big\}.    
\end{equation}
Equivalently, we can decompose any $f \in F_{\xi}(\Omega)$ as $f := \sum_{\|z\|_{\infty} \leq \xi}f_z\phi_z$.
\end{definition}

\begin{lemma}({Local Rademacher Complexity of Localized Truncated Fourier Series)}
For a fixed $\xi \in \mathbb{Z}^{+}$, we consider a localized class of functions $\mF_{\rho,\xi}(\Omega) = \Big\{f \in F_{\xi}(\Omega) \ \Big| \ \|f\|_{H^1(\Omega)}^2 \leq \rho \Big\}$, where $\rho >0$ is fixed. Then we have the following upper bound on the local Rademacher complexity:
\begin{equation}
R_{n}(\mF_{\rho,\xi}(\Omega)) = \bE_{X}\left[\bE_{\sigma}\Big[\sup_{f \in \mF_{\rho,\xi}(\Omega)}\frac{1}{n}\sum_{i=1}^{n}\sigma_{i}f(X_{i}) \ \Big| \ X_1, \cdots, X_n \Big]\right] \lesssim \sqrt{\frac{\rho}{n}}\xi^{\frac{d-2}{2}}.
\end{equation}
\end{lemma}
\begin{proof}
Take an arbitrary function $f \in \mF_{\rho,\xi}(\Omega)$. Let $f=\sum_{\|z\|_{\infty} \leq \xi}f_z\phi_z$ be the Fourier basis expansion of $f$. $\rho \geq \|f\|_{H^1(\Omega)}^2$ implies constraint $\sum_{\|z\|_{\infty} \leq \xi}|f_z|^2\|z\|^2  \lesssim \rho$ on the Fourier coefficients\cite{adams2003sobolev}.

On the other hand, substituting the Fourier expansion into the average sum $\frac{1}{n}\sum_{i=1}^{n}\sigma_{i}f(X_i)$ and using Cauchy-Schwarz inequality imply:
\begin{align*}
\frac{1}{n}\sum_{i=1}^{n}\sigma_{i}f(X_i) &= \frac{1}{n}\sum_{i=1}^{n}\sigma_{i}\sum_{\|z\|_{\infty} \leq \xi}f_z\phi_z(X_i) = \frac{1}{n}\sum_{\|z\|_{\infty} \leq \xi}\sum_{i=1}^{n}\sigma_i f_z \phi_z(X_i)\\
&\leq \frac{1}{n}\Big(\sum_{\|z\|_{\infty} \leq \xi}|f_z|^2\|z\|^2 \Big)^{\frac{1}{2}} \Big(\sum_{\|z\|_{\infty} \leq \xi} \Big|\sum_{i=1}^{n}\frac{\sigma_i}{\|z\|}\phi_z(X_i) \Big|^2 \Big)^{\frac{1}{2}}\\
&\lesssim \frac{\sqrt{\rho}}{n}\Big(\sum_{\|z\|_{\infty} \leq \xi} \Big|\sum_{i=1}^{n}\frac{\sigma_i}{\|z\|}\phi_z(X_i) \Big|^2 \Big)^{\frac{1}{2}}.
\end{align*}
where we have used the constraint $\sum_{\|z\|_{\infty} \leq \xi}|f_z|^2\|z\|^2  \lesssim \rho$ in the last step above. Moreover, by taking expectation with respect to the i.i.d Rademacher random variables $\sigma_{i} \ (1 \leq i \leq n)$ and the uniformly sampled data points $\{X_{i}\}_{i=1}^{n}$ on both sides and applying Jensen's inequality, we can deduce that:
\begin{align*}
\mathbb{E}_{X}\mathbb{E}_{\sigma}\Big[\frac{1}{n}\sum_{i=1}^{n}\sigma_{i}f(X_i) \Big] &\lesssim \frac{\sqrt{\rho}}{n}\mathbb{E}_{X,\sigma}\left[ \Big(\sum_{\|z\|_{\infty} \leq \xi}\Big|\sum_{i=1}^{n}\frac{\sigma_i}{\|z\|}\phi_z(X_i)\Big|^2\Big)^{\frac{1}{2}}\right]\\ &\leq \frac{\sqrt{\rho}}{n}\left(\mathbb{E}_{X,\sigma}\Big[\sum_{\|z\|_{\infty} \leq \xi}\Big|\sum_{i=1}^{n}\frac{\sigma_i}{\|z\|}\phi_z(X_i)\Big|^2\Big]\right)^{\frac{1}{2}}.  
\end{align*}
Using independence between the random variables $\sigma_{i} \ (1 \leq i \leq n)$, we can further simplify the expectation inside the square root above as below:
\begin{align*}
\mathbb{E}_{X,\sigma}\Big[\sum_{\|z\|_{\infty} \leq \xi}\Big|\sum_{i=1}^{n}\frac{\sigma_i}{\|z\|}\phi_z(X_i)\Big|^2\Big] &= \sum_{\|z\|_{\infty} \leq \xi}\mathbb{E}_{X,\sigma}\Big[\Big|\sum_{i=1}^{n}\frac{\sigma_i}{\|z\|}\phi_z(X_i)\Big|^2\Big]\\
&= \sum_{\|z\|_{\infty} \leq \xi}\sum_{i=1}^{n}\mathbb{E}_{X,\sigma}\Big[\frac{\sigma_{i}^2}{\|z\|^2}\Big|\phi_{z}(X_i)\Big|^2\Big]\\
&= \sum_{\|z\|_{\infty} \leq \xi}\sum_{i=1}^{n}\frac{|\Omega|}{\|z\|^2} \lesssim n\sum_{\|z\|_{\infty} \leq \xi}\frac{1}{\|z\|^2} \lesssim n\frac{\xi^{d}}{\xi^2} = n\xi^{d-2}.
\end{align*}
Combining the two bounds above yields the desired upper bound:
\begin{align*}
\bE_{X}\left[\bE_{\sigma}\Big[\sup_{f \in \mF_{\rho,\xi}(\Omega)}\frac{1}{n}\sum_{i=1}^{n}\sigma_{i}f(X_{i}) \ \Big| \ X_1, \cdots, X_n \Big]\right] \lesssim \frac{\sqrt{\rho}}{n}\sqrt{n\xi^{d-2}} = \sqrt{\frac{\rho}{n}}\xi^{\frac{d-2}{2}}.    
\end{align*}
\end{proof}
\begin{lemma}{(Local Rademacher Complexity of Localized Truncated Fourier Series' Gradient)}
For a fixed $\xi \in \mathbb{Z}^{+}$, we consider a localized class of functions $\mG_{\rho,\xi}(\Omega) = \{\|\nabla f\| \ | \ f \in F_{\rho,\xi}(\Omega)\}$, where $\rho >0$ is fixed. Then for any sample $\{X_{i}\}_{i=1}^{n} \subset \Omega$, we have the following upper bound on the local Rademacher complexity:
\begin{equation}
R_{n}(\mG_{\rho,\xi}(\Omega)) = \bE_{X}\left[\bE_{\sigma}\Big[\sup_{f \in \mF_{\rho,\xi}(\Omega)}\frac{1}{n}\sum_{i=1}^{n}\sigma_{i}\|\nabla f(X_{i})\| \ \Big| \ X_1, \cdots, X_n \Big]\right] \lesssim \sqrt{\frac{\rho}{n}}\xi^{\frac{d}{2}}.
\end{equation}
\end{lemma}
\begin{proof}

Take an arbitrary function $f \in \mF_{\rho,\xi}(\Omega)$. Let $f=\sum_{\|z\|_{\infty} \leq \xi}f_z\phi_z$ be the Fourier basis expansion of $f$. Similarly, the norm restriction condition $\|f\|_{H^1(\Omega)}^2 \le\rho$ can be reduced to the following condition about Fourier coefficients:
\begin{align*}
\sum_{\|z\|_{\infty} \leq \xi}|f_z|^2\|z\|^2  \lesssim \rho.
\end{align*}
Moreover, substituting the Fourier expansion into the average sum $\frac{1}{n}\sum_{i=1}^{n}\sigma_{i}\|\nabla f(X_i)\|$ and using Cauchy-Schwarz inequality imply:
\begin{align*}
\frac{1}{n}\sum_{i=1}^{n}\sigma_{i}\|\nabla f(X_i)\| &= \frac{1}{n}\sum_{i=1}^{n}\sigma_{i}\|\sum_{\|z\|_{\infty} \leq \xi}f_z \nabla \phi_z(X_i)\| \leq \frac{1}{n}\sum_{\|z\|_{\infty} \leq \xi}\sum_{i=1}^{n}\sigma_i\|f_z \nabla \phi_z(X_i)\|\\
&\leq \frac{1}{n}\Big(\sum_{\|z\|_{\infty} \leq \xi}|f_z|^2\|z\|^2 \Big)^{\frac{1}{2}} \Big(\sum_{\|z\|_{\infty} \leq \xi} \Big|\sum_{i=1}^{n}\frac{\sigma_i}{\|z\|}\|\nabla \phi_z(X_i)\| \Big|^2 \Big)^{\frac{1}{2}}\\
&\lesssim \frac{\sqrt{\rho}}{n}\Big(\sum_{\|z\|_{\infty} \leq \xi} \Big|\sum_{i=1}^{n}\frac{\sigma_i}{\|z\|}\|\nabla \phi_z(X_i)\| \Big|^2 \Big)^{\frac{1}{2}}.
\end{align*}
where we have used the constraint $\sum_{\|z\|_{\infty} \leq \xi}|f_z|^2\|z\|^2  \lesssim \rho$ in the last step above. Moreover, by taking expectation with respect to the i.i.d Rademacher random variables $\sigma_{i} \ (1 \leq i \leq n)$ and the uniformly sampled data points $\{X_{i}\}_{i=1}^{n}$ on both sides and applying Jensen's inequality, we can deduce that:
\begin{align*}
\mathbb{E}_{X}\mathbb{E}_{\sigma}\Big[\frac{1}{n}\sum_{i=1}^{n}\sigma_{i}\|\nabla f(X_i)\| \Big] &\lesssim \frac{\sqrt{\rho}}{n}\mathbb{E}_{X,\sigma}\left[ \Big(\sum_{\|z\|_{\infty} \leq \xi}\Big|\sum_{i=1}^{n}\frac{\sigma_i}{\|z\|}\|\nabla \phi_z(X_i)\|\Big|^2\Big)^{\frac{1}{2}}\right]\\ &\leq \frac{\sqrt{\rho}}{n}\left(\mathbb{E}_{X,\sigma}\Big[\sum_{\|z\|_{\infty} \leq \xi}\Big|\sum_{i=1}^{n}\frac{\sigma_i}{\|z\|}\|\nabla \phi_z(X_i)\|\Big|^2\Big]\right)^{\frac{1}{2}}.
\end{align*}
Using independence between the random variables $\sigma_{i} \ (1 \leq i \leq n)$, we can further simplify the expectation inside the square root above as below:
\begin{align*}
\mathbb{E}_{X,\sigma}\Big[\sum_{\|z\|_{\infty} \leq \xi}\Big|\sum_{i=1}^{n}\frac{\sigma_i}{\|z\|}\|\nabla \phi_z(X_i)\|\Big|^2\Big] &= \sum_{\|z\|_{\infty} \leq \xi}\mathbb{E}_{X,\sigma}\Big[\Big|\sum_{i=1}^{n}\frac{\sigma_i}{\|z\|}\|\nabla \phi_z(X_i)\|\Big|^2\Big]\\
&= \sum_{\|z\|_{\infty} \leq \xi}\sum_{i=1}^{n}\mathbb{E}_{X,\sigma}\Big[\frac{\sigma_{i}^2}{\|z\|^2}\|\nabla \phi_z(X_i)\|^2\Big]\\
&= \sum_{\|z\|_{\infty} \leq \xi}\sum_{i=1}^{n}|\Omega|\frac{4\pi^2\|z\|^2}{\|z\|^2} \lesssim n\sum_{\|z\|_{\infty} \leq \xi}1 \lesssim n\xi^{d}.
\end{align*}
Combining the two bounds above yields the desired upper bound:
\begin{align*}
\bE_{X}\left[\bE_{\sigma}\Big[\sup_{f \in \mF_{\rho,\xi}(\Omega)}\frac{1}{n}\sum_{i=1}^{n}\sigma_{i}\|\nabla f(X_{i})\| \ \Big| \ X_1, \cdots, X_n \Big]\right] \lesssim \frac{\sqrt{\rho}}{n}\sqrt{n\xi^{d}} = \sqrt{\frac{\rho}{n}}\xi^{\frac{d}{2}}.    
\end{align*} 
\end{proof}

\begin{lemma}{(Local Rademacher Complexity of Localized Truncated Fourier Series' Laplacian)}
For a fixed $\xi \in \mathbb{Z}^{+}$, we consider a localized class of functions $\mJ_{\rho,\xi}(\Omega) := \Big\{f \in F_{\xi}(\Omega) \ \Big| \ \|f\|_{H^2(\Omega)}^2 \leq \rho \Big\}$, where $\rho >0$ is fixed. Correspondingly, we define a localized class of Laplacians $\mK_{\rho,\xi}(\Omega) := \{\Delta f \ | \ f \in J_{\rho,\xi}(\Omega)\}$. Then for any sample $\{X_{i}\}_{i=1}^{n} \subset \Omega$, we have the following upper bound on the local Rademacher complexity:
\begin{equation}
R_{n}(\mK_{\rho,\xi}(\Omega)) = \bE_{X}\left[\bE_{\sigma}\Big[\sup_{f \in \mF_{\rho,\xi}(\Omega)}\frac{1}{n}\sum_{i=1}^{n}\sigma_{i}\Delta f(X_{i}) \ \Big| \ X_1, \cdots, X_n \Big]\right] \lesssim \sqrt{\frac{\rho}{n}}\xi^{\frac{d}{2}}.
\end{equation}
\end{lemma}
\begin{proof}

Take an arbitrary function $f \in \mJ_{\rho,\xi}(\Omega)$. Let $f=\sum_{\|z\|_{\infty} \leq \xi}f_z\phi_z$ be the Fourier basis expansion of $f$. Similarly, the norm restriction condition $\|f\|_{H^2(\Omega)}^2 \le\rho$ can be reduced to the following condition about fourier coefficients:
\begin{align*}
\sum_{\|z\|_{\infty} \leq \xi}|f_z|^2\|z\|^4  \lesssim \rho.
\end{align*}
Moreover, substituting the Fourier expansion into the average sum $\frac{1}{n}\sum_{i=1}^{n}\sigma_{i}\Delta f(X_i)$ and using Cauchy-Schwarz inequality imply:
\begin{align*}
\frac{1}{n}\sum_{i=1}^{n}\sigma_{i}\Delta f(X_i) &= \frac{1}{n}\sum_{i=1}^{n}\sigma_{i}\sum_{\|z\|_{\infty} \leq \xi}f_z\Delta \phi_z(X_i) = \frac{1}{n}\sum_{\|z\|_{\infty} \leq \xi}\sum_{i=1}^{n}\sigma_i f_z \Delta \phi_z(X_i)\\
&\leq \frac{1}{n}\Big(\sum_{\|z\|_{\infty} \leq \xi}|f_z|^2\|z\|^4 \Big)^{\frac{1}{2}} \Big(\sum_{\|z\|_{\infty} \leq \xi} \Big|\sum_{i=1}^{n}\frac{\sigma_i}{\|z\|^2}\Delta \phi_z(X_i) \Big|^2 \Big)^{\frac{1}{2}}\\
&\lesssim \frac{\sqrt{\rho}}{n}\Big(\sum_{\|z\|_{\infty} \leq \xi} \Big|\sum_{i=1}^{n}\frac{\sigma_i}{\|z\|^2}\Delta \phi_z(X_i) \Big|^2 \Big)^{\frac{1}{2}}.
\end{align*}
where we have used the constraint $\sum_{\|z\|_{\infty} \leq \xi}|f_z|^2\|z\|^4  \lesssim \rho$ in the last step above. Moreover, by taking expectation with respect to the i.i.d Rademacher random variables $\sigma_{i} \ (1 \leq i \leq n)$ and the uniformly sampled data points $\{X_{i}\}_{i=1}^{n}$ on both sides and applying Jensen's inequality, we can deduce that:
\begin{align*}
\mathbb{E}_{X}\mathbb{E}_{\sigma}\Big[\frac{1}{n}\sum_{i=1}^{n}\sigma_{i}\Delta f(X_i) \Big] &\lesssim \frac{\sqrt{\rho}}{n}\mathbb{E}_{X,\sigma}\left[ \Big(\sum_{\|z\|_{\infty} \leq \xi}\Big|\sum_{i=1}^{n}\frac{\sigma_i}{\|z\|^2}\Delta \phi_z(X_i)\Big|^2\Big)^{\frac{1}{2}}\right]\\ 
&\leq \frac{\sqrt{\rho}}{n}\left(\mathbb{E}_{X,\sigma}\Big[\sum_{\|z\|_{\infty} \leq \xi}\Big|\sum_{i=1}^{n}\frac{\sigma_i}{\|z\|^2}\Delta \phi_z(X_i)\Big|^2\Big]\right)^{\frac{1}{2}}.
\end{align*}
Using independence between the random variables $\sigma_{i} \ (1 \leq i \leq n)$, we can further simplify the expectation inside the square root above as below:
\begin{align*}
\mathbb{E}_{X,\sigma}\Big[\sum_{\|z\|_{\infty} \leq \xi}\Big|\sum_{i=1}^{n}\frac{\sigma_i}{\|z\|^2}\Delta \phi_z(X_i)\Big|^2\Big] &= \sum_{\|z\|_{\infty} \leq \xi}\mathbb{E}_{X,\sigma}\Big[\Big|\sum_{i=1}^{n}\frac{\sigma_i}{\|z\|^2}\Delta \phi_z(X_i)\Big|^2\Big]\\
&= \sum_{\|z\|_{\infty} \leq \xi}\sum_{i=1}^{n}\mathbb{E}_{X,\sigma}\Big[\frac{\sigma_{i}^2}{\|z\|^4}|\Delta \phi_z(X_i)|^2\Big]\\
&= \sum_{\|z\|_{\infty} \leq \xi}\sum_{i=1}^{n}|\Omega|\frac{16\pi^4\|z\|^4}{\|z\|^4} \lesssim n\sum_{\|z\|_{\infty} \leq \xi}1 \lesssim n\xi^{d}.
\end{align*}
Combining the two bounds above yields the desired upper bound:
\begin{align*}
\bE_{X}\left[\bE_{\sigma}\Big[\sup_{f \in \mF_{\rho,\xi}(\Omega)}\frac{1}{n}\sum_{i=1}^{n}\sigma_{i}\Delta f(X_{i}) \ \Big| \ X_1, \cdots, X_n \Big]\right] \lesssim \frac{\sqrt{\rho}}{n}\sqrt{n\xi^{d}} = \sqrt{\frac{\rho}{n}}\xi^{\frac{d}{2}}.    
\end{align*} 
\end{proof}

\subsubsection{Local Rademacher Complexity of the Deep Neural Network Model}
In this section we aim to bound the local Rademacher Complexity of a Deep Neural Network. We first bound the covering number of the function space composed by the gradient of all possible neural networks and then apply a Duley Integral to achieve the final bound.
\begin{definition}
Let $\eta_{l}$ denote the $l$-ReLU activiation function. Here we use $\eta_{3}:=\max\{0,x\}^3$\cite{weinan2018deep} as the activation function to ensure smoothness. We can define the space consisting of all neural network models with depth $L$, width $W$, sparsity constraint $S$ and norm constraint $B$ as follows:
\begin{align}
& \Phi(L,W,S,B) 
:= \Big\{(\mathcal{W}^{(L)}\eta_3(\cdot) + b^{(L)}) \cdots (\mathcal{W}^{(1)}x + b^{(1)}) \ | \ \mathcal{W}^{(L)} \in \mathbb{R}^{1 \times W}, b^{(L)} \in \mathbb{R}, \\
&\mathcal{W}^{(1)} \in \mathbb{R}^{W \times d}, b^{(1)} \in \mathbb{R}^{W}, \mathcal{W}^{(l)} \in \mathbb{R}^{W \times W},b^{(l)} \in \mathbb{R}^{W} (1 <l<L),\\
&\sum_{l=1}^{L}(\|\mathcal{W}^{(l)}\|_0 + \|b^{(l)}\|_0) \leq S, \max_{l}\|\mathcal{W}^{(l)}\|_{\infty,\infty} \vee \|b^{(l)}\|_{\infty} \leq B\Big\}.
\end{align}
where $\|\cdot\|_{0}$ measures the number of nonzero entries in a matrix and $\|\cdot\|_{\infty,\infty}$ measures the maximum of the absolute values of the entries in a matrix.\\
For any $d \in \mathbb{Z}^{+}$, we refer to an arbitrary element in $\Phi(L,W,S,B)$ as a ReLU3 Deep Neural Network. Then for any index $1 \leq k \leq L$, we use $F_k$ to denote the $k-$ReLU3 Deep Neural Network composed by the first $k$ layers, i.e:
\begin{align*}
F_k(x) := (\mathcal{W}_{F}^{(k)}\eta_3(\cdot) + b_{F}^{(k)}) \cdots (\mathcal{W}_{F}^{(1)}x + b_{F}^{(1)}).     
\end{align*}
Also, we use $\Phi_{k}(L,W,S,B)$ to denote the space consisting of all $F_{k}$. In particular, when $k=L$, we have:
\begin{align*}
F(x) := F_L(x) = (\mathcal{W}_{F}^{(L)}\eta_3(\cdot) + b_{F}^{(L)}) \cdots (\mathcal{W}_{F}^{(1)}x + b_{F}^{(1)}), \text{ and } \Phi_{L}(L,W,S,B) = \Phi(L,W,S,B).         
\end{align*}
Furthermore, given that the domain $\Omega \subset [0,1]^{d}$ is bounded, we have $\sup_{x \in \Omega}\|x\|_{\infty} = 1$.
\end{definition}
\begin{lemma}{(Upper bound on $\infty$-norm of functions in DNN space)}
\label{dnn_inf_norm}
For any $1 \leq k \leq L$, the following inequality holds:
\begin{align*}
\sup_{x \in \Omega, \ F_k \in \Phi_{k}(L,W,S,B)}\|F_k(x)\|_{\infty} \leq W^{\frac{3^{k-1}-1}{2}}(B \vee d)^{\frac{5 \cdot 3^{k-1}-1}{2}}2^{\frac{3^k-1}{2}-k+1}.
\end{align*}
\end{lemma}
\begin{proof}
We use induction to prove this claim.\\
Base cases: When $k=1$, we have that for any $x \in \Omega$ and any $F_1 \in \Phi_{1}(L,W,S,B)$, the following holds:
\begin{equation}
\label{inf_bound_base}
\begin{aligned}
\|F_1(x)\|_{\infty} &= \|\mathcal{W}_{F}^{(1)}x+b_{F}^{(1)}\|_{\infty} \leq \|\mathcal{W}_{F}^{(1)}\|_{\infty}\|x\|_{\infty} + \|b_{F}^{(1)}\|_{\infty}\\
&\leq d\|\mathcal{W}_{F}^{(1)}\|_{\infty,\infty} + B \leq dB+B \leq 2(B \vee d)^2.   
\end{aligned}
\end{equation}
When $k =2$, we have that for any $x \in \Omega$ and any $F_2 \in \Phi_{2}(L,W,S,B)$, the following holds:
\begin{align*}
\|F_2(x)\|_{\infty} &= \|\mathcal{W}_{F}^{(2)}\eta_3(F_1(x))+b_{F}^{(2)}\|_{\infty} \leq \|\mathcal{W}_{F}^{(2)}\|_{\infty}\|\eta_3(F_1(x))\|_{\infty} + \|b_{F}^{(2)}\|_{\infty} \leq W\|\mathcal{W}_{F}^{(2)}\|_{\infty,\infty}\|F_1(x)\|_{\infty}^3 + B.
\end{align*}
By applying the bound proved in the case when $k=1$, we have:
\begin{align*}
\|F_2(x)\|_{\infty} &\leq WB(dB+B)^3 + B = WB^4(d+1)^3 + B\\
&=WB^4(d^3+3d^2+3d+1) + B \leq 8W(B \vee d)^7.
\end{align*}
where the last inequality follows from the assumption that $W \geq 2$.\\
Inductive Step: Now we assume that the claim has been proved for $k-1$, where $3 \leq k \leq L$. Similarly, for any $x \in \Omega$ and any $F_k \in \Phi_{k}(L,W,S,B)$, we have:
\begin{align*}
\|F_k(x)\|_{\infty} &= \|\mathcal{W}_{F}^{(k)}\eta_{3}(F_{k-1}(x)) + b_{F}^{(k)}\|_{\infty} \leq \|\mathcal{W}_{F}^{(k)}\|_{\infty} \|\eta_{3}(F_{k-1}(x))\|_{\infty} + \|b_{F}^{(k)}\|_{\infty}\\
&\leq W\|\mathcal{W}_{F}^{(k)}\|_{\infty,\infty}\|F_{k-1}(x)\|_{\infty}^3 + B \leq WB\|F_{k-1}(x)\|_{\infty}^3 + B.
\end{align*}
Using inductive hypothesis, we can further deduce that:
\begin{align*}
\|F_k(x)\|_{\infty} &\leq WB \times W^{\frac{3^{k-1}-3}{2}}(B \vee d)^{\frac{5 \cdot 3^{k-1}-3}{2}}2^{\frac{3^k-3}{2}-3k+6} + B \\
&\leq W^{\frac{3^{k-1}-1}{2}}(B \vee d)^{\frac{5 \cdot 3^{k-1}-1}{2}}2^{\frac{3^k-3}{2}-3k+6} + B \vee d \\
&\leq W^{\frac{3^{k-1}-1}{2}}(B \vee d)^{\frac{5 \cdot 3^{k-1}-1}{2}}[2^{\frac{3^k-3}{2}-3k+6} + 1]\\
&\leq W^{\frac{3^{k-1}-1}{2}}(B \vee d)^{\frac{5 \cdot 3^{k-1}-1}{2}}2^{\frac{3^k-3}{2}-k+2} \ (k \geq 3)\\
&= W^{\frac{3^{k-1}-1}{2}}(B \vee d)^{\frac{5 \cdot 3^{k-1}-1}{2}}2^{\frac{3^k-1}{2}-k+1}.
\end{align*}
Taking supremum with respect to $x \in \Omega$ and $F_k \in \Phi_{k}(L,W,S,B)$ on the LHS implies that the given upper bound also holds for $k$. By induction, the claim is proved.
\end{proof}
We also need to show that the ReLU3 activation function is a Lipschitzness functions over a bounded domain.
\begin{lemma}
\label{Lipschitz-ReLU}
For any $k \in \mathbb{Z}^{+}$, consider the $k-$ReLU activation function $\eta_{k}$ defined on some bounded domain $\mathcal{D} \subset \mathbb{R}^{d}$ (i.e, $\sup_{x \in \mathcal{D}}\|x\|_{\infty} \leq C$ for some $C > 0$). Then we have that for any $x,y \in \mathcal{D}$, the following inequalities hold:
\begin{align*}
\|\eta_1(x) - \eta_1(y)\|_{\infty} &\leq \|x-y\|_{\infty},\\
\|\eta_2(x) - \eta_2(y)\|_{\infty} &\leq 2C\|x-y\|_{\infty},\\
\|\eta_3(x) - \eta_3(y)\|_{\infty} &\leq 3C^2\|x-y\|_{\infty}.
\end{align*}
\begin{proof}
This is because  $|\nabla \eta_1(x)|=|\max\{1,0\}| = 1$,  $|\nabla \eta_2(x)|=|2\max\{x,0\}|\le 2C$ and $|\nabla \eta_3(x)|=|3\max\{x,0\}^2|\le 3C^2$. 
\end{proof}
\end{lemma}
\begin{lemma}{(Relation between the covering number of DNN space and parameter space)}
\label{dnn_diff_bound}
For any $1 \leq k \leq L$, suppose that a pair of different two networks $F_k,G_k \in \Phi_{k}(L,W,S,B)$ are given by:
\begin{align*}
F_k(x) &:= (\mathcal{W}_{F}^{(k)}\eta_3(\cdot) + b_{F}^{(k)}) \cdots (\mathcal{W}_{F}^{(1)}x + b_{F}^{(1)}),\\
G_k(x) &:= (\mathcal{W}_{G}^{(k)}\eta_3(\cdot) + b_{G}^{(k)}) \cdots (\mathcal{W}_{G}^{(1)}x + b_{G}^{(1)}).
\end{align*}
Furthermore, assume that the $\| \ \|_{\infty}$ norm of the distance between the parameter spaces of $F_k$ and $G_k$ is uniformly upper bounded by $\delta$, i.e
\begin{equation}
\label{dnn_para_constraint}
\|W_{F}^{(l)} - W_{G}^{(l)}\|_{\infty,\infty} \leq \delta, \ \|b_{F}^{(l)} - b_{G}^{(l)}\|_{\infty} \leq \delta, \ (\forall \ 1 \leq l \leq k).
\end{equation}
Then we have:
\begin{align}
\sup_{x \in \Omega}\|F_k(x)-G_k(x)\|_{\infty} \leq \delta W^{\frac{3^{k-1}-1}{2}}(B \vee d)^{\frac{5 \cdot 3^{k-1}-1}{2}}2^{\frac{3^k-1}{2}-k+1}3^{k-1}.
\end{align}
\end{lemma}
\begin{proof}
Let's prove the claim by using induction on $k$.\\
Base Case: When $k=1$, we have that for any $x \in \Omega$ and any $F_1, G_1 \in \Phi_{1}(L,W,S,B)$ satisfying constraint \ref{dnn_para_constraint}, the following holds:
\begin{equation}
\label{base_case_covering}
\begin{aligned}
\|F_1(x) - G_1(x)\|_{\infty} &= \|\mathcal{W}_F^{(1)}x + b_F^{(1)} - \mathcal{W}_G^{(1)}x - b_G^{(1)}\|_{\infty} \\
&\leq \|\mathcal{W}_F^{(1)} - \mathcal{W}_G^{(1)}\|_{\infty}\|x\|_{\infty} + \|b_F^{(1)} - b_G^{(1)}\|_{\infty}\\
&\leq \delta d + \delta = \delta(d+1) \leq 2\delta (B \vee d) \leq 2\delta (B \vee d)^2.
\end{aligned}
\end{equation}
When $k=2$, we have that for any $x \in \Omega$ and any $F_2,G_2 \in \Phi_{2}(L,W,S,B)$ satisfying constraint \ref{dnn_para_constraint}, the following inequality holds:
\begin{align*}
\|F_2(x) - G_2(x)\|_{\infty} &= \|\mathcal{W}_F^{(2)}\eta_3(F_{1}(x)) + b_{F}^{(2)} -\mathcal{W}_G^{(2)}\eta_3(G_{1}(x)) - b_{G}^{(2)}\|_{\infty} \\
&\leq \|\mathcal{W}_F^{(2)}\eta_3(F_{1}(x)) - \mathcal{W}_G^{(2)}\eta_3(G_{1}(x))\|_{\infty} + \|b_F^{(2)} - b_G^{(2)}\|_{\infty}\\
&\leq \|\mathcal{W}_F^{(2)}\eta_3(F_{1}(x)) - \mathcal{W}_G^{(2)}\eta_3(F_{1}(x))\|_{\infty} + \|\mathcal{W}_G^{(2)}\eta_3(F_{1}(x)) -\mathcal{W}_G^{(2)}\eta_3(G_{1}(x))\|_{\infty} + \delta.
\end{align*}
By applying the upper bound proved in equation \ref{inf_bound_base}, we can upper bound the first part $\|\mathcal{W}_F^{(2)}\eta_3(F_{1}(x)) - \mathcal{W}_G^{(2)}\eta_3(F_{1}(x))\|_{\infty}$ by:
\begin{align*}
\|\mathcal{W}_F^{(2)}\eta_3(F_{1}(x)) - \mathcal{W}_G^{(2)}\eta_3(F_{1}(x))\|_{\infty} &\leq  \|\mathcal{W}_F^{(2)} - \mathcal{W}_G^{(2)}\|_{\infty}\|\eta_3(F_{1}(x))\|_{\infty} \\
&\leq W\delta \|F_1(x)\|_{\infty}^3 \leq  \delta W [2(B \vee d)^2]^3.
\end{align*}
By applying the Lipschitz condition proved in Lemma \ref{Lipschitz-ReLU} and the bound proved in equation \ref{base_case_covering}, we can further upper bound the second part $\|\mathcal{W}_G^{(2)}\eta_3(F_{1}(x)) -\mathcal{W}_G^{(2)}\eta_3(G_{1}(x))\|_{\infty}$ by:
\begin{align*}
\|\mathcal{W}_G^{(2)}\eta_3(F_{1}(x)) -\mathcal{W}_G^{(2)}\eta_3(G_{1}(x))\|_{\infty} &\leq \|\mathcal{W}_G^{(2)}\|_{\infty}\|\eta_3(F_{1}(x))-\eta_3(G_{1}(x))\|_{\infty}\\ 
&\leq WB \times 3\sup_{F_1 \in \Phi_1(L,W,S,B)}\|F_{1}(x)\|_{\infty}^2 \times \|F_1(x) - G_1(x)\|_{\infty} \\
&\leq WB \times 3 [2(B \vee d)^2]^2 \times 2\delta (B \vee d) \\
&\leq 24\delta W(B \vee d)^6. 
\end{align*}
Summing the two upper bounds above yields:
\begin{align*}
\|F_2(x) - G_2(x)\|_{\infty} &\leq 8\delta W(B \vee d)^6+24\delta W(B \vee d)^6 + \delta \leq 24\delta W(B \vee d)^7. 
\end{align*}
where we again use the assumption $d \geq 2$ in the last step.\\
Inductive Step: Now we assume that the claim has been proved for $k-1$, where $k \geq 3$. For any $x \in \Omega$ and $F_k \in \Phi_{k}(L,W,S,B)$, we have that:
\begin{align*}
\|F_k(x) - G_k(x)\|_{\infty} &= \|\mathcal{W}_F^{(k)}\eta_3(F_{k-1}(x)) + b_{F}^{(k)} -\mathcal{W}_G^{(k)}\eta_3(G_{k-1}(x)) - b_{G}^{(k)}\|_{\infty} \\
&\leq \|\mathcal{W}_F^{(k)}\eta_3(F_{k-1}(x)) - \mathcal{W}_G^{(k)}\eta_3(G_{k-1}(x))\|_{\infty} + \|b_F^{(k)} - b_G^{(k)}\|_{\infty}\\
&\leq \|\mathcal{W}_F^{(k)}\eta_3(F_{k-1}(x)) - \mathcal{W}_G^{(k)}\eta_3(G_{k-1}(x))\|_{\infty} + \delta.
\end{align*}
Applying triangle inequality helps us upper bound the first term above as follows:
\begin{align*}
&\|\mathcal{W}_F^{(k)}\eta_3(F_{k-1}(x)) - \mathcal{W}_G^{(k)}\eta_3(G_{k-1}(x))\|_{\infty} \\
&\leq \|\mathcal{W}_F^{(k)}\eta_3(F_{k-1}(x)) - \mathcal{W}_G^{(k)}\eta_3(F_{k-1}(x))\|_{\infty}
+ \|\mathcal{W}_G^{(k)}\eta_3(F_{k-1}(x)) - \mathcal{W}_G^{(k)}\eta_3(G_{k-1}(x))\|_{\infty}\\
&\leq \|\mathcal{W}_F^{(k)} - \mathcal{W}_G^{(k)}\|_{\infty}\|\eta_3(F_{k-1}(x))\|_{\infty} + \|\mathcal{W}_G^{(k)}\|_{\infty}\|\eta_3(F_{k-1}(x)) - \eta_3(G_{k-1}(x))\|_{\infty}\\
&\leq \delta W\|F_{k-1}(x)\|_{\infty}^3 + BW\|\eta_3(F_{k-1}(x)) - \eta_3(G_{k-1}(x))\|_{\infty}.
\end{align*}
From Lemma \ref{dnn_inf_norm},  we can upper bound the first term $\delta W\|F_{k-1}(x)\|_{\infty}^3$ by:
\begin{align*}
\delta W\|F_{k-1}(x)\|_{\infty}^3 \leq \delta W^{\frac{3^{k-1}-1}{2}}(B \vee d)^{\frac{5 \cdot 3^{k-1}-3}{2}}2^{\frac{3^k-3}{2}-3k+6}.    
\end{align*}
Moreover, applying Lemma \ref{Lipschitz-ReLU} and the inductive hypothesis let us upper bound the second term $BW\|\eta_3(F_{k-1}(x)) - \eta_3(G_{k-1}(x))\|_{\infty}$ as follows:
\begin{align*}
&BW\|\eta_3(F_{k-1}(x)) - \eta_3(G_{k-1}(x))\|_{\infty} \\
&\leq BW \times 3\sup_{x \in \Omega, \ F_{k-1} \in \Phi_{k-1}(L,W,S,B)}\|F_{k-1}(x)\|_{\infty}^2 \times \|F_{k-1}(x) - G_{k-1}(x)\|_{\infty}\\
&\leq 3BW \times W^{3^{k-2}-1}(B \vee d)^{5 \times 3^{k-2}-1}2^{3^{k-1}-1-2k+4}\|F_{k-1}(x) - G_{k-1}(x)\|_{\infty}\\
&\leq 3BW \times W^{3^{k-2}-1}(B \vee d)^{5 \times 3^{k-2}-1}2^{3^{k-1}-1-2k+4} \times \delta W^{\frac{3^{k-2}-1}{2}}(B \vee d)^{\frac{5 \cdot 3^{k-2}-1}{2}}2^{\frac{3^{k-1}-1}{2}-k+2}3^{k-2}\\
&\leq 3^{k-1}\delta W^{\frac{3^{k-1}-1}{2}}(B \vee d)^{\frac{5 \times 3^{k-1}-1}{2}}2^{\frac{3^k-1}{2}-3k+5}.
\end{align*}
Combining the two upper bounds derived above yields:
\begin{align*}
\|F_k(x) - G_k(x)\|_{\infty} &\leq \delta W^{\frac{3^{k-1}-1}{2}}(B \vee d)^{\frac{5 \cdot 3^{k-1}-3}{2}}2^{\frac{3^k-3}{2}-3k+6} \\
&+ 3^{k-1}\delta W^{\frac{3^{k-1}-1}{2}}(B \vee d)^{\frac{5 \times 3^{k-1}-1}{2}}2^{\frac{3^k-1}{2}-3k+5}+\delta\\
&\leq \delta 3^{k-1}W^{\frac{3^{k-1}-1}{2}}(B \vee d)^{\frac{5 \times 3^{k-1}-1}{2}}2^{\frac{3^k-1}{2}-k+1},
\end{align*}
where the last inequality above follows from $k \geq 3$. Taking supremum with respect to $x \in \Omega$ on the LHS implies the given upper bound also holds for $k$. By induction, the claim is proved.
\end{proof}
\begin{theorem}{(Bounding the DNN space covering number)}
\label{dnn_covering_num}
Fix some sufficiently large $N \in \mathbb{Z}^{+}$. Consider a Deep Neural Network space $\Phi(L,W,S,B)$ with $L=O(1), W=O(N),S=O(N)$ and $B=O(N)$. Then the $\log$ value of the covering number of this DNN space with respect to the inf-norm $\|F(x)\|_{\infty} := \sup_{x \in \Omega}|F(x)|$, which is denoted by $\mathcal{N}(\delta, \Phi(L,W,S,B), \| \cdot \|_{\infty})$, can be upper bounded by: 
\begin{equation}
\begin{aligned}
\log \mathcal{N}(\delta, \Phi(L,W,S,B), \| \cdot \|_{\infty})  = O\left(S\Big[\log(\delta^{-1}) + 3^{L}\log(WB)\Big]\right).
\end{aligned}    
\end{equation}
\begin{proof}
We firstly fix a sparsity pattern (i.e, the locations of the non-zero entries are fixed). By picking $k=L$ in Lemma \ref{dnn_diff_bound}, we get the following upper bound on the covering number with respect to $\|\cdot\|_{\infty}$:
\begin{align*}
\Big(\frac{\delta}{3^{L-1}W^{\frac{3^{L-1}-1}{2}}(B \vee d)^{\frac{5 \times 3^{L-1}-1}{2}}2^{\frac{3^L-1}{2}-L+2}}\Big)^{-S}.    
\end{align*}
Furthermore, note that the number of feasible configurations is upper bounded by ${(W+1)^L \choose S} \leq (W+1)^{LS}$.\cite{schmidt2020nonparametric,farrell2021deep} Plug in the previous inequality and yields:
\begin{align*}
\log \mathcal{N}(\delta, \Phi(L,W,S,B), \| \cdot \|_{\infty}) &\leq \log \left[(W+1)^{LS}\Big(\frac{\delta}{3^{L-1}W^{\frac{3^{L-1}-1}{2}}(B \vee d)^{\frac{5 \times 3^{L-1}-1}{2}}2^{\frac{3^L-1}{2}-L+1}}\Big)^{-S}\right]\\
&\leq S\log\Big[\delta^{-1}(W+1)^{L}3^{L-1}W^{\frac{3^{L-1}-1}{2}}(B \vee d)^{\frac{5 \times 3^{L-1}-1}{2}}2^{\frac{3^L-1}{2}-L+1}\Big]\\
&\lesssim S\Big[\log(\delta^{-1}) + L\log(3W) + 3^{L}\log(W(B \vee d)) + 3^{L}\log 2\Big].
\end{align*} 
Note that here the dimension $d$ is some constant. Thus, by plugging in thee given magnitudes $L=O(1), W=O(N),S=O(N)$ and $B=O(N)$, we can further deduce that:
\begin{align*}
\log \mathcal{N}(\delta, \Phi(L,W,S,B), \| \cdot \|_{\infty}) \lesssim S\Big[\log(\delta^{-1}) + 3^{L}\log(WB)\Big].    
\end{align*}
This finishes our proof.
\end{proof}
\end{theorem}

Now let's consider upper bounding the covering number of the $l_{2}$ norm of the sparse Deep Neural Networks' gradients. Note that for any $1 \leq k \leq L-1$, any $k-$ReLU3 Deep Neural Network $F_{k} \in \Phi_{k}(L,W,S,B)$ is a map from $\mathbb{R}^{d}$ to $\mathbb{R}^{W}$. For any $1 \leq l \leq W$, we use $F_{k,l}(x)$ to denote the $l$-th component of the map $F_{k}$. This helps us write the map $F_{k}(x)$ and its Jacobian matrix $J[F_k](x)$ explicitly as:
\begin{align*}
F_{k}(x) &= [F_{k,1}(x),F_{k,2}(x),\cdots,F_{k,W}(x)]^{T} \in \mathbb{R}^{W}. \\
J[F_k](x) &= \begin{bmatrix}
\frac{\partial}{\partial x_{1}}F_{k,1}(x) &\frac{\partial}{\partial x_{2}}F_{k,1}(x) &\cdots &\frac{\partial}{\partial x_{d}}F_{k,1}(x)\\
\frac{\partial}{\partial x_{1}}F_{k,2}(x) &\frac{\partial}{\partial x_{2}}F_{k,2}(x) &\cdots &\frac{\partial}{\partial x_{d}}F_{k,2}(x)\\
\cdots &\cdots &\ddots \\ 
\frac{\partial}{\partial x_{1}}F_{k,W}(x) &\frac{\partial}{\partial x_{2}}F_{k,W}(x) &\cdots &\frac{\partial}{\partial x_{d}}F_{k,W}(x)\\
\end{bmatrix} \in \mathbb{R}^{W \times d}.
\end{align*}
In particular, when $k = L$, we have that any $F_{L} \in \Phi_{L}(L,W,S,B) = \Phi(L,W,S,B)$ is a map from $\mathbb{R}^{d}$ to $\mathbb{R}$. Thus, its Jacobian can be explicitly written as the following row vector:
\begin{align*}
J[F_{L}](x) = [\frac{\partial}{\partial x_{1}}F_{L}(x), \frac{\partial}{\partial x_{2}}F_{L}(x), \cdots \frac{\partial}{\partial x_{d}}F_{L}(x)] \in \mathbb{R}^{1 \times d}.    
\end{align*}

\begin{lemma}{(Upper bound on $\infty$-norm of Jacobian/Gradient of elements in the DNN space)}
\label{dnn_grad_inf_norm}
For any $1 \leq k \leq L$, the following inequality holds:
\begin{align*}
\sup_{x \in \Omega, F_{k} \in \Phi_{k}(L,W,S,B)}\|J[F_k](x)\|_{\infty} \leq W^{\frac{3^{k-1}-1}{2}}(B \vee d)^{\frac{5 \cdot 3^{k-1}-1}{2}}2^{\frac{3^k-1}{2}-k+1}3^{k-1}.
\end{align*}
\end{lemma}
\begin{proof}
We use induction on $k$ to prove the claim.\\
Base case: $k=1$. By the definition of Jacobian matrix, we have that for any $x \in \Omega$ and any $F_{1} \in \Phi_{1}(L,W,S,B)$, the following holds:
\begin{align*}
\|J[F_1](x)\|_{\infty} = \|\mathcal{W}_{F}^{(1)}\|_{\infty} \leq dB \leq 2(B \vee d)^2.  
\end{align*}
Inductive Step: Assume that the claim has been proved for $k-1$, where $2 \leq k \leq L$. For any $x \in \Omega$ and any $F_k \in \Phi_{k}(L,W,S,B)$, by applying the Chain Rule, we can write the Jacobian matrix $J[F_{k}](x)$ as $J[F_{k}](x) = \mathcal{W}_{F}^{(k)}J[\eta_3 \circ F_{k-1}](x)$, where the ReLU3 activation function $\eta_{3}$ is applied to each component $F_{k-1,l} \ (1 \leq l \leq W)$ of the map $F_{k}$. Then we have the following upper bound:
\begin{equation}
\label{init_bound}
\begin{aligned}
\|J[F_{k}](x)\|_{\infty} \leq \|\mathcal{W}_{F}^{(k)}\|_{\infty}\|J[\eta_3 \circ F_{k-1}](x)\|_{\infty} \leq WB\|J[\eta_3 \circ F_{k-1}](x)\|_{\infty}. 
\end{aligned}    
\end{equation}
Note that the composition $\eta_3 \circ F_{k-1}$ is a map from $\mathbb{R}^{d}$ to $\mathbb{R}^{W}$. Hence, the Jacobian matrix $J[\eta_3 \circ F_{k-1}](x)$ is of shape $\mathbb{R}^{W \times d}$. Applying the Chain Rule again implies:
\begin{align*}
\|J[\eta_3 \circ F_{k-1}](x)\|_{\infty} = \sup_{1 \leq l \leq W}(\sum_{j=1}^{d}|3\eta_{2}(F_{k-1,l}(x))\frac{\partial F_{k-1,l}(x)}{\partial x_{j}}|).    
\end{align*}
Furthermore, for any $1 \leq l \leq W$, the summation on the RHS above can be upper bounded by:
\begin{align*}
\sum_{j=1}^{d}|3\eta_{2}(F_{k-1,l}(x))\frac{\partial F_{k-1,l}(x)}{\partial x_{j}}| &\leq 3\|F_{k-1}(x)\|_{\infty}^2(\sum_{j=1}^{d}|\frac{\partial}{\partial x_{j}}F_{k-1,l}(x)|) \leq 3\|F_{k-1}(x)\|_{\infty}^2\|J[F_{k-1}](x)\|_{\infty}.
\end{align*}
Now let's take supremum with respect to $l$ and apply the inductive hypothesis and Lemma \ref{dnn_inf_norm}. This yields:
\begin{equation}
\label{inter_step}
\begin{aligned}
\|J[\eta_3 \circ F_{k-1}](x)\|_{\infty} &\leq 3W^{3^{k-2}-1}(B \vee d)^{5 \cdot 3^{k-2}-1}2^{3^{k-1}-1-2k+4} \times W^{\frac{3^{k-2}-1}{2}}(B \vee d)^{\frac{5 \cdot 3^{k-2}-1}{2}}2^{\frac{3^{k-1}-1}{2}-k+2}3^{k-2} \\
&= W^{\frac{3^{k-1}-1}{2}-1}(B \vee d)^{\frac{5 \cdot 3^{k-1}-1}{2}-1}2^{\frac{3^k-1}{2}-3k+5}3^{k-1}.  
\end{aligned}    
\end{equation}
By substituting equation \ref{inter_step} into equation \ref{init_bound}, we can derive the final bound:
\begin{align*}
\|J[F_{k}](x)\|_{\infty} &\leq WB\|J[\eta_3 \circ F_{k-1}](x)\|_{\infty} \leq W^{\frac{3^{k-1}-1}{2}}(B \vee d)^{\frac{5 \cdot 3^{k-1}-1}{2}}2^{\frac{3^k-1}{2}-3k+5}3^{k-1}\\
&\leq W^{\frac{3^{k-1}-1}{2}}(B \vee d)^{\frac{5 \cdot 3^{k-1}-1}{2}}2^{\frac{3^k-1}{2}-k+1}3^{k-1}. 
\end{align*}
where the last inequality above follows from $k \geq 2$. Taking supremum with respect to $x \in \Omega$ and $F_k \in \Phi_{k}(L,W,S,B)$ on the LHS implies that the given upper bound also holds for $k$. By induction, the claim is proved.
\end{proof}
For the convenience of the following proof, we first prove this lemma for vector 2 norm and $\infty$ norm.
\begin{lemma}
\label{norm_compare}
Given any two row vectors $\boldsymbol{u},\boldsymbol{v} \in \mathbb{R}^{1 \times d}$, we have:
\begin{align*}
\Big | \|\boldsymbol{u}\| - \|\boldsymbol{v}\| \Big | \leq \|\boldsymbol{u} - \boldsymbol{v}\|_{\infty}.    
\end{align*}
\end{lemma}
\begin{proof}
Assume that the two vectors $\boldsymbol{u},\boldsymbol{v} \in \mathbb{R}^{d}$ can be explicitly written as $\boldsymbol{u} = [u_{1},u_{2},\cdots,u_{d}]$ and $v = [v_{1},v_{2},\cdots,v_{d}]$, respectively. By applying Cauchy-Schwarz inequality, we have:
\begin{align*}
\Big | \|\boldsymbol{u}\| - \|\boldsymbol{v}\| \Big |^2 &=\left|\sqrt{\sum_{i=1}^{d}u_{i}^2} -\sqrt{\sum_{i=1}^{d}v_{i}^2}\right|^2\\
&=\sum_{i=1}^{d}u_{i}^2 +\sum_{i=1}^{d}v_{i}^2 -2\sqrt{\sum_{i=1}^{d}u_{i}^2}\sqrt{\sum_{i=1}^{d}v_{i}^2}\\
&\leq \sum_{i=1}^{d}u_{i}^2 +\sum_{i=1}^{d}v_{i}^2 - 2\sum_{i=1}^{d}u_{i}v_{i} = \sum_{i=1}^{d}|u_{i}-v_{i}|^2\\
&\leq \Big(\sum_{i=1}^{d}|u_i - v_i|\Big)^2 = \|\boldsymbol{u} - \boldsymbol{v}\|_{\infty}^2.
\end{align*}
Taking the square root on both sides yields the desired inequality.
\end{proof}

Then we upper bound the Lipschitz constant of the gradient of the neural network. Given a DNN space $\Phi(L,W,S,B)$, we define a corresponding DNN Gradient space $\nabla \Phi(L,W,S,B)$ as:
\begin{equation}
\label{DNN_grad_space}
\nabla \Phi(L,W,S,B) := \{\|\nabla F\| \ | \ F \in \Phi(L,W,S,B)\}.
\end{equation}

\begin{lemma}{(Relation between the covering number of the DNN Gradient space and parameter space)}
\label{dnn_grad_diff}
For any $1 \leq k \leq L$, suppose that a pair of different two networks $F_k,G_k \in \Phi_{k}(L,W,S,B)$ are given by:
\begin{align*}
F_k(x) &:= (\mathcal{W}_{F}^{(k)}\eta_3(\cdot) + b_{F}^{(k)}) \cdots (\mathcal{W}_{F}^{(1)}x + b_{F}^{(1)}),\\
G_k(x) &:= (\mathcal{W}_{G}^{(k)}\eta_3(\cdot) + b_{G}^{(k)}) \cdots (\mathcal{W}_{G}^{(1)}x + b_{G}^{(1)}).
\end{align*}
Furthermore, assume that the $\| \ \|_{\infty}$ norm of the distance between the parameter spaces is uniformly upper bounded by $\delta$, i.e
\begin{equation}
\label{dnn_grad_para_constraint} 
\|W_{F}^{(l)} - W_{G}^{(l)}\|_{\infty,\infty} \leq \delta, \ \|b_{F}^{(l)} - b_{G}^{(l)}\|_{\infty} \leq \delta, \ (\forall \ 1 \leq l \leq k).
\end{equation}
Then we have:
\begin{equation}
\label{dnn_grad_jacobian_bound}
\sup_{x \in \Omega}\|J[F_k](x)-J[G_k](x)\|_{\infty} \leq \delta  W^{\frac{3^{k-1}-1}{2}}(B \vee d)^{\frac{5 \cdot 3^{k-1}-1}{2}}2^{\frac{3^k-1}{2}-k+1}3^{2k-2}.     
\end{equation}
In particular, when $k=L$, we have:
\begin{equation}
\label{dnn_grad_gradient_l2_bound}
\sup_{x \in \Omega}\Big|\|\nabla F_{L}(x)\| - \|\nabla G_{L}(x)\|\Big| \leq \delta  W^{\frac{3^{L-1}-1}{2}}(B \vee d)^{\frac{5 \cdot 3^{L-1}-1}{2}}2^{\frac{3^L-1}{2}-L+1}3^{2L-2}. 
\end{equation}
\end{lemma}
\begin{proof}
We use induction on $k$ to prove the claim.\\
Base case: When $k=1$, we have that for any $x \in \Omega$ and any $F_1,G_1 \in \Phi_{1}(L,W,S,B)$, the following holds:
\begin{align*}
\|J[F_1](x) - J[G_1](x)\|_{\infty} = \|\mathcal{W}_{F}^{(1)} - \mathcal{W}_{G}^{(1)}\|_{\infty} \leq \delta d \leq 2\delta (B \vee d)^2.
\end{align*}
Inductive Step: assume that the claim has been proved for $k-1$, where $2 \leq k \leq L$. Then for any $x \in \Omega$ and $F_k,G_k \in \Phi_{k}(L,W,S,B)$ satisfying constraint \ref{dnn_grad_para_constraint}, applying the Chain Rule and triangle inequality help us upper bound the inf-norm $\|J[F_k](x) -J[G_k](x)\|_{\infty}$ by:
\begin{equation}
\begin{aligned}
&\|J[F_k](x) - J[G_k](x)\|_{\infty} = \|\mathcal{W}_{F}^{(k)}J[\eta_3 \circ F_{k-1}](x) - \mathcal{W}_{G}^{(k)}J[\eta_3 \circ G_{k-1}](x)\|_{\infty}\\
&\leq \|\mathcal{W}_{F}^{(k)}J[\eta_3 \circ F_{k-1}](x) - \mathcal{W}_{G}^{(k)}J[\eta_3 \circ F_{k-1}](x)\|_{\infty} + \|\mathcal{W}_{G}^{(k)}J[\eta_3 \circ F_{k-1}](x) - \mathcal{W}_{G}^{(k)}J[\eta_3 \circ G_{k-1}](x)\|_{\infty}\\
&\leq \|\mathcal{W}_{F}^{(k)} - \mathcal{W}_{G}^{(k)}\|_{\infty}\|J[\eta_3 \circ F_{k-1}](x)\|_{\infty} + \|\mathcal{W}_{G}^{(k)}\|_{\infty}\|J[\eta_3 \circ F_{k-1}](x)-J[\eta_3 \circ G_{k-1}](x)\|_{\infty}\\
&\leq \delta W\|J[\eta_3 \circ F_{k-1}](x)\|_{\infty} + BW\|J[\eta_3 \circ F_{k-1}](x)-J[\eta_3 \circ G_{k-1}](x)\|_{\infty}.
\end{aligned}    
\end{equation}
Using equation \ref{inter_step} helps us upper bound the first term by:
\begin{equation}
\label{bound_part1}
\delta W\|J[\eta_3 \circ F_{k-1}](x)\|_{\infty} \leq \delta W^{\frac{3^{k-1}-1}{2}}(B \vee d)^{\frac{5 \cdot 3^{k-1}-1}{2}-1}2^{\frac{3^k-1}{2}-3k+5}3^{k-1}.  
\end{equation}
Note that the two compositions $\eta_3 \circ F_{k-1}$ and $\eta_3 \circ G_{k-1}$ both map from $\mathbb{R}^{d}$ to $\mathbb{R}^{W}$. Hence, the two Jacobian matrices $J[\eta_3 \circ F_{k-1}](x)$ and $J[\eta_3 \circ G_{k-1}](x)$ are of shape $\mathbb{R}^{W \times d}$. Applying the Chain Rule again implies:
\begin{align*}
\|J[\eta_3 \circ F_{k-1}](x)-J[\eta_3 \circ G_{k-1}](x)\|_{\infty} = \sup_{1 \leq l \leq W}(\sum_{j=1}^{d}|3\eta_{2}(F_{k-1,l}(x))\frac{\partial F_{k-1,l}(x)}{\partial x_{j}}-3\eta_{2}(G_{k-1,l}(x))\frac{\partial G_{k-1,l}(x)}{\partial x_{j}}|).    
\end{align*}
For any $1 \leq l \leq W$, the summation on the RHS above can be upper bounded by:
\begin{align*}
&\sum_{j=1}^{d}|3\eta_{2}(F_{k-1,l}(x))\frac{\partial F_{k-1,l}(x)}{\partial x_{j}}-3\eta_{2}(G_{k-1,l}(x))\frac{\partial G_{k-1,l}(x)}{\partial x_{j}}|\\
&\leq \sum_{j=1}^{d}|3\eta_{2}(F_{k-1,l}(x))\frac{\partial F_{k-1,l}(x)}{\partial x_{j}} - 3\eta_{2}(G_{k-1,l}(x))\frac{\partial F_{k-1,l}(x)}{\partial x_{j}}|\\
&+\sum_{j=1}^{d}|3\eta_{2}(G_{k-1,l}(x))\frac{\partial F_{k-1,l}(x)}{\partial x_{j}} - 3\eta_{2}(G_{k-1,l}(x))\frac{\partial G_{k-1,l}(x)}{\partial x_{j}}|\\
&\leq \sum_{j=1}^{d}|3\eta_2(F_{k-1,l}(x))-3\eta_{2}(G_{k-1,l}(x))\|\frac{\partial F_{k-1,l}(x)}{\partial x_{j}}|+\sum_{j=1}^{d}|3\eta_{2}(G_{k-1,l}(x))\|\frac{\partial F_{k-1,l}(x)}{\partial x_{j}}-\frac{\partial G_{k-1,l}(x)}{\partial x_{j}}|.
\end{align*}
We denote the two summations above by $T_{1}$ and $T_{2}$, respectively: 
\begin{align*}
T_{1} &:= \sum_{j=1}^{d}|3\eta_2(F_{k-1,l}(x))-3\eta_{2}(G_{k-1,l}(x))\|\frac{\partial F_{k-1,l}(x)}{\partial x_{j}}|,\\
T_{2} &:= \sum_{j=1}^{d}|3\eta_{2}(G_{k-1,l}(x))\|\frac{\partial F_{k-1,l}(x)}{\partial x_{j}}-\frac{\partial G_{k-1,l}(x)}{\partial x_{j}}|.
\end{align*}
For the first sum $T_{1}$, applying Lemma \ref{dnn_inf_norm}, Lemma \ref{Lipschitz-ReLU}, Lemma \ref{dnn_diff_bound} and Lemma \ref{dnn_grad_inf_norm} yields the following upper bound:
\begin{align*}
T_1 &\leq 6\Big(\sup_{x \in \Omega, \ F_{k-1} \in \Phi_{k-1}(L,W,S,B)}\|F_{k-1}(x)\|_{\infty}\Big)\|F_{k-1}(x) - G_{k-1}(x)\|_{\infty}\sum_{j=1}^{d}|\frac{\partial F_{k-1,l}(x)}{\partial x_{j}}|\\
&\leq 6\Big(\sup_{x \in \Omega, \ F_{k-1} \in \Phi_{k-1}(L,W,S,B)}\|F_{k-1}(x)\|_{\infty}\Big)\|F_{k-1}(x) - G_{k-1}(x)\|_{\infty}\|J[F_{k-1}](x)\|_{\infty} \\
&\leq 3 \times W^{\frac{3^{k-2}-1}{2}}(B \vee d)^{\frac{5 \cdot 3^{k-2}-1}{2}}2^{\frac{3^{k-1}-1}{2}-k+2} \times \delta W^{\frac{3^{k-2}-1}{2}}(B \vee d)^{\frac{5 \cdot 3^{k-2}-1}{2}}2^{\frac{3^{k-1}-1}{2}-k+2}3^{k-2}\\
&\times W^{\frac{3^{k-2}-1}{2}}(B \vee d)^{\frac{5 \cdot 3^{k-2}-1}{2}}2^{\frac{3^{k-1}-1}{2}-k+2}3^{k-2} = \delta W^{\frac{3^{k-1}-3}{2}}(B \vee d)^{\frac{5 \cdot 3^{k-1}-3}{2}}2^{\frac{3^k-3}{2}-3k+6}3^{2k-3}.
\end{align*}
For the second sum $T_2$, applying Lemma \ref{dnn_inf_norm} and inductive hypothesis yields:
\begin{align*}
T_2 &\leq  3\Big(\sup_{x \in \Omega, \ G_{k-1} \in \Phi_{k-1}(L,W,S,B)}\|G_{k-1}(x)\|_{\infty}\Big)^2\|J[F_{k-1}](x)-J[G_{k-1}](x)\|_{\infty} \\
&\leq 3 \times W^{3^{k-2}-1}(B \vee d)^{5 \cdot 3^{k-2}-1}2^{3^{k-1}-1-2k+4} \times \delta  W^{\frac{3^{k-2}-1}{2}}(B \vee d)^{\frac{5 \cdot 3^{k-2}-1}{2}}2^{\frac{3^{k-1}-1}{2}-k+2}3^{2k-4}\\
&= \delta W^{\frac{3^{k-1}-3}{2}}(B \vee d)^{\frac{5 \cdot 3^{k-1} - 3}{2}}2^{\frac{3^{k}-3}{2}-3k+6}3^{2k-3}.
\end{align*}
Combining the two upper bounds on $T_1$ and $T_2$ yields:
\begin{align*}
&\sum_{j=1}^{d}|3\eta_{2}(F_{k-1,l}(x))\frac{\partial F_{k-1,l}(x)}{\partial x_{j}}-3\eta_{2}(G_{k-1,l}(x))\frac{\partial G_{k-1,l}(x)}{\partial x_{j}}| \\
&\leq T_1 + T_2 \leq 2 \times \delta W^{\frac{3^{k-1}-3}{2}}(B \vee d)^{\frac{5 \cdot 3^{k-1} - 3}{2}}2^{\frac{3^{k}-3}{2}-3k+6}3^{2k-3}.
\end{align*}
By taking supremum with respect to $1 \leq l \leq W$ on the LHS yields:
\begin{equation}
\label{bound_part2}
\begin{aligned}
BW\|J[\eta_3 \circ F_{k-1}](x)-J[\eta_3 \circ G_{k-1}](x)\|_{\infty} &\leq  \delta W^{\frac{3^{k-1}-1}{2}}(B \vee d)^{\frac{5 \cdot 3^{k-1} - 1}{2}}2^{\frac{3^{k}-1}{2}-3k+6}3^{2k-3}.
\end{aligned}
\end{equation}
By adding the two upper bounds in \ref{bound_part1} and \ref{bound_part2}, we can deduce that:
\begin{align*}
\|J[F_k](x) - J[G_k](x)\|_{\infty} &\leq \delta W^{\frac{3^{k-1}-1}{2}}(B \vee d)^{\frac{5 \cdot 3^{k-1}-1}{2}-1}2^{\frac{3^k-1}{2}-3k+5}3^{k-1}\\
&+ \delta W^{\frac{3^{k-1}-1}{2}}(B \vee d)^{\frac{5 \cdot 3^{k-1} - 1}{2}}2^{\frac{3^{k}-1}{2}-3k+6}3^{2k-3}\\
&\leq \delta W^{\frac{3^{k-1}-1}{2}}(B \vee d)^{\frac{5 \cdot 3^{k-1} - 1}{2}}2^{\frac{3^{k}-1}{2}-k+1}3^{2k-2}.
\end{align*}
where the last inequality above follows from $k \geq 2$. Taking supremum with respect to $x \in \Omega$ on the LHS implies the given upper bound also holds for $k$. By induction, the claim is proved.\\
In particular, when $k=L$, we have $\nabla F_{L}(x) = J[F_{L}](x)^{T}$ for any $x \in \Omega$. Applying Lemma \ref{norm_compare} then yields:
\begin{align*}
\sup_{x \in \Omega}\Big|\|\nabla F_{L}(x)\| - \|\nabla G_{L}(x)\|\Big| &= \sup_{x \in \Omega}\Big|\|\nabla J[F_{L}](x)^{T}\| - \|\nabla J[G_{L}](x)^{T}\|\Big|\\
&\leq \sup_{x \in \Omega}\|J[F_L](x)-J[G_L](x)\|_{\infty}\\
&\leq \delta  W^{\frac{3^{L-1}-1}{2}}(B \vee d)^{\frac{5 \cdot 3^{L-1}-1}{2}}2^{\frac{3^L-1}{2}-k+1}3^{2L-2}.
\end{align*}
This finishes our proof of the Lemma.
\end{proof}

\begin{theorem}{(Bounding the DNN Gradient space covering number)}
\label{dnn_grad_covering_num}
Fix some sufficiently large $N \in \mathbb{Z}^{+}$. Consider a Deep Neural Network space $\Phi(L,W,S,B)$ with $L=O(1), W=O(N),S=O(N)$ and $B=O(N)$. Then the $\log$ value of the covering number of the DNN Gradient space with respect to the $\|\cdot\|_{\infty}$ norm $\|F(x)\|_{\infty} := \sup_{x \in \Omega}|F(x)|$, which is denoted by $\mathcal{N}(\delta, \nabla \Phi(L,W,S,B), \| \cdot \|_{\infty})$, can be upper bounded by: 
\begin{equation}
\begin{aligned}
\log \mathcal{N}(\delta, \nabla \Phi(L,W,S,B), \| \cdot \|_{\infty})  = O\left(S\Big[\log(\delta^{-1}) + 3^{L}\log(WB)\Big]\right).
\end{aligned}    
\end{equation}
\begin{proof}
We firstly fix a sparsity pattern (i.e, the locations of the non-zero entries are fixed). Using equation \ref{dnn_grad_gradient_l2_bound} in Lemma \ref{dnn_grad_diff}, yields the following upper bound on the covering number with respect to $\|\cdot\|_{\infty}$:
\begin{align*}
\Big(\frac{\delta}{W^{\frac{3^{L-1}-1}{2}}(B \vee d)^{\frac{5 \cdot 3^{L-1}-1}{2}}2^{\frac{3^L-1}{2}-L+1}3^{2L-2}}\Big)^{-S}.    
\end{align*}
Furthermore, note that the number of feasible configurations is upper bounded by:
${(W+1)^L \choose S} \leq (W+1)^{LS}. $\cite{schmidt2020nonparametric,farrell2021deep}. Plug this inequality into the previous estimation then yields:
\begin{align*}
\log \mathcal{N}(\delta, \Phi(L,W,S,B), \| \cdot \|_{\infty}) &\leq \log \left[(W+1)^{LS}\Big(\frac{\delta}{W^{\frac{3^{L-1}-1}{2}}(B \vee d)^{\frac{5 \cdot 3^{L-1}-1}{2}}2^{\frac{3^L-1}{2}-L+1}3^{2L-2}}\Big)^{-S}\right]\\
&\leq S\log\Big[\delta^{-1}(W+1)^{L}3^{2L-2}W^{\frac{3^{L-1}-1}{2}}(B \vee d)^{\frac{5 \cdot 3^{L-1}-1}{2}}2^{\frac{3^L-1}{2}-L+1}\Big]\\
&\lesssim S\Big[\log(\delta^{-1}) + 2L\log(3W) + 3^{L}\log(W(B \vee d)) + 3^{L}\log 2\Big].
\end{align*} 
Note that here the dimension $d$ is some constant. Thus, by plugging in thee given magnitudes $L=O(1), W=O(N),S=O(N)$ and $B=O(N)$, we can further deduce that:
\begin{align*}
\log \mathcal{N}(\delta, \Phi(L,W,S,B), \| \cdot \|_{\infty}) \lesssim S\Big[\log(\delta^{-1}) + 3^{L}\log(WB)\Big].    
\end{align*}
This finishes our proof.
\end{proof}
\end{theorem}

Now let's consider upper bounding the covering number of the Laplacian of the sparse Deep Neural Networks. Note that for any $1 \leq k \leq L-1$, any $k-$ReLU3 Deep Neural Network $F_{k} \in \Phi_{k}(L,W,S,B)$ is a vector-valued function mapping from $\mathbb{R}^{d}$ to $\mathbb{R}^{W}$. Moreover, we define the Laplacian of $F_k(x)$, which is denoted by $\Delta[F_{k}](x)$, as follows:
\begin{align*}
\Delta[F_{k}](x) &= [\Delta F_{k,1}(x), \Delta F_{k,2}(x),\cdots,\Delta F_{k,W}(x)]^{T} \in \mathbb{R}^{W},
\end{align*}
where for any $1 \leq l \leq W$, we have:
\begin{align*}
\Delta F_{k,l}(x) = \sum_{j=1}^{d}\frac{\partial^2}{\partial x_{j}^2}F_{k,l}(x).    
\end{align*}
In particular, when $k = L$, we have that any $F_{L} \in \Phi_{L}(L,W,S,B) = \Phi(L,W,S,B)$ is a scalar-valued function mapping from $\mathbb{R}^{d}$ to $\mathbb{R}$. Thus, its Laplacian can be explicitly written as:
\begin{align*}
\Delta[F_{L}](x) = \Delta F_{L}(x) =  \sum_{j=1}^{d}\frac{\partial^2}{\partial x_{j}^2}F_{L}(x).
\end{align*}

For both Lemma \ref{dnn_laplacian_inf_norm} and Lemma \ref{dnn_laplacian_diff} below, we consider a fixed Deep Neural Network space $\Phi(L,W,S,B)$ with $L=O(1), W=O(N),S=O(N)$ and $B=O(N)$, where $N \in \mathbb{Z}^{+}$ is fixed and sufficiently large.

\begin{lemma}{(Upper bound on $\infty$-norm of Laplacian of elements in the DNN space)}
\label{dnn_laplacian_inf_norm}
For any $1 \leq k \leq L$, we have the following upper bound:
\begin{align*}
\sup_{x \in \Omega, F_{k} \in \Phi_{k}(L,W,S,B)}\|\Delta[F_k](x)\|_{\infty} = O\Big(W^{\frac{3^{k-1}-1}{2}}(B \vee d)^{\frac{5 \cdot 3^{k-1} - 1}{2}}\Big). 
\end{align*}
\end{lemma}
\begin{proof}
We use induction on $k$ to prove the claim.\\
Base case: $k=1$. Note that any $F_{1} \in \Phi_{1}(L,W,S,B)$ is a linear transform, so the Laplacian $\Delta[F_{1}](x)$ must be the zero vector for any $x \in \Omega$. This implies:
\begin{align*}
\|\Delta[F_{1}](x)\|_{\infty} = 0 \lesssim (B \vee d)^2.
\end{align*}
Inductive Step: Assume that the claim has been proved for $k-1$, where $2 \leq k \leq L$. For any $x \in \Omega$ and any $F_k \in \Phi_{k}(L,W,S,B)$, using linearity of the Laplacian operator implies:
\begin{align*}
\Delta[F_k](x) = \mathcal{W}_{F}^{(k)}\Delta[\eta_3 \circ F_{k-1}](x).    
\end{align*}
Taking the inf-norm on both sides of the identity above implies:
\begin{align*}
\|\Delta[F_k](x)\|_{\infty} \leq \|\mathcal{W}_{F}^{(k)}\|_{\infty}\|\Delta[\eta_3 \circ F_{k-1}](x)\|_{\infty} \leq WB\|\Delta[\eta_3 \circ F_{k-1}](x)\|_{\infty}. 
\end{align*}
It now remains to upper bound the term $\|\Delta[\eta_3 \circ F_{k-1}](x)\|_{\infty}$. For any $1 \leq l \leq W$, we will use the Chain Rule to write the $l$-th component $\Big(\Delta[\eta_3 \circ F_{k-1}](x)\Big)_{l}$ in an explicit form. For any $1 \leq j \leq d$, we have:
\begin{align*}
\frac{\partial}{\partial x_{j}}\eta_3[F_{k-1,l}(x)] = 3\eta_2[F_{k-1,l}(x)]\frac{\partial}{\partial x_{j}}F_{k-1,l}(x).
\end{align*}
Differentiating with respect to $x_{j}$ on both sides above yields:
\begin{equation}
\label{dnn_laplacian_comp}
\frac{\partial^2}{\partial x_{j}^2}\eta_3[F_{k-1,l}(x)] = 6\eta_{1}[F_{k-1,l}(x)]\Big(\frac{\partial}{\partial x_{j}}F_{k-1,l}(x)\Big)^2+3\eta_2[F_{k-1,l}(x)]\frac{\partial^2}{\partial x_{j}^2}F_{k-1,l}(x).
\end{equation}
Summing the expression above from $j=1$ to $j=d$ implies:
\begin{align*}
\left|\Big(\Delta[\eta_3 \circ F_{k-1}](x)\Big)_{l}\right| &= \left|\sum_{j=1}^{d}\frac{\partial^2}{\partial x_{j}^2}\eta_3[F_{k-1,l}(x)]\right|\\ 
&= \left|6\eta_{1}[F_{k-1,l}(x)]\sum_{j=1}^{d}\Big(\frac{\partial}{\partial x_{j}}F_{k-1,l}(x)\Big)^2+3\eta_2[F_{k-1,l}(x)]\sum_{j=1}^{d}\frac{\partial^2}{\partial x_{j}^2}F_{k-1,l}(x) \right|\\
&\leq 6\Big|\eta_{1}[F_{k-1,l}(x)]\Big|\left(\sum_{j=1}^{d}\Big|\frac{\partial}{\partial x_{j}}F_{k-1,l}(x)\Big|\right)^2 + 3\Big|\eta_2[F_{k-1,l}(x)]\Big|\left|\sum_{j=1}^{d}\frac{\partial^2}{\partial x_{j}^2}F_{k-1,l}(x)\right|.
\end{align*}
We denote the two summations above by $U_1$ and $U_2$, respectively:
\begin{align*}
U_1 &:= 6\Big|\eta_{1}[F_{k-1,l}(x)]\Big|\left(\sum_{j=1}^{d}\Big|\frac{\partial}{\partial x_{j}}F_{k-1,l}(x)\Big|\right)^2.\\
U_2 &:= 3\Big|\eta_2[F_{k-1,l}(x)]\Big|\left|\sum_{j=1}^{d}\frac{\partial^2}{\partial x_{j}^2}F_{k-1,l}(x)\right|.
\end{align*}
On the one hand, by applying Lemma \ref{dnn_inf_norm} and Lemma \ref{dnn_grad_inf_norm}, we can upper bound $U_1$ by:
\begin{align*}
U_1 &\leq 6\Big(\sup_{x \in \Omega, \ F_{k-1} \in \Phi_{k-1}(L,W,S,B)}\|F_{k-1}(x)\|_{\infty}\Big)\Big(\sup_{x \in \Omega, F_{k-1} \in \Phi_{k-1}(L,W,S,B)}\|J[F_{k-1}](x)\|_{\infty} \Big)^2\\
&\leq 6 \times W^{\frac{3^{k-2}-1}{2}}(B \vee d)^{\frac{5 \cdot 3^{k-2}-1}{2}}2^{\frac{3^{k-1}-1}{2}-k+2} \times W^{3^{k-2}-1}(B \vee d)^{5 \cdot 3^{k-2}-1}2^{3^{k-1}-1-2k+4}3^{2k-4}\\
&\lesssim W^{\frac{3^{k-1}-3}{2}}(B \vee d)^{\frac{5 \cdot 3^{k-1}-3}{2}},
\end{align*}
where the last step above follows from $k \leq L$ and $L=O(1)$.\\
On the other hand, by applying Lemma \ref{dnn_inf_norm} and the inductive hypothesis, we have:
\begin{align*}
U_2 &\leq 3\Big(\sup_{x \in \Omega, \ F_{k-1} \in \Phi_{k-1}(L,W,S,B)}\|F_{k-1}(x)\|_{\infty}\Big)^2\|\Delta[F_{k-1}](x)\|_{\infty}\\
&\lesssim 3 \times W^{3^{k-2}-1}(B \vee d)^{5 \cdot 3^{k-2}-1}2^{3^{k-1}-1-2k+4} \times W^{\frac{3^{k-2}-1}{2}}(B \vee d)^{\frac{5 \cdot 3^{k-2} - 1}{2}}\\
&\lesssim W^{\frac{3^{k-1}-3}{2}}(B \vee d)^{\frac{5 \cdot 3^{k-1}-3}{2}},
\end{align*}
where the last step above follows from $k \leq L$ and $L=O(1)$.\\
Summing the two bounds on $U_1$ and $U_2$ implies that for any $1 \leq l \leq W$, we have:
\begin{equation}
\label{laplacian_inter_step}
\left|\Big(\Delta[\eta_3 \circ F_{k-1}](x)\Big)_{l}\right| \leq U_1 + U_2 \lesssim  W^{\frac{3^{k-1}-3}{2}}(B \vee d)^{\frac{5 \cdot 3^{k-1}-3}{2}}. 
\end{equation}
Taking supremum with respect to $1 \leq l \leq W$ then yields:
\begin{align*}
\|\Delta[F_k](x)\|_{\infty} \leq WB\|\Delta[\eta_3 \circ F_{k-1}](x)\|_{\infty} \lesssim W^{\frac{3^{k-1}-1}{2}}(B \vee d)^{\frac{5 \cdot 3^{k-1}-1}{2}}.      
\end{align*}
Taking supremum with respect to $x \in \Omega$ and $F_k \in \Phi_{k}(L,W,S,B)$ on the LHS implies that the given upper bound also holds for $k$. By induction, the claim is proved.
\end{proof}

\begin{lemma}{(Relation between the covering number of the DNN Laplacian space and parameter space)}
\label{dnn_laplacian_diff}
For any $1 \leq k \leq L$, suppose that a pair of different two networks $F_k,G_k \in \Phi_{k}(L,W,S,B)$ are given by:
\begin{align*}
F_k(x) &:= (\mathcal{W}_{F}^{(k)}\eta_3(\cdot) + b_{F}^{(k)}) \cdots (\mathcal{W}_{F}^{(1)}x + b_{F}^{(1)}),\\
G_k(x) &:= (\mathcal{W}_{G}^{(k)}\eta_3(\cdot) + b_{G}^{(k)}) \cdots (\mathcal{W}_{G}^{(1)}x + b_{G}^{(1)}).
\end{align*}
Furthermore, assume that the $\| \ \|_{\infty}$ norm of the distance between the parameter spaces is uniformly upper bounded by $\delta$, i.e
\begin{equation}
\label{dnn_laplacian_para_constraint} 
\|W_{F}^{(l)} - W_{G}^{(l)}\|_{\infty,\infty} \leq \delta, \ \|b_{F}^{(l)} - b_{G}^{(l)}\|_{\infty} \leq \delta, \ (\forall \ 1 \leq l \leq k).
\end{equation}
Then we have:
\begin{equation}
\label{dnn_laplacian_jacobian_bound}
\sup_{x \in \Omega}\|\Delta[F_k](x)-\Delta[G_k](x)\|_{\infty} = O\Big(\delta W^{\frac{3^{k-1}-1}{2}}(B \vee d)^{\frac{5 \cdot 3^{k-1} - 1}{2}}\Big).
\end{equation}
\end{lemma}
\begin{proof}
We use induction on $k$ to prove the claim. \\
Base case: $k=1$. Note that any $F_{1} \in \Phi_{1}(L,W,S,B)$ is a linear transform, so the Laplacian $\Delta[F_{1}](x)$ must be the zero vector for any $x \in \Omega$. Hence, for any $x \in \Omega$ and any $F_1,G_1 \in \Phi_{1}(L,W,S,B)$, we have:
\begin{align*}
\|\Delta [F_1](x) - \Delta [G_1](x)\|_{\infty} = 0 \lesssim \delta (B \vee d)^2.
\end{align*}
Inductive Step: assume that the claim has been proved for $k-1$, where $2 \leq k \leq L$. Then for any $x \in \Omega$ and $F_k,G_k \in \Phi_{k}(L,W,S,B)$ satisfying constraint \ref{dnn_laplacian_para_constraint}, applying linearity of the Laplacian operator indicates:
\begin{align*}
\|\Delta [F_k](x) - \Delta [G_k](x)\|_{\infty} &= \|\mathcal{W}_{F}^{(k)}\Delta[\eta_3 \circ F_{k-1}](x) - \mathcal{W}_{G}^{(k)}\Delta[\eta_3 \circ G_{k-1}](x)\|_{\infty}\\
&= \left\|\Big(\mathcal{W}_{F}^{(k)} - \mathcal{W}_{G}^{(k)}\Big)\Delta[\eta_3 \circ F_{k-1}](x)\right\|_{\infty}\\
&+ \left\|\mathcal{W}_{G}^{(k)}\Big(\Delta[\eta_3 \circ F_{k-1}](x) -\Delta[\eta_3 \circ G_{k-1}](x)\Big)\right\|_{\infty}\\
&\leq \|\mathcal{W}_{F}^{(k)} - \mathcal{W}_{G}^{(k)}\|_{\infty}\|\Delta[\eta_3 \circ F_{k-1}](x)\|_{\infty}\\
&+ \|\mathcal{W}_{G}^{(k)}\|_{\infty}\|\Delta[\eta_3 \circ F_{k-1}](x) -\Delta[\eta_3 \circ G_{k-1}](x)\|_{\infty}.
\end{align*}
For the first term $ \|\mathcal{W}_{F}^{(k)} - \mathcal{W}_{G}^{(k)}\|_{\infty}\|\Delta[\eta_3 \circ F_{k-1}](x)\|_{\infty}$, applying the bound in equation \ref{laplacian_inter_step} and equation \ref{dnn_laplacian_para_constraint} yields:
\begin{equation}
\label{laplacian_bound_term1}
\begin{aligned}
\|\mathcal{W}_{F}^{(k)} - \mathcal{W}_{G}^{(k)}\|_{\infty}\|\Delta[\eta_3 \circ F_{k-1}](x)\|_{\infty} &\lesssim \delta W \times W^{\frac{3^{k-1}-3}{2}}(B \vee d)^{\frac{5 \cdot 3^{k-1}-3}{2}}\\
&= \delta W^{\frac{3^{k-1}-1}{2}}(B \vee d)^{\frac{5 \cdot 3^{k-1}-3}{2}}. 
\end{aligned}    
\end{equation}
For the second term $\|\mathcal{W}_{G}^{(k)}\|_{\infty}\|\Delta[\eta_3 \circ F_{k-1}](x) -\Delta[\eta_3 \circ G_{k-1}](x)\|_{\infty}$, we need to upper bound the norm $\|\Delta[\eta_3 \circ F_{k-1}](x) -\Delta[\eta_3 \circ G_{k-1}](x)\|_{\infty}$ at first. Note that for any $1 \leq l \leq W$, we can use equation \ref{dnn_laplacian_comp} to write the $l$-th component of $\Delta[\eta_3 \circ F_{k-1}](x) -\Delta[\eta_3 \circ G_{k-1}](x)$ as:
\begin{align*}
\Big(\Delta[\eta_3 \circ F_{k-1}](x) &-\Delta[\eta_3 \circ G_{k-1}](x)\Big)_{l} = \sum_{j=1}^{d}\frac{\partial^2}{\partial x_{j}^2}\eta_3[F_{k-1,l}(x)] - \sum_{j=1}^{d}\frac{\partial^2}{\partial x_{j}^2}\eta_3[G_{k-1,l}(x)]\\
&=6\eta_{1}[F_{k-1,l}(x)]\sum_{j=1}^{d}\Big(\frac{\partial}{\partial x_{j}}F_{k-1,l}(x)\Big)^2- 6\eta_{1}[G_{k-1,l}(x)]\sum_{j=1}^{d}\Big(\frac{\partial}{\partial x_{j}}G_{k-1,l}(x)\Big)^2\\
&+3\eta_2[F_{k-1,l}(x)]\sum_{j=1}^{d}\frac{\partial^2}{\partial x_{j}^2}F_{k-1,l}(x)- 3\eta_2[G_{k-1,l}(x)]\sum_{j=1}^{d}\frac{\partial^2}{\partial x_{j}^2}G_{k-1,l}(x)\\
&=6\eta_{1}[F_{k-1,l}(x)]\sum_{j=1}^{d}\Big(\frac{\partial}{\partial x_{j}}F_{k-1,l}(x)\Big)^2- 6\eta_{1}[G_{k-1,l}(x)]\sum_{j=1}^{d}\Big(\frac{\partial}{\partial x_{j}}F_{k-1,l}(x)\Big)^2\\
&+6\eta_{1}[G_{k-1,l}(x)]\sum_{j=1}^{d}\Big(\frac{\partial}{\partial x_{j}}F_{k-1,l}(x)\Big)^2- 6\eta_{1}[G_{k-1,l}(x)]\sum_{j=1}^{d}\Big(\frac{\partial}{\partial x_{j}}G_{k-1,l}(x)\Big)^2\\
&+3\eta_2[F_{k-1,l}(x)]\sum_{j=1}^{d}\frac{\partial^2}{\partial x_{j}^2}F_{k-1,l}(x)- 3\eta_2[G_{k-1,l}(x)]\sum_{j=1}^{d}\frac{\partial^2}{\partial x_{j}^2}F_{k-1,l}(x)\\
&+3\eta_2[G_{k-1,l}(x)]\sum_{j=1}^{d}\frac{\partial^2}{\partial x_{j}^2}F_{k-1,l}(x)- 3\eta_2[G_{k-1,l}(x)]\sum_{j=1}^{d}\frac{\partial^2}{\partial x_{j}^2}G_{k-1,l}(x).
\end{align*}
We denote the four summations above by $V_1, V_2, V_3$ and $V_4$, respectively:
\begin{align*}
V_1 &:= 6\eta_{1}[F_{k-1,l}(x)]\sum_{j=1}^{d}\Big(\frac{\partial}{\partial x_{j}}F_{k-1,l}(x)\Big)^2- 6\eta_{1}[G_{k-1,l}(x)]\sum_{j=1}^{d}\Big(\frac{\partial}{\partial x_{j}}F_{k-1,l}(x)\Big)^2,\\
V_2 &:= 6\eta_{1}[G_{k-1,l}(x)]\sum_{j=1}^{d}\Big(\frac{\partial}{\partial x_{j}}F_{k-1,l}(x)\Big)^2- 6\eta_{1}[G_{k-1,l}(x)]\sum_{j=1}^{d}\Big(\frac{\partial}{\partial x_{j}}G_{k-1,l}(x)\Big)^2,\\
V_3 &:= 3\eta_2[F_{k-1,l}(x)]\sum_{j=1}^{d}\frac{\partial^2}{\partial x_{j}^2}F_{k-1,l}(x)- 3\eta_2[G_{k-1,l}(x)]\sum_{j=1}^{d}\frac{\partial^2}{\partial x_{j}^2}F_{k-1,l}(x),\\
V_4 &:= 3\eta_2[G_{k-1,l}(x)]\sum_{j=1}^{d}\frac{\partial^2}{\partial x_{j}^2}F_{k-1,l}(x)- 3\eta_2[G_{k-1,l}(x)]\sum_{j=1}^{d}\frac{\partial^2}{\partial x_{j}^2}G_{k-1,l}(x).
\end{align*}
By applying Lemma \ref{dnn_inf_norm}, Lemma \ref{Lipschitz-ReLU}, Lemma \ref{dnn_diff_bound} and Lemma \ref{dnn_grad_inf_norm}, we can upper bound $V_1$ by:
\begin{align*}
V_1 &= 6\Big(\eta_{1}[F_{k-1,l}(x)]- \eta_{1}[G_{k-1,l}(x)]\Big)\sum_{j=1}^{d}\Big(\frac{\partial}{\partial x_{j}}F_{k-1,l}(x)\Big)^2\\
&\leq 6 |F_{k-1,l}(x) - G_{k-1,l}(x)|\left(\sum_{j=1}^{d}\Big|\frac{\partial}{\partial x_{j}}F_{k-1,l}(x)\Big|\right)^2 \leq 6\|F_{k-1}(x) - G_{k-1}(x)\|_{\infty}\|J[F_{k-1}](x)\|_{\infty}^2\\
&\lesssim \delta W^{\frac{3^{k-2}-1}{2}}(B \vee d)^{\frac{5 \cdot 3^{k-2}-1}{2}}2^{\frac{3^{k-1}-1}{2}-k+2}3^{k-2} \times W^{3^{k-2}-1}(B \vee d)^{5 \cdot 3^{k-2}-1}2^{3^{k-1}-1-2k+4}3^{2k-4}\\
&\lesssim \delta W^{\frac{3^{k-1}-3}{2}}(B \vee d)^{\frac{5 \cdot 3^{k-1} - 3}{2}}.
\end{align*}
where the last step above follows from $k \leq L$ and $L=O(1)$.\\
Furthermore, note that for any $1 \leq j \leq d$, we can upper bound the difference $\Big(\frac{\partial}{\partial x_{j}}F_{k-1,l}(x)\Big)^2 - \Big(\frac{\partial}{\partial x_{j}}G_{k-1,l}(x)\Big)^2$ as follows:
\begin{equation}
\label{laplacian_V2_term}
\begin{aligned}
\Big(&\frac{\partial}{\partial x_{j}}F_{k-1,l}(x)\Big)^2 - \Big(\frac{\partial}{\partial x_{j}}G_{k-1,l}(x)\Big)^2 \leq \left|\Big(\frac{\partial}{\partial x_{j}}F_{k-1,l}(x)\Big)^2 - \Big(\frac{\partial}{\partial x_{j}}G_{k-1,l}(x)\Big)^2\right|\\
&= \left|\frac{\partial}{\partial x_{j}}F_{k-1,l}(x) + \frac{\partial}{\partial x_{j}}G_{k-1,l}(x)\right|\left|\frac{\partial}{\partial x_{j}}F_{k-1,l}(x) - \frac{\partial}{\partial x_{j}}G_{k-1,l}(x)\right|\\
&\leq \left(\Big|\frac{\partial}{\partial x_{j}}F_{k-1,l}(x)\Big| + \Big|\frac{\partial}{\partial x_{j}}G_{k-1,l}(x)\Big|\right)\left|\frac{\partial}{\partial x_{j}}F_{k-1,l}(x) - \frac{\partial}{\partial x_{j}}G_{k-1,l}(x)\right|.
\end{aligned}    
\end{equation}
Note that $\eta_1(G_{k-1,l}(x)) \geq 0$. Combining the non-negativity with equation \ref{laplacian_V2_term}, Lemma \ref{dnn_inf_norm}, Lemma \ref{dnn_grad_inf_norm} and Lemma \ref{dnn_grad_diff} helps us upper bound $V_2$ by:
\begin{align*}
V_2 &= 6\eta_{1}[G_{k-1,l}(x)]\sum_{j=1}^{d}\left[\Big(\frac{\partial}{\partial x_{j}}F_{k-1,l}(x)\Big)^2-\Big(\frac{\partial}{\partial x_{j}}G_{k-1,l}(x)\Big)^2\right] \\
&\leq 6\|G_{k-1}(x)\|_{\infty}\sum_{j=1}^{d}\left(\Big|\frac{\partial}{\partial x_{j}}F_{k-1,l}(x)\Big| + \Big|\frac{\partial}{\partial x_{j}}G_{k-1,l}(x)\Big|\right)\left|\frac{\partial}{\partial x_{j}}F_{k-1,l}(x) - \frac{\partial}{\partial x_{j}}G_{k-1,l}(x)\right|\\
&\leq 6\|G_{k-1}(x)\|_{\infty}\left(\sum_{j=1}^{d}\Big|\frac{\partial}{\partial x_{j}}F_{k-1,l}(x)\Big| + \sum_{j=1}^{d}\Big|\frac{\partial}{\partial x_{j}}G_{k-1,l}(x)\Big|\right)\left(\sum_{j=1}^{d}\Big|\frac{\partial}{\partial x_{j}}F_{k-1,l}(x) - \frac{\partial}{\partial x_{j}}G_{k-1,l}(x)\Big|\right)\\
&\leq 6\|G_{k-1}(x)\|_{\infty}\Big(\|J[F_{k-1}](x)\|_{\infty} + \|J[G_{k-1}](x)\|_{\infty}\Big)\Big\|J[F_{k-1}](x)-J[G_{k-1}](x)\Big\|_{\infty}\\
&\leq 6W^{\frac{3^{k-2}-1}{2}}(B \vee d)^{\frac{5 \cdot 3^{k-2}-1}{2}}2^{\frac{3^{k-1}-1}{2}-k+2} \times 2W^{\frac{3^{k-2}-1}{2}}(B \vee d)^{\frac{5 \cdot 3^{k-2}-1}{2}}2^{\frac{3^{k-1}-1}{2}-k+2}3^{k-2}\\
&\times \delta  W^{\frac{3^{k-2}-1}{2}}(B \vee d)^{\frac{5 \cdot 3^{k-2}-1}{2}}2^{\frac{3^{k-1}-1}{2}-k+1}3^{2k-4} \lesssim \delta W^{\frac{3^{k-1}-3}{2}}(B \vee d)^{\frac{5 \cdot 3^{k-1}-3}{2}}.
\end{align*}
where the last step above follows from $k \leq L$ and $L=O(1)$.\\
Moreover, using Lemma \ref{dnn_inf_norm}, Lemma \ref{Lipschitz-ReLU} and Lemma \ref{dnn_laplacian_inf_norm} helps us upper bound $V_3$ by:
\begin{align*}
V_3 &= \Big(3\eta_2[F_{k-1,l}(x)]- 3\eta_2[G_{k-1,l}(x)]\Big)\sum_{j=1}^{d}\frac{\partial^2}{\partial x_{j}^2}F_{k-1,l}(x)\\
&\leq \Big|3\eta_2[F_{k-1,l}(x)]- 3\eta_2[G_{k-1,l}(x)]\Big|\Big|\sum_{j=1}^{d}\frac{\partial^2}{\partial x_{j}^2}F_{k-1,l}(x)\Big|\\
&\leq 6\Big(\sup_{x \in \Omega, \ F_{k-1} \in \Phi_{k-1}(L,W,S,B)}\|F_{k-1}(x)\|_{\infty}\Big)\|F_{k-1}(x) - G_{k-1}(x)\|_{\infty}\|\Delta [F_{k-1}](x)\|_{\infty}\\
&\lesssim 6W^{\frac{3^{k-2}-1}{2}}(B \vee d)^{\frac{5 \cdot 3^{k-2}-1}{2}}2^{\frac{3^{k-1}-1}{2}-k+2} \times \delta W^{\frac{3^{k-2}-1}{2}}(B \vee d)^{\frac{5 \cdot 3^{k-2}-1}{2}}2^{\frac{3^{k-1}-1}{2}-k+2}3^{k-2} \\
&\times W^{\frac{3^{k-2}-1}{2}}(B \vee d)^{\frac{5 \cdot 3^{k-2} - 1}{2}} \lesssim \delta W^{\frac{3^{k-1}-3}{2}}(B \vee d)^{\frac{5 \cdot 3^{k-1} - 3}{2}}
\end{align*}
where the last step above follows from $k \leq L$ and $L=O(1)$.\\
Finally, applying Lemma \ref{dnn_inf_norm} and inductive hypothesis helps us upper bound $V_4$ by:
\begin{align*}
V_4 &= 3\eta_2[G_{k-1,l}(x)]\Big(\sum_{j=1}^{d}\frac{\partial^2}{\partial x_{j}^2}F_{k-1,l}(x)-\sum_{j=1}^{d}\frac{\partial^2}{\partial x_{j}^2}G_{k-1,l}(x)\Big)\\
&\leq 3\|G_{k-1}(x)\|_{\infty}^2\|\Delta[F_{k-1}](x)-\Delta[G_{k-1}](x)\|_{\infty}\\
&\lesssim 3W^{3^{k-2}-1}(B \vee d)^{5 \cdot 3^{k-2}-1}2^{3^{k-1}-1-2k+4} \times \delta  W^{\frac{3^{k-2}-1}{2}}(B \vee d)^{\frac{5 \cdot 3^{k-2}-1}{2}}\\
&\lesssim \delta  W^{\frac{3^{k-1}-3}{2}}(B \vee d)^{\frac{5 \cdot 3^{k-1}-3}{2}}
\end{align*}
where the last step above follows from $k \leq L$ and $L=O(1)$.\\
Combining the four bounds on $V_1, V_2, V_3$ and $V_4$ implies:
\begin{align*}
\Big(\Delta[\eta_3 \circ F_{k-1}](x) &-\Delta[\eta_3 \circ G_{k-1}](x)\Big)_{l} = \sum_{i=1}^{4}V_{i} \lesssim \delta W^{\frac{3^{k-1}-3}{2}}(B \vee d)^{\frac{5 \cdot 3^{k-1}-3}{2}}
\end{align*}
Taking supremum with respect to $1 \leq l \leq W$ gives us an upper bound on the second term $\|\mathcal{W}_{G}^{(k)}\|_{\infty}\|\Delta[\eta_3 \circ F_{k-1}](x) -\Delta[\eta_3 \circ G_{k-1}](x)\|_{\infty}$:
\begin{equation}
\label{laplacian_bound_term2}
\begin{aligned}
\|\mathcal{W}_{G}^{(k)}\|_{\infty}\|\Delta[\eta_3 \circ F_{k-1}](x) -\Delta[\eta_3 \circ G_{k-1}](x)\|_{\infty}  &\lesssim WB \times \delta W^{\frac{3^{k-1}-3}{2}}(B \vee d)^{\frac{5 \cdot 3^{k-1}-3}{2}} \\
&= \delta W^{\frac{3^{k-1}-1}{2}}(B \vee d)^{\frac{5 \cdot 3^{k-1}-1}{2}}
\end{aligned}    
\end{equation}
Combining the two bounds derived in equation \ref{laplacian_bound_term1} and equation \ref{laplacian_bound_term2} then implies:
\begin{align*}
\|\Delta [F_k](x) - \Delta [G_k](x)\|_{\infty} &\lesssim \delta W^{\frac{3^{k-1}-1}{2}}(B \vee d)^{\frac{5 \cdot 3^{k-1}-3}{2}} + \delta W^{\frac{3^{k-1}-1}{2}}(B \vee d)^{\frac{5 \cdot 3^{k-1}-1}{2}}\\
&\lesssim \delta W^{\frac{3^{k-1}-1}{2}}(B \vee d)^{\frac{5 \cdot 3^{k-1}-1}{2}}
\end{align*}
Taking supremum with respect to $x \in \Omega$ on the LHS implies that the given upper bound also holds for $k$. By induction, the claim is proved.
\end{proof}

Given a Neural Network function space $\Phi(L,W,S,B)$, we define a corresponding Neural Network Laplacian space $\Delta \Phi(L,W,S,B)$ as:
\begin{equation}
\label{DNN_laplacian_space}
\Delta \Phi(L,W,S,B) := \{\Delta F \ | \ F \in \Phi(L,W,S,B)\}.
\end{equation}

\begin{theorem}{(Bounding the Neural Network Laplacian space covering number)}
\label{dnn_laplacian_covering_num}
Fix some sufficiently large $N \in \mathbb{Z}^{+}$. Consider a Deep Neural Network space $\Phi(L,W,S,B)$ with $L=O(1), W=O(N),S=O(N)$ and $B=O(N)$. Then the $\log$ value of the covering number of the DNN Laplacian space with respect to the $\|\cdot\|_{\infty}$ norm $\|F(x)\|_{\infty} := \sup_{x \in \Omega}|F(x)|$, which is denoted by $\mathcal{N}(\delta, \Delta \Phi(L,W,S,B), \| \cdot \|_{\infty})$, can be upper bounded by: 
\begin{equation}
\begin{aligned}
\log \mathcal{N}(\delta, \Delta \Phi(L,W,S,B), \| \cdot \|_{\infty})  = O\left(S\Big[\log(\delta^{-1}) + 3^{L}\log(WB)\Big]\right)
\end{aligned}    
\end{equation}
\begin{proof}
We firstly fix a sparsity pattern (i.e, the locations of the non-zero entries are fixed). Applying Lemma \ref{dnn_laplacian_diff} yields that there exists some constant $C=O(1)$, such that the covering number with respect to $\|\cdot\|_{\infty}$ can be upper bounded by:
\begin{align*}
\Big(\frac{\delta}{CW^{\frac{3^{L-1}-1}{2}}(B \vee d)^{\frac{5 \cdot 3^{L-1}-1}{2}}}\Big)^{-S}    
\end{align*}
Furthermore, note that the number of feasible configurations is upper bounded by ${(W+1)^L \choose S} \leq (W+1)^{LS}$\cite{schmidt2020nonparametric,farrell2021deep}. Then we plug this into the pervious estimation and yields:
\begin{align*}
\log \mathcal{N}(\delta, \Phi(L,W,S,B), \| \cdot \|_{\infty}) &\leq \log \left[(W+1)^{LS}\Big(\frac{\delta}{CW^{\frac{3^{L-1}-1}{2}}(B \vee d)^{\frac{5 \cdot 3^{L-1}-1}{2}}}\Big)^{-S}\right]\\
&\leq S\log\Big[\delta^{-1}(W+1)^{L}W^{\frac{3^{L-1}-1}{2}}(B \vee d)^{\frac{5 \cdot 3^{L-1}-1}{2}}\Big]\\
&\lesssim S\Big[\log(\delta^{-1}) + L\log(W) + 3^{L}\log(W(B \vee d))\Big]
\end{align*} 
Note that here the dimension $d$ is some constant. Thus, by plugging in thee given magnitudes $L=O(1), W=O(N),S=O(N)$ and $B=O(N)$, we can further deduce that:
\begin{align*}
\log \mathcal{N}(\delta, \Delta \Phi(L,W,S,B), \| \cdot \|_{\infty}) \lesssim S\Big[\log(\delta^{-1}) + 3^{L}\log(WB)\Big]    
\end{align*}
This finishes our proof.
\end{proof}
\end{theorem}

\begin{lemma}[Local Rademacher Complexity Bound for Deep Ritz Method]

Consider a Deep Neural Network space $\mF(\Omega) = \Phi(L,W,S,B)$ with $L=O(1), W=O(N),S=O(N)$ and $B=O(N)$, where $N \in \mathbb{Z}^{+}$ is fixed to be sufficiently large. Moreover, assume that the gradients and function value of $\mF(\Omega), V$ and $f$ are uniformly bounded 
\begin{equation}
    \max \Big\{\sup_{u \in \mF(\Omega)}\|u\|_{L^{\infty}(\Omega)}, \sup_{u \in \mF(\Omega)}\|\nabla u\|_{L^{\infty}(\Omega)}, \|u^{\ast}\|_{L^{\infty}(\Omega)}, \|\nabla u^{\ast}\|_{L^{\infty}(\Omega)}, V_{max}, \|f\|_{L^{\infty}(\Omega)} \Big \} \leq C.
\end{equation} 
For any $\rho >0$, we consider a localized set $L_{\rho}$ defined by:
$$
\mL_\rho(\Omega):=\{u:u \in \mF(\Omega),\|u-u^\ast\|_{H^1}^2\le \rho\}.
$$
Then for any $\rho \gtrsim n^{-2}$, the Rademacher complexity of a localized function space $\mS_{\rho}(\Omega) := \Big\{h :=|\Omega| \cdot \left[ \frac{1}{2}\Big(\|\nabla u\|^2-\|\nabla u^{\ast}\|^2\Big) + \frac{1}{2}V(|u|^2-|u^{\ast}|^2)-f(u-u^{\ast})\right] \ \ \Big | \ u \in L_{\rho}(\Omega)\Big \}$ can be upper bounded by a sub-root function 
$$
\phi(\rho):= O\left(\sqrt{\frac{S3^L\rho}{n}\log\left(BWn\right)}\right).
$$
\emph{i.e.} we have
\begin{equation}
\phi(4\rho) \leq 2\phi(\rho) \text{ and } R_{n}(\mS_{\rho}(\Omega)) \leq \phi(\rho). \ 
\end{equation}
holds for all $\rho \gtrsim n^{-2}$.
\end{lemma}
\begin{proof}
Firstly, we will check that for any $u \in L_{\rho}(\Omega)$, the corresponding function $h$ in $\mS_{\rho}(\Omega)$ is Lipschitz with respect to $u- u^\ast$ and $\|\nabla u\|-\|\nabla u^\ast\|$. Note that for any $u_1,u_2 \in L_{\rho}(\Omega)$ with corresponding functions $h_1,h_2 \in \mS_{\rho}(\Omega)$, applying boundedness condition \ref{assp:boundedness_drm_local_rad}
yields:
\begin{align*}
|h_1(x) - h_2(x)| &\leq \frac{1}{2}\Big|\|\nabla u_1(x)\|^2 - \|\nabla u_2(x)\|^2\Big| + \frac{1}{2}|V(x)| |u_1(x)^2 - u_2(x)^2| +  |f(x)||u_1(x) - u_2(x)|\\
&\leq C\Big|\|\nabla u_1(x)\| - \|\nabla u_2(x)\|\Big| + (C^2+C)|u_1(x) - u_2(x)|\\
&= C\left|\Big(\|\nabla u_1(x)\|-\|\nabla u^\ast(x)\|\Big) - \Big(\|\nabla u_2(x)\|-\|\nabla u^\ast(x)\|\Big)\right|\\
&+ (C^2+C)\Big|(u_1(x)-u^\ast(x)) - (u_2(x)-u^\ast(x))\Big|.
\end{align*}
Let's pick $L = C^2+C > C$. Applying the Talagrand Contraction Lemma \ref{lem:Talagrand contraction} helps us upper bound the local Rademacher complexity $R_{n}(\mS_{\rho}(\Omega))$ by
\begin{align*}
R_{n}(\mS_\rho(\Omega))  &= \bE_{x}\bE_{\sigma}\left[\sup_{u \in \mL_{\rho}(\Omega)}\frac{1}{n}\sum_{i=1}^{n}\sigma_{i}\Big[ \frac{1}{2}\Big(\|\nabla u\|^2-\|\nabla u^{\ast}\|^2\Big) + \frac{1}{2}V(|u|^2-|u^{\ast}|^2)-f(u-u^{\ast})\Big]\right]\\
&\leq 2L\bE_{x}\bE_{\sigma}\left[\sup_{u \in \mL_{\rho}(\Omega)}\frac{1}{n}\sum_{i=1}^{n}\sigma_{i}\Big(u(x_i)-u^{\ast}(x_i)\Big)\right]\\
&+2L\bE_{x'}\bE_{\sigma'}\left[\sup_{u \in \mL_{\rho}(\Omega)}\frac{1}{n}\sum_{i=1}^{n}\sigma_{i}'\Big(\|\nabla u(x_i')\|- \|\nabla u^{\ast}(x_i')\|\Big)\right]\\
&\lesssim R_n\left(\Big\{u-u_*: u\in \mL_\rho \Big\}\right)+ R_n \left(\Big\{\|\nabla u\|-\|\nabla u^\ast\|:u\in \mL_\rho \Big\}\right)
\end{align*}
From the localization constraint $\rho \geq \|u -u^\ast\|_{H^1(\Omega)}^2 = \|u -u^\ast\|_{L^2(\Omega)}^2 + \|\nabla u -\nabla u^\ast\|_{L^2(\Omega)}^2$, we can deduce that
\begin{equation}
\label{grad_local_rad_ineq_1}
\|u -u^\ast\|_{L^2(\Omega)} \leq \sqrt{\rho} \text{ and } \|\nabla u -\nabla u^\ast\|_{L^2(\Omega)} \leq \sqrt{\rho}    
\end{equation}
Moreover, note that $\Omega \subset [0,1]^{d}$. Applying triangle inequality yields:
\begin{equation}
\label{grad_local_rad_ineq_2}
\begin{aligned}
\Big\|\|\nabla u\| - \|\nabla u^\ast\| \Big\|_{L^2(\Omega)}^2 &= \int_{\Omega}\Big|\|\nabla u(x)\| - \|\nabla u^\ast(x)\|\Big|^2dx \leq \int_{\Omega}\|\nabla u(x) - \nabla u^\ast(x)\|^2 dx \\
&= \|\nabla u - \nabla u^\ast\|_{L^2(\Omega)}^2 \leq \rho  \Rightarrow \Big\|\|\nabla u\| - \|\nabla u^\ast\| \Big\|_{L^2(\Omega)} \leq 
\sqrt{\rho}
\end{aligned}
\end{equation}
Using inequality \ref{grad_local_rad_ineq_1} and inequality \ref{grad_local_rad_ineq_2}, we have:
\begin{align*}
R_{n}(\mS_\rho(\Omega))  &\lesssim R_n\left(\Big\{u-u_*: u\in \mL_\rho \Big\}\right)+ R_n \left(\Big\{\|\nabla u\|-\|\nabla u^\ast\|:u\in \mL_\rho \Big\}\right)\\
&\le R_n \left( \Big\{u-u^\ast: u\in \Phi(L,W,S,B),\|u -u^\ast\|_{L^2(\Omega)} \le \sqrt{\rho}\Big\} \right)\\
&+R_n \left(\Big\{\|\nabla u\|-\|\nabla u^\ast\|: u \in \Phi(L,W,S,B),\Big\|\|\nabla u\| - \|\nabla u^\ast\| \Big\|_{L^2(\Omega)} \leq 
\sqrt{\rho} \Big\}\right)
\end{align*}
\cite{geer2000empirical,rakhlin2017empirical} showed a “upper isometry” property, where the metric $\|\cdot\|_{L_2}$ is equivalent to $\|\cdot\|_{n,2}$ with high probability. Combining this fact with Theorem \ref{thm:dudley}, we can bound the local Rademacher complexities using Dudley integral:
\begin{align*}
R_{n}(\mS_\rho(\Omega)) &\lesssim R_n \left( \Big\{u-u^\ast: u\in \Phi(L,W,S,B),\|u -u^\ast\|_{L^2(\Omega)} \le \sqrt{\rho}\Big\} \right)\\
&+R_n \left(\Big\{\|\nabla u\|-\|\nabla u^\ast\|: u \in \Phi(L,W,S,B),\Big\|\|\nabla u\| - \|\nabla u^\ast\| \Big\|_{L^2(\Omega)} \leq 
\sqrt{\rho} \Big\}\right)\\
&\le R_n \left(\Big\{u-u^\ast: u\in\Phi(L,W,S,B),\|u-u^\ast\|_{n,2} \le 2\sqrt{\rho} \Big\} \right)\\
&+R_n \left(\Big\{\|\nabla u\|-\|\nabla u^\ast\|:u\in\Phi(L,W,S,B),\Big\| \|\nabla u\|-\|\nabla u^\ast\| \Big\|_{n,2} \le 2\sqrt{\rho} \Big\}\right)\\
&\lesssim \inf_{0<\alpha<2\sqrt{\rho}}\big\{4\alpha+\frac{12}{\sqrt{n}}\int_\alpha^{2\sqrt{\rho}}\sqrt{\log\mathcal{N}(\delta,\Phi(L,W,S,B),\|\cdot\|_n)}d\delta \big\}\\
&+ \inf_{0<\alpha<2\sqrt{\rho}}\big\{4\alpha+\frac{12}{\sqrt{n}}\int_\alpha^{2\sqrt{\rho}}\sqrt{\log\mathcal{N}(\delta,\nabla \Phi(L,W,S,B),\|\cdot\|_{n,2})}d\delta \big\}\\
&\lesssim \inf_{0<\alpha<2\sqrt{\rho}}\big\{4\alpha+\frac{12}{\sqrt{n}}\int_\alpha^{2\sqrt{\rho}}\sqrt{\log\mathcal{N}(\delta,\Phi(L,W,S,B),\|\cdot\|_{\infty})}d\delta \big\}\\
&+ \inf_{0<\alpha<2\sqrt{\rho}}\big\{4\alpha+\frac{12}{\sqrt{n}}\int_\alpha^{2\sqrt{\rho}}\sqrt{\log\mathcal{N}(\delta,\nabla \Phi(L,W,S,B),\|\cdot\|_{\infty})}d\delta \big\}
\end{align*}
For any $\rho \gtrsim \frac{1}{n^2}$, we pick $\alpha = \frac{1}{n} \lesssim \sqrt{\rho}$ and plug in the upper bounds proved in Theorem \ref{dnn_covering_num} and Theorem \ref{dnn_grad_covering_num}, which implies:
\begin{align*}
R_{n}(\mS_\rho(\Omega)) &\lesssim \frac{1}{n}+\frac{1}{\sqrt{n}}\int_{\frac{1}{n}}^{2\sqrt{\rho}} \sqrt{S\Big[\log(\delta^{-1}) + 3^{L}\log(WB)\Big]}d\delta+\frac{1}{\sqrt{n}}\int_{\frac{1}{n}}^{2\sqrt{\rho}} \sqrt{S\Big[\log(\delta^{-1}) + 3^{L}\log(WB)\Big]}d \delta\\
&\lesssim \sqrt{\frac{S3^L\rho}{n}\log\left(BWn\right)}    
\end{align*}

\end{proof}

\begin{lemma}[Local Rademacher Complexity Bound for Physics Informed Neural Network]

Consider a Deep Neural Network space $\mF(\Omega) = \Phi(L,W,S,B)$ with $L=O(1), W=O(N),S=O(N)$ and $B=O(N)$, where $N \in \mathbb{Z}^{+}$ is fixed to be sufficiently large. Moreover, assume that the gradients and function value of $\mF(\Omega), V$ and $f$ are uniformly bounded 
\begin{equation}
    \max \Big\{\sup_{u \in \mF(\Omega)}\|u\|_{L^{\infty}(\Omega)}, \sup_{u \in \mF(\Omega)}\|\Delta u\|_{L^{\infty}(\Omega)}, \|u^{\ast}\|_{L^{\infty}(\Omega)}, \|\Delta u^{\ast}\|_{L^{\infty}(\Omega)}, V_{max}, \|f\|_{L^{\infty}(\Omega)} \Big \} \leq C.
\end{equation} 
For any $\rho >0$, we consider a localized set $M_{\rho}$ defined by:
$$
\mM_\rho(\Omega):=\{u:u \in \mF(\Omega),\|u-u^\ast\|_{H^2}^2\le \rho\}.
$$
Then for any $\rho \gtrsim n^{-2}$, the Rademacher complexity of a localized function space $\mT_{\rho}(\Omega) := \Big\{h :=|\Omega| \cdot \left[ (\Delta u -Vu +f)^2-(\Delta u^\ast -Vu^\ast +f)^2\right] \ \ \Big | \ u \in M_{\rho}(\Omega)\Big \}$ can be upper bounded by a sub-root function 
$$
\phi(\rho):= O\left(\sqrt{\frac{S3^L\rho}{n}\log\left(BWn\right)}\right).
$$
\emph{i.e.} we have
\begin{equation}
\phi(4\rho) \leq 2\phi(\rho) \text{ and } R_{n}(\mT_{\rho}(\Omega)) \leq \phi(\rho). \ 
\end{equation}
holds for all $\rho \gtrsim n^{-2}$.
\end{lemma}

\begin{proof}
Firstly, we will check that for any $u \in L_{\rho}(\Omega)$, the corresponding function $h$ in $\mS_{\rho}(\Omega)$ is Lipschitz with respect to $u- u^\ast$ and $\Delta u-\Delta u^\ast$. Note that for any $u_1,u_2 \in L_{\rho}(\Omega)$ with corresponding functions $h_1,h_2 \in \mS_{\rho}(\Omega)$, applying boundedness condition \ref{assp:boundedness_pinn_local_rad} yields:
\begin{align*}
|h_1(x) - h_2(x)| &\leq |\Delta u_1 -\Delta u_2 -V(u_1-u_2)||\Delta u_1 -Vu_1 +\Delta u_2 -Vu_2 +2f|\\
&\leq (2C^2+4C)\left(|\Delta u_1 -\Delta u_2|+ C|u_1-u_2|\right)\\
&= (2C^2+4C)\Big|(\Delta u_1(x)-\Delta u^\ast(x)) - (\Delta u_2(x)-\Delta u^\ast(x))\Big|\\
&+(2C^3+4C^2)\Big|(u_1(x)-u^\ast(x)) - (u_2(x)-u^\ast(x))\Big|
\end{align*}
Let's pick $L = \max\{2C^2+4C,2C^3+4C^2\}$. Applying the Talagrand Contraction Lemma \ref{lem:Talagrand contraction} helps us upper bound the local Rademacher complexity $R_{n}(\mT_{\rho}(\Omega))$ by
\begin{align*}
R_{n}(\mT_\rho(\Omega))  &= \bE_{x}\bE_{\sigma}\Big[\sup_{u \in \mM_{\rho}(\Omega)}\frac{1}{n}\sum_{i=1}^{n}\sigma_{i}\left[ (\Delta u -Vu +f)^2-(\Delta u^\ast -Vu^\ast +f)^2\right]\Big]\\
&\leq 2L\bE_{x}\bE_{\sigma}\left[\sup_{u \in \mM_{\rho}(\Omega)}\frac{1}{n}\sum_{i=1}^{n}\sigma_{i}\Big(u(x_i)-u^{\ast}(x_i)\Big)\right]\\
&+2L\bE_{x'}\bE_{\sigma'}\left[\sup_{u \in \mM_{\rho}(\Omega)}\frac{1}{n}\sum_{i=1}^{n}\sigma_{i}'\Big(\Delta u(x_i')- \Delta u^{\ast}(x_i')\Big)\right]\\
&\lesssim R_n\left(\Big\{u-u_*: u\in \mM_\rho \Big\}\right)+ R_n \left(\Big\{\Delta u-\Delta u^\ast:u\in \mM_\rho \Big\}\right)\\
&\lesssim R_n \left( \Big\{u-u^\ast: u\in \Phi(L,W,S,B),\|u -u^\ast\|_{L^2(\Omega)} \le \sqrt{\rho}\Big\} \right)\\
&+R_n \left(\Big\{\Delta u-\Delta u^\ast: u \in \Phi(L,W,S,B),\|\Delta u -\Delta u^\ast\|_{L^2(\Omega)}  \leq 
\sqrt{\rho} \Big\}\right)\\
&\le R_n \left(\Big\{u-u^\ast: u\in\Phi(L,W,S,B),\|u-u^\ast\|_{n,2} \le 2\sqrt{\rho} \Big\} \right)\\
&+R_n \left(\Big\{\Delta u -\Delta u^\ast:u\in\Phi(L,W,S,B),\|\Delta u -\Delta u^\ast \|_{n,2} \le 2\sqrt{\rho} \Big\}\right)\\
&\lesssim \inf_{0<\alpha<2\sqrt{\rho}}\big\{4\alpha+\frac{12}{\sqrt{n}}\int_\alpha^{2\sqrt{\rho}}\sqrt{\log\mathcal{N}(\delta,\Phi(L,W,S,B),\|\cdot\|_n)}d\delta \big\}\\
&+ \inf_{0<\alpha<2\sqrt{\rho}}\big\{4\alpha+\frac{12}{\sqrt{n}}\int_\alpha^{2\sqrt{\rho}}\sqrt{\log\mathcal{N}(\delta,\Delta \Phi(L,W,S,B),\|\cdot\|_{n,2})}d\delta \big\}\\
&\lesssim \inf_{0<\alpha<2\sqrt{\rho}}\big\{4\alpha+\frac{12}{\sqrt{n}}\int_\alpha^{2\sqrt{\rho}}\sqrt{\log\mathcal{N}(\delta,\Phi(L,W,S,B),\|\cdot\|_{\infty})}d\delta \big\}\\
&+ \inf_{0<\alpha<2\sqrt{\rho}}\big\{4\alpha+\frac{12}{\sqrt{n}}\int_\alpha^{2\sqrt{\rho}}\sqrt{\log\mathcal{N}(\delta,\Delta \Phi(L,W,S,B),\|\cdot\|_{\infty})}d\delta \big\}
\end{align*}
For any $\rho \gtrsim \frac{1}{n^2}$, we pick $\alpha = \frac{1}{n} \lesssim \sqrt{\rho}$ and plug in the upper bounds proved in Theorem \ref{dnn_covering_num} and Theorem \ref{dnn_grad_covering_num}, which implies:
\begin{align*}
R_{n}(\mT_\rho(\Omega)) &\lesssim \frac{1}{n}+\frac{1}{\sqrt{n}}\int_{\frac{1}{n}}^{2\sqrt{\rho}} \sqrt{S\Big[\log(\delta^{-1}) + 3^{L}\log(WB)\Big]}d\delta+\frac{1}{\sqrt{n}}\int_{\frac{1}{n}}^{2\sqrt{\rho}} \sqrt{S\Big[\log(\delta^{-1}) + 3^{L}\log(WB)\Big]}d \delta\\
&\lesssim \sqrt{\frac{S3^L\rho}{n}\log\left(BWn\right)} 
\end{align*}

\end{proof}
\subsection{Proof of The Meta-Theorem for PINN}
\label{appendix:PINNmeta}

\begin{proof}

To upper bound the excess risk $\Delta \mE^{(n)}$, following\cite{xu2020finite,lu2021priori,duan2021convergence}, we decompose the excess risk into approximation error and generalization error with probability $1-e^{-t}$:
\begin{equation}
\label{eq:decomp_pinn}
\begin{aligned}
\Delta \mE^{(n)}(\hat u_{\text{PINN}}) = \mE(\hat u_{\text{PINN}}) -\mE(u^{\star}) &= \big[\mE(\hat u_{\text{PINN}}) - \mE_{n}(\hat u_{\text{PINN}})\big] +\big[ \mE_{n}(\hat u_{\text{PINN}}) - \mE_{n}(u_{\mF})\big] \\
&+ \big[\mE_{n}(u_{\mF}) - \mE(u_{\mF}) \big]+\big[ \mE(u_{\mF}) - \mE(u^{\star})\big]\\
& \leq \big[\mE(\hat u_{\text{PINN}}) - \mE_{n}(\hat u_{\text{PINN}})\big]+\big[\mE_{n}(u_{\mF}) - \mE(u_{\mF})\big] +\big[ \mE(u_{\mF}) - \mE(u^{\star})\big]\\
&\le \big[\mE(\hat \mE(\hat u_{\text{PINN}})-\mE(u^\ast)+\mE_{n}(u^{\ast})- \mE_{n}(\hat u_{\text{PINN}})\big]\\ &+\frac{3}{2}\big[ \mE(u_{\mF}) - \mE(u^{\star})\big]+\frac{t}{2n},
\end{aligned}
\end{equation}
where the expectation is on all sampled data. The inequality of the third line is because the  $u$ is the minimizer of the empirical loss $\mE_{n}$ in the solution set $\mF(\Omega)$, so we have $\mE_{n}(u) \leq \mE_{n}(u_{\mF})$. The last inequality is based on the Bernstein inequality. The variance of $h = |\Omega| \cdot \left[ (\Delta u -Vu +f)^2-(\Delta u^\ast -Vu^\ast +f)^2\right]$ can be bounded by $\big[\mE(u_{\mF}) - \mE(u^{\star})\big]$ due to the strong convexity of the variation objective (\ref{cond: talagrand strong convex_pinn}). According to the Brenstein inequality, we know with probability $1-e^{-t}$ we have
\begin{align*}
\mE_{n}(u_{\mF})-\mE_{n}(u^\ast)-\mE(u_{\mF})+\mE(u^\ast) \le \sqrt{\frac{t\big[\mE(u_{\mF}) - \mE(u^{\star})\big]}{n}} \le \frac{1}{2}\big[\mE(u_{\mF}) - \mE(u^{\star})\big] +\frac{t}{2n}.
\end{align*}

Note that \ref{eq:decomp} holds for all function lies in the function space $\mF$. Thus, we can take $u_{\mF}:=\arg\min_{u_{0} \in \mF(\Omega)}\Big(\mE(u_{0}) - \mE(u^{\star})\Big)$ and finally get
\begin{align*}
\Delta \mE^{(n)} &\le \underbrace{\mE(\hat u_{\text{PINN}})-\mE(u^\ast)+\mE_{n}(u^{\ast})- \mE_{n}(\hat u_{\text{PINN}})}_{\Delta \mE_{\text{gen}}} + \frac{3}{2}\underbrace{\inf_{u_{\mF} \in \mF(\Omega)}\Big(\mE(u_{\mF}) - \mE(u^{\star})\Big)}_{\Delta \mE_{\text{app}} } + \frac{t}{2n}.
\end{align*}
This inequality decompose the excess risk to the generalization error $\Delta \mE_{\text{gen}} := \mE(\hat u_{\text{PINN}})-\mE(u^\ast)+\mE_{n}(u^{\ast})- \mE_{n}(\hat u_{\text{PINN}})$ and the approximation error $\Delta \mE_{\text{app}} = \inf_{u_{\mF} \in \mF(\Omega)}\Big(\mE(u_{\mF}) - \mE(u^{\star})\Big)$. \\
We'll focus on providing fast rate upper bounds of the generalization error for the two estimators using the localization techinque\cite{bartlett2005local,xu2020finite}. To achieve the fast generalization bound, we focus on the following normalized empirical process

\begin{align*}
\tilde{\mT}_{r}(\Omega) := \big\{\tilde{h}(x) := \frac{\mathbb{E}[h]-h(x)}{\bE[h] + r} \ | \ h \in \mT(\Omega)\big\} \ (r > 0).
\end{align*}

First, we try to bound the expectation of the normalized empirical process. Applying the Symmetrization Lemma \ref{lem:radcomp}, we can first bound the expectation as

\begin{align*}
\sup_{\tilde{h} \in \tilde{T}_{r}(\Omega)}\mathbb{E}_{x'}\left[\frac{1}{n}\sum_{i=1}^{n}\tilde{h}(x_i')\right] \leq\mathbb{E}_{x'}\left[\sup_{h \in T(\Omega)}\Big|\frac{1}{n}\sum_{i=1}^{n}\frac{h(x_i')-\bE[h] }{\bE[h] + r}\Big|\right]
\leq 2R_{n}(\hat{\mT}_{r}(\Omega)).
\end{align*}
where the function class $\hat{\mS}_{r}(\Omega)$ is defined as:
\begin{align*}
\hat{\mT}_{r}(\Omega) := \big\{\hat{h}(x) := \frac{h(x)}{\bE[h] + r} \ | \ h \in \mT(\Omega)\big\},    
\end{align*}

where $ \mT(\Omega)=\Big\{h :=|\Omega| \cdot \left[ (\Delta u -Vu +f)^2-(\Delta u^\ast -Vu^\ast +f)^2\right]\Big \}.$ Then Applying the Peeling Lemma to any function $h \in \mT(\Omega)$ helps us upper bound the local Rademacher complexity $R_{n}(\hat{\mT}_{r}(\Omega))$ with the function $\phi$ defined in equation \ref{peelingcond_pinn}:
$$
R_{n}(\hat{\mT}_{r}(\Omega)) = \bE_{\sigma}\left[\bE_{x}\Big[\sup_{h \in \mT(\Omega)}\frac{\frac{1}{n}\sum_{i=1}^{n}\sigma_{i}h(x_i)}{\bE[h] + r}\Big]\right] \leq \frac{4\phi(r)}{r}. 
$$
Combining all inequalities derived above yields:
\begin{equation}
\sup_{\tilde{h} \in \tilde{T}_{r}(\Omega)}\mathbb{E}_{x'}\left[\frac{1}{n}\sum_{i=1}^{n}\tilde{h}(x_i')\right]\leq 2R_{n}(\hat{\mT}_{r}(\Omega)) \leq  \frac{8\phi(r)}{r} \ (r > 0).
\end{equation}

Secondly we'll apply the Talagrand concentration inequality, which requires us to verify the condition needed. We will first check that the expectation value $\bE[h]$ is always non-negative for any $h \in \mS(\Omega)$:
\begin{align*}
\bE[h]&= \frac{1}{|\Omega|}\int_{\Omega}|\Omega| \cdot (\frac{1}{2} |\nabla u(x)|^2 + \frac{1}{2} V(x) |u(x)|^2- f(x)u(x))dx\\
&-\frac{1}{|\Omega|}\int_{\Omega}|\Omega| \cdot (\frac{1}{2} |\nabla u^{\star}(x)|^2 + \frac{1}{2} V(x) |u^{\star}(x)|^2- f(x)u^{\star}(x))dx\\
&= \mE(u) -\mE(u^\star) \geq 0 \Rightarrow \bE[h] \geq 0.
\end{align*}
We will proceed to verify that any $\tilde{h}=\frac{\bE[h]-h}{\bE[h]+r} \in \tilde{\mT}_{r}(\Omega)$ is of bounded inf-norm. We need to prove that any $h \in \mT(\Omega)$ is of bounded inf-norm beforehand. Using boundedness condition listed in equation \ref{assp:boundedness_pinn} implies:
\begin{align*}
\|h\|_{\infty} &= |\Omega| \cdot \|(\Delta u -Vu +f)^2-(\Delta u^\ast -Vu^\ast +f)^2\|_{\infty} = |\Omega|\cdot \|(\Delta u -Vu +f)^2\|_{\infty}\\
&\leq |\Omega| \cdot (\|\Delta u\|_{\infty} + V_{\text{max}}\|u\|_{\infty} + \|f\|_{\infty})^2 \leq |\Omega|(V_{\text{max}} + 2)^2C^2
\end{align*}
By taking $M := |\Omega|(V_{\text{max}} + 2)^2C^2$, we then have $\|h\|_{\infty} \leq M$ for all $h \in \mT(\Omega)$. Note that the denominator can be lower bounded by $|\bE[h]+r| \geq r > 0$. Combining these two inequalities help us upper bound the inf-norm $\|\tilde{h}\|_{\infty} = \sup_{x \in \Omega}|\tilde{h}(x)|$ as follows:
\begin{align*}
\|\tilde{h}\|_{\infty} = \frac{\|\bE[h]-h\|_{\infty}}{|\bE[h] + r|} \leq \frac{2\|h\|_{\infty}}{r} \leq \frac{2M}{r} =: \beta.    
\end{align*}

We will then check the normalized functions $\frac{\mathbb{E}[h]-h(x)}{\bE[h] + r}$ in $\tilde{T}_{r}(\Omega)$ have bounded second moment, which is satisfied because of the regularity results of the PDE. We aim to show that there exist some constants $\alpha,
\alpha' > 0$, such that for any $h \in \mT(\Omega)$, the following inequality holds:
\begin{equation}\label{cond: talagrand strong convex_pinn}
\alpha \mathbb{E}[h^2] \leq \|u-u^\ast\|_{H^2(\Omega)}^2 \leq \alpha'\mathbb{E}[h].     
\end{equation}
The RHS of the inequality follows from strong convexity of the PINN objective function proved in Theorem \ref{thm: PDE regularity_pinn}:
\begin{align*}
\mathbb{E}[h] = \mE(u) - \mE(u^\ast) \geq \frac{1}{\min\{1,C_{\min}\}}\|u-u^\ast\|_{H^2(\Omega)}^2
\end{align*}

The LHS of the inequality follows from boundedness condition listed in equation \ref{assp:boundedness_pinn} and the QM-AM inequality:
\begin{align*}
\mathbb{E}[h^2] &= \int_{\Omega}\left[ (\Delta u -Vu +f)^2-(\Delta u^\ast -Vu^\ast +f)^2\right]^2 dx = \int_{\Omega}(\Delta u -Vu +f)^4dx\\
&\leq M^2 \int_{\Omega}(\Delta u -Vu -\Delta u^\ast + Vu^\ast)^2dx \leq 2M^2\int_{\Omega}[(\Delta u -\Delta u^\ast)^2 + V^2(u-u^\ast)^2]dx\\
&\leq 2M^2\max\{1,V_{\text{max}}^2\}\|u-u^{\ast}\|_{H^2(\Omega)}^2
\end{align*}
By picking $\alpha' = \frac{1}{\min\{1,C_{\min}\}}$ and $\alpha = \frac{1}{M^2\max\{1,V_{\text{max}}^2\}}$, we have finished proving inequality \ref{cond: talagrand strong convex_pinn}. Then we can can upper bound the expectation  $\bE[\tilde{h}^2]$ as:

\begin{align*}
\mathbb{E}[\tilde{h}^2] = \frac{\mathbb{E}[(h-\mathbb{E}[h])^2]}{|\bE[h]+r|^2} = \frac{\bE[h^2] - \bE[h]^2}{|\bE[h]+r|^2} \leq \frac{\bE[h^2]}{|\bE[h]+r|^2}.    
\end{align*}
Using the fact that $\bE[h] \geq 0$ and inequality \ref{cond: talagrand strong convex_pinn}, we can lower bound the denominator $|\bE[h]+r|^2$ as follows:
\begin{align*}
|\bE[h]+r|^2 \geq 2\bE[h]r \geq \frac{2r\alpha}{\alpha'}\mathbb{E}[h^2].    
\end{align*}
Therefore, we can deduce that:
\begin{align*}
\mathbb{E}[\tilde{h}^2]  \leq \frac{\mathbb{E}[h^2]}{|\bE[h]+r|^2} \leq \frac{\bE[h^2]}{\frac{2r\alpha}{\alpha'}\bE[h^2]} = \frac{\alpha'}{2r\alpha} =: \sigma^2.     
\end{align*}
Hence, any function in the localized class $\tilde{\mT}_{r}(\Omega)$ is of bounded second moment.

It is easy to check that for any $\tilde{h} \in \tilde{\mT}_{r}(\Omega)$, we have
\begin{align*}
\bE[\tilde{h}] =  \frac{\bE[h]-\bE[h]}{\bE[h]+r} = 0,
\end{align*}
\emph{i.e.} any function in the localized class $\tilde{\mS}_{r}(\Omega)$ is of zero mean.

Now we have verified that any function $\tilde{h} \in \tilde{\mS}_{r}(\Omega)$ satisfies all the required conditions. By taking $\mu$ to be the uniform distribution on the domain $\Omega$ and applying Talagrand's Concentration inequality given in Lemma \ref{lem:Talagrand ineq}, we have:
\begin{align*}
\mathbb{P}_{x}\left[\sup_{\tilde{h} \in \tilde{\mT}_{r}(\Omega)}\frac{1}{n}\sum_{i=1}^{n}\tilde{h}(x_i) \geq 2\sup_{\tilde{h} \in \tilde{\mT}_{r}(\Omega)}\mathbb{E}_{x'}\Big[\frac{1}{n}\sum_{i=1}^{n}\tilde{h}(x_i')\Big]+ \sqrt{\frac{2t\sigma^2}{n}}+\frac{2t\beta}{n}\right] \leq e^{-t}.
\end{align*}

By using the upper bound deduced above and plugging in the expressions of $\beta$ and $\sigma$, we can rewrite Talagrand's Concentration Inequality in the following way. With probability at least $1-e^{-t}$, the inequality below holds:
\begin{align*}
\frac{1}{n}\sum_{i=1}^{n}\tilde{h}(x_i) \leq \sup_{\tilde{h} \in \tilde{\mS}_{r}(\Omega)}\frac{1}{n}\sum_{i=1}^{n}\tilde{h}(x_i) &\leq 2\sup_{\tilde{h} \in \tilde{\mS}_{r}(\Omega)}\mathbb{E}_{x'}\Big[\frac{1}{n}\sum_{i=1}^{n}\tilde{h}(x_i')\Big]+ \sqrt{\frac{2t\sigma^2}{n}}+\frac{2t\beta}{n}\\
&\leq \frac{16\phi(r)}{r} + \sqrt{\frac{t\alpha'}{n\alpha r}}+\frac{4Mt}{nr} =: \psi(r)    
\end{align*}
Let's pick the critical radius $r_0$ to be:
\begin{equation}
\label{thresholdradius_pinn}
\begin{aligned}
r_0 = \max\{2^{14}r^{\ast}, \frac{24Mt}{n}, \frac{36\alpha' t}{\alpha n}\}. 
\end{aligned}    
\end{equation}
Note that concavity of the function $\phi$ implies that $\phi(r) \leq r$ for any $r \geq r^{\ast}$. Combining this with the first inequality listed in \ref{peelingcond_pinn} yields:
\begin{align*}
\frac{16\phi(r)}{r} &\leq \frac{2^{11} \phi(\frac{r_0}{2^{14}})}{2^{14}\frac{r_0}{2^{14}}} = \frac{1}{8} \times \frac{\phi(\frac{r_0}{2^{14}})}{\frac{r_)}{2^{14}}} \leq \frac{1}{8}.
\end{align*}
On the other hand, applying equation \ref{thresholdradius_pinn} yields:
\begin{align*}
\sqrt{\frac{\alpha' t}{n \alpha r_0}} &\leq \sqrt{\frac{\alpha' t}{n \alpha}\frac{\alpha n}{36\alpha' t}} = \frac{1}{6},\\
\frac{4Mt}{nr_0} &\leq \frac{4Mt}{n} \times \frac{n}{24Mt} = \frac{1}{6}.
\end{align*}
Summing the three inequalities above implies:
\begin{align*}
\psi(r_0) = \frac{16\phi(r_0)}{r_0} + \sqrt{\frac{t\alpha'}{n\alpha r_0}}+\frac{4Mt}{nr_0} \leq \frac{1}{8}+\frac{1}{6}+\frac{1}{6} < \frac{1}{2}.
\end{align*}
By picking $r=r_0$, we can further deduce that for any function $u \in \mF(\Omega)$, the following inequality holds with probability $1-e^{-t}$: 
\begin{align*}
\frac{\mE(u) - \mE(u^\ast)- \mE_{n}(u) + \mE_{n}(u^\ast)}{\mE(u) - \mE(u^\ast)+r_0} = \frac{1}{n}\sum_{i=1}^{n}\tilde{h}(x_i) &\leq \psi(r_0) < \frac{1}{2}.
\end{align*}
Multiplying the denominator on both sides indicates:
\begin{align*}
\Delta \mE_{\text{gen}} = \mE(u) - \mE(u^\ast)- \mE_{n}(u) + \mE_{n}(u^\ast) \leq \frac{1}{2}\Big[\mE(u) - \mE(u^\ast)\Big] + \frac{1}{2}r_0= \frac{1}{2}\Delta \mE^{(n)} + \frac{1}{2}r_0.
\end{align*}
Substituting the upper bound above into the decomposition $\Delta \mE^{(n)} \leq \Delta E_{\text{gen}} + \frac{3}{2}\Delta E_{\text{app}}+\frac{t}{2n}$ yields that with probability $1-e^{-t}$, we have:
\begin{align*}
\Delta \mE^{(n)} \leq \Delta \mE_{\text{gen}} + \frac{3}{2}\Delta \mE_{\text{app}} + \frac{t}{2n} \leq \frac{1}{2}\Delta \mE^{(n)} + \frac{1}{2}r_0 + \frac{3}{2}\Delta \mE_{\text{app}}+\frac{t}{2n}.    
\end{align*}
Simplifying the inequality above yields that with probability $1-e^{-t}$, we have:
\begin{align*}
\Delta \mE^{(n)} &\leq r_0 + 3\Delta \mE_{\text{app}} + \frac{t}{n} = 3\inf_{u_{\mF} \in \mF(\Omega)}\Big(\mE(u_{\mF}) - \mE(u^{\star})\Big) + \max\{2^{14}r^{\ast}, 24M\frac{t}{n}, \frac{36\alpha'}{\alpha}\frac{t}{n}\} + \frac{t}{n}\\
&\lesssim \inf_{u_{\mF} \in \mF(\Omega)}\Big(\mE(u_{\mF}) - \mE(u^{\star})\Big)+ \max\Big\{r^*,\frac{t}{n}\Big\}
\end{align*}
Moreover, using strong convexity of the PINN objective function proved in Theorem \ref{thm: PDE regularity_drm} implies:
\begin{align*}
\Delta \mE^{(n)} = \mE(\hat u_{\text{PINN}})-\mE(u^\ast) \geq \frac{1}{\{1,C_{\min}\}}\|\hat u_{\text{PINN}}-u^\ast\|_{H^1(\Omega)}^2     
\end{align*}
Combining the two bounds above yields that with probability $1-e^{-t}$, we have:
\begin{align*}
\|\hat u_{\text{PINN}}-u^\ast\|_{H^1(\Omega)}^2  \lesssim  \inf_{u_{\mF} \in \mF(\Omega)}\Big(\mE(u_{\mF}) - \mE(u^{\star})\Big)+ \max\Big\{r^*,\frac{t}{n}\Big\}  
\end{align*}
\end{proof}

\subsection{Proof of The Meta-Theorem for MDRM}
\label{appendix:MDRMmeta}
\begin{proof}

To upper bound the excess risk $\Delta \mE^{(N,n)}:=\mE(\hat u_{\text{MDRM}})-\mE(u^\ast)$, following\cite{xu2020finite,lu2021priori,duan2021convergence}, we decompose the excess risk into approximation error and generalization error with probability $1-e^{-t}$:
\begin{equation}
\label{eq:decomp}
\begin{aligned}
\Delta \mE^{(N,n)} =\big[ \mE(\hat u_{\text{MDRM}}) -\mE(u^{\star})\big] &= \big[\mE(\hat u_{\text{MDRM}}) - \mE_{N,n}(\hat u_{\text{MDRM}})\big] +\big[ \mE_{N,n}(\hat u_{\text{MDRM}}) - \mE_{N,n}(u_{\mF})\big] \\
&+ \big[\mE_{N,n}(u_{\mF}) - \mE(u_{\mF}) \big]+\big[ \mE(u_{\mF}) - \mE(u^{\star})\big]\\
& \leq \big[\mE(\hat u_{\text{MDRM}}) - \mE_{N,n}(\hat u_{\text{MDRM}})\big]+\big[\mE_{N,n}(u_{\mF}) - \mE(u_{\mF})\big] +\big[ \mE(u_{\mF}) - \mE(u^{\star})\big]\\
&\le \big[\mE(\hat u_{\text{MDRM}})-\mE(u^\ast)+\mE_{N,n}(u^{\ast})- \mE_{N,n}(\hat u_{\text{MDRM}})]\big]\\ &+2\big[ \mE(u_{\mF}) - \mE(u^{\star})\big]+\frac{4t}{\min\{N,n\}},
\end{aligned}
\end{equation}
where the expectation is on all sampled data. The inequality of the third line is because $\hat u_{\text{MDRM}}$ is the minimizer of the empirical loss $\mE_{n}$ in the solution set $\mF(\Omega)$, so we have $\mE_{N,n}(\hat u_{\text{MDRM}}) \leq \mE_{N,n}(u_{\mF})$. The last inequality is based on the Bernstein inequality. For any $u_{\mF} \in \mF(\Omega)$, we use $h_{\mF,1},h_{\mF,2}$ to denote the following two functions:
\begin{align*}
h_{\mF,1}: &= \frac{1}{2}\Big(\|\nabla u_{\mF}\|^2-\|\nabla u^{\ast}\|^2\Big),\\
h_{\mF,2}: &= \frac{1}{2}V(|u_{\mF}|^2-|u^{\ast}|^2)-f(u_{\mF}-u^{\ast}).
\end{align*}
Applying Bernstein's inequality twice to $h_{\mF,1}$ and $h_{\mF,2}$ implies that there exists some constant $C_q$, such that with probability $1-2e^{-C_q t}$, the following two inequalities hold simultaneously:
\begin{align*}
\mE_{N}(h_{\mF,1})-\mE(h_{\mF,1}) &\le \sqrt{\frac{t\frac{\alpha}{\alpha'}\mE[h_{\mF,1}^2]}{N}},\\
\mE_{n}(h_{\mF,2})-\mE(h_{\mF,2}) &\le \sqrt{\frac{t\frac{\alpha}{\alpha'}\mE[h_{\mF,2}^2]}{n}}.     
\end{align*}
Note that the variance sum $\mE[h_{\mF,1}^2] + \mE[h_{\mF,2}^2]$ can be upper bounded by $\frac{\alpha'}{\alpha}\big[\mE(u_{\mF}) - \mE(u^{\star})\big]$ due to the strong convexity of the variation objective (\ref{modified cond: talagrand strong convex}). Adding the two inequalities above implies with probability $1-2e^{-C_q t}$ we have:
\begin{align*}
\mE_{N,n}(u_{\mF})-\mE_{N,n}(u^\ast)-\mE(u_{\mF})+\mE(u^\ast) &= \mE_{N}(h_{\mF,1})-\mE(h_{\mF,1})+\mE_{n}(h_{\mF,2})-\mE(h_{\mF,2})\\
&\le \sqrt{\frac{t\frac{\alpha}{\alpha'}\mE[h_{\mF,1}^2]}{N}} + \sqrt{\frac{t\frac{\alpha}{\alpha'}\mE[h_{\mF,2}^2]}{n}}\\
&\le \sqrt{\frac{2t\frac{\alpha}{\alpha'}\Big(\mE[h_{\mF,1}^2]+\mE[h_{\mF,1}^2]\Big)}{\min\{N,n\}}}\\
&\le \sqrt{\frac{2t\big[\mE(u_{\mF}) - \mE(u^{\star})\big]}{\min\{N,n\}}} \le \big[\mE(u_{\mF}) - \mE(u^{\star})\big] +\frac{4t}{\min\{N,n\}}.
\end{align*}
Note that \ref{eq:decomp} holds for all function lies in the function space $\mF$. Thus, we can take $u_{\mF}:=\arg\min_{u_{\mF} \in \mF(\Omega)}\Big(\mE(u_{\mF}) - \mE(u^{\star})\Big)$ and finally get:
\begin{align*}
\Delta \mE^{(N,n)} &\le  \underbrace{\mE(\hat u_{\text{MDRM}})-\mE(u^\ast)+\mE_{N,n}(u^{\ast})- \mE_{N,n}(\hat u_{\text{MDRM}})}_{\Delta \mE_{\text{gen}}} + 2\underbrace{\inf_{u_{\mF} \in \mF(\Omega)}\Big(\mE(u_{\mF}) - \mE(u^{\star})\Big)}_{\Delta \mE_{\text{app}} }+\frac{4t}{n}.
\end{align*}
This inequality decomposes the excess risk to the generalization error $\Delta \mE_{\text{gen}} := \mE(\hat u_{\text{MDRM}})-\mE(u^\ast)+\mE_{N,n}(u^{\ast})- \mE_{N,n}(\hat u_{\text{MDRM}})$ and the approximation error $\Delta \mE_{\text{app}} = \inf_{u_{\mF} \in \mF(\Omega)}\Big(\mE(u_{\mF}) - \mE(u^{\star})\Big)$. 
From the lemmata proved in Section \ref{appendix:approximation}, we already have an estimation of the approximation error's convergence rate. So now we'll focus on providing fast rate upper bounds of the generalization error for the two estimators using the localization techinque\cite{bartlett2005local,xu2020finite}. To achieve the fast generalization bound, we focus on the following two normalized empirical processes:
\begin{align*}
\tilde{\mS}_{r,1}(\Omega) &:= \big\{\tilde{h}_{1}(x) := \frac{\mathbb{E}[h_1]-h_1(x)}{\bE[h_1] +\bE[h_2] + r} \ | \ (h_1,h_2) \in \mS(\Omega)\big\} \ (r > 0),\\
\tilde{\mS}_{r,2}(\Omega) &:= \big\{\tilde{h}_{2}(x) := \frac{\mathbb{E}[h_2]-h_2(x)}{\bE[h_1] + \bE[h_2] + r} \ | \ (h_1,h_2) \in \mS(\Omega)\big\} \ (r > 0).
\end{align*}
where the space $\mS(\Omega)$ is defined as:
\begin{align*}
\mS(\Omega) := \Big\{(h_1,h_2)\big|&h_1 :=|\Omega| \cdot \left[ \frac{1}{2}\Big(\|\nabla u\|^2-\|\nabla u^{\ast}\|^2\Big)\right], \\
&h_2 :=|\Omega| \cdot \left[\frac{1}{2}V(|u|^2-|u^{\ast}|^2)-f(u-u^{\ast})\right], u \in \mF(\Omega) \Big \}.    
\end{align*}
First, we try to bound the expectation of the two normalized empirical processes. Applying the Symmetrization Lemma \ref{lem:radcomp}, we can first bound the two expectations as:

\begin{align*}
\sup_{\tilde{h}_{1} \in \tilde{S}_{r,1}(\Omega)}\mathbb{E}_{y'}\left[\frac{1}{N}\sum_{i=1}^{N}\tilde{h}_{1}(y_i')\right] &\leq\mathbb{E}_{y'}\left[\sup_{h_1 \in S_{1}(\Omega)}\Big|\frac{1}{N}\sum_{i=1}^{N}\frac{h_{1}(y_i')-\bE[h_1] }{\bE[h_1]+\bE[h_2] + r}\Big|\right]
\leq 2R_{N}(\hat{\mS}_{r,1}(\Omega)),\\
\sup_{\tilde{h}_{2} \in \tilde{S}_{r,2}(\Omega)}\mathbb{E}_{y}\left[\frac{1}{n}\sum_{j=1}^{n}\tilde{h}_{2}(y_j)\right] &\leq\mathbb{E}_{y}\left[\sup_{h_2 \in S_2(\Omega)}\Big|\frac{1}{n}\sum_{i=1}^{n}\frac{h_{2}(y_j)-\bE[h_2] }{\bE[h_1] +\bE[h_2] + r}\Big|\right]
\leq 2R_{n}(\hat{\mS}_{r,2}(\Omega)).
\end{align*}
where the function classes $\hat{\mS}_{r,k}(\Omega) \ (1 \leq k \leq 2)$ are defined as:
\begin{align*}
\hat{\mS}_{r,1}(\Omega) &:= \big\{\hat{h}_{1}(x) := \frac{h_1(x)}{\bE[h_1] +\bE[h_2] + r} \ | \ (h_1,h_2) \in \mS(\Omega)\big\},\\    
\hat{\mS}_{r,2}(\Omega) &:= \big\{\hat{h}_{2}(x) := \frac{h_2(x)}{\bE[h_1] +\bE[h_2] + r} \ | \ (h_1,h_2) \in \mS(\Omega)\big\}.
\end{align*}

Applying the modified Peeling Lemma \ref{lem:genpeel} to any function $h=(h_1,h_2) \in \mS(\Omega)$ helps us upper bound the sum of the two local Rademacher complexities $R_{N}(\hat{\mS}_{r,1}(\Omega)) + R_{n}(\hat{\mS}_{r,2}(\Omega))$ with the function $\phi$ defined in equation \ref{peelingcond_mdrm}:
\begin{align*}
R_{N}(\hat{\mS}_{r,1}(\Omega)) + R_{n}(\hat{\mS}_{r,2}(\Omega)) &= \bE_{\sigma}\left[\bE_{y}\Big[\sup_{h \in \mS(\Omega)}\frac{\frac{1}{N}\sum_{i=1}^{N}\sigma_{i}h_{1}(y_i)}{\bE[h_1] + \bE[h_2] + r}\Big]\right]+\bE_{\tau}\left[\bE_{y'}\Big[\sup_{h \in \mS(\Omega)}\frac{\frac{1}{n}\sum_{j=1}^{n}\tau_{j}h_{2}(y_j')}{\bE[h_1] + \bE[h_2] + r}\Big]\right]\\
&= \bE_{\sigma}\left[\bE_{y,y'}\Big[\sup_{h \in \mS(\Omega)}\frac{\frac{1}{N}\sum_{i=1}^{N}\sigma_{i}h_{1}(y_i)}{\bE[h_1] + \bE[h_2] + r} + \sup_{h \in \mS(\Omega)}\frac{\frac{1}{n}\sum_{j=1}^{n}\tau_{j}h_{2}(y_j')}{\bE[h_1] + \bE[h_2] + r}\Big]\right]\\
&= R_{N,n}(\hat{\mS}_{r}(\Omega)) \leq \frac{4\phi(r)}{r}.
\end{align*}
Combining all inequalities derived above yields:
\begin{equation}
\begin{aligned}
\sup_{\tilde{h}_{1} \in \tilde{S}_{r,1}(\Omega)}&\mathbb{E}_{y'}\left[\frac{1}{N}\sum_{i=1}^{N}\tilde{h}_{1}(y_i')\right] + \sup_{\tilde{h}_{2} \in \tilde{S}_{r,2}(\Omega)}\mathbb{E}_{y}\left[\frac{1}{n}\sum_{j=1}^{n}\tilde{h}_{2}(y_j)\right] \\
&\leq 2R_{N}(\hat{\mS}_{r,1}(\Omega)) + 2R_{n}(\hat{\mS}_{r,2}(\Omega))=  2R_{N,n}(\hat{\mS}_{r}(\Omega)) \leq \frac{8\phi(r)}{r} \ (r > 0).
\end{aligned}
\end{equation}

Secondly we'll apply the Talagrand concentration inequality to the two function classes $\tilde{S}_{r,1}(\Omega)$ and $\tilde{S}_{r,2}(\Omega)$, which requires us to verify the conditions needed. We will first check that the expectation sum $\bE[h_1]+\bE[h_2]$ is always non-negative for any $(h_1,h_2) \in \mS(\Omega)$:
\begin{align*}
\bE[h_1]+\bE[h_2]&= \frac{1}{|\Omega|}\int_{\Omega}|\Omega| \cdot (\frac{1}{2} \|\nabla u(x)\|^2 + \frac{1}{2} V(x) |u(x)|^2- f(x)u(x))dx\\
&-\frac{1}{|\Omega|}\int_{\Omega}|\Omega| \cdot (\frac{1}{2} \|\nabla u^{\star}(x)\|^2 + \frac{1}{2} V(x) |u^{\star}(x)|^2- f(x)u^{\star}(x))dx\\
&= \mE(u) -\mE(u^\star) \geq 0 \Rightarrow \bE[h_1] + \bE[h_2] \geq 0.
\end{align*}
Next, We will verify that $\tilde{S}_{r,1}(\Omega)$ satisfies all three requirements. At first, we will show that any $\tilde{h}_{1}=\frac{\bE[h_1]-h_1}{\bE[h_1] +\bE[h_2]+r} \in \tilde{\mS}_{r,1}(\Omega)$ is of bounded inf-norm. We need to prove that any $h_1 \in \mS_{1}(\Omega)$ is of bounded inf-norm beforehand. Using boundedness condition listed in equation \ref{assp:boundedness_modified} implies:
\begin{align*}
\|h_1\|_{\infty} &= \|\frac{1}{2}\Big(\|\nabla u\|^2-\|\nabla u^{\ast}\|^2\Big)\|_{\infty} \leq \frac{1}{2}\Big(\|\nabla u \|_{\infty}^2 + \|\nabla u^{\ast}\|_{\infty}^2 \Big) \leq C^2.
\end{align*}
By taking $M_1 := C^2$, we then have $\|h_1\|_{\infty} \leq M_1$ for all $h_1 \in \mS_{1}(\Omega)$. Note that the denominator of $\tilde{h}_1$ can be lower bounded by $|\bE[h_1] + \bE[h_2]+r| \geq r > 0$. Combining these two inequalities help us upper bound the inf-norm $\|\tilde{h}_1\|_{\infty} = \sup_{x \in \Omega}|\tilde{h}_1(x)|$ as follows:
\begin{align*}
\|\tilde{h}_1\|_{\infty} = \frac{\|\bE[h_1]-h_1\|_{\infty}}{|\bE[h_1] + \bE[h_2] + r|} \leq \frac{2\|h_1\|_{\infty}}{r} \leq \frac{2M_1}{r} =: \beta_1.    
\end{align*}
Also, it is easy to check that for any $\tilde{h}_1 \in \tilde{\mS}_{r,1}(\Omega)$, we have
\begin{align*}
\bE[\tilde{h}_1] =  \frac{\bE[h_1]-\bE[h_1]}{\bE[h_1] + \bE[h_2]+r} = 0,
\end{align*}
\emph{i.e.} any function in the localized class $\tilde{\mS}_{r,1}(\Omega)$ is of zero mean.\\
Moreover, we take $\sigma_{1}^2 = \sup_{\tilde{h}_1 \in \tilde{\mS}_{r,1}(\Omega)}\bE[\tilde{h}_1^2]$ to be the upper bound on the second moment of functions in $\tilde{\mS}_{r,1}(\Omega)$. Now we have verified that any function $\tilde{h}_{1} \in \tilde{\mS}_{r,1}(\Omega)$ satisfies all the required conditions. By taking $\mu$ to be the uniform distribution on the domain $\Omega$ and applying Talagrand's Concentration inequality given in Lemma \ref{lem:Talagrand ineq}, we have:
\begin{equation}
\label{talagrand1}
\mathbb{P}_{x}\left[\sup_{\tilde{h}_{1} \in \tilde{\mS}_{r,1}(\Omega)}\frac{1}{N}\sum_{i=1}^{N}\tilde{h}_{1}(x_i) \geq 2\sup_{\tilde{h}_{1} \in \tilde{\mS}_{r,1}(\Omega)}\mathbb{E}_{y}\Big[\frac{1}{N}\sum_{i=1}^{N}\tilde{h}_{1}(y_i)\Big]+ \sqrt{\frac{2t\sigma_{1}^2}{N}}+\frac{2t\beta_{1}}{N}\right] \leq e^{-t}.
\end{equation}
Moreover, We will verify that $\tilde{S}_{r,2}(\Omega)$ also satisfies all three requirements. At first, we will show that any $\tilde{h}_{2}=\frac{\bE[h_2]-h_2}{\bE[h_1] + \bE[h_2]+r} \in \tilde{\mS}_{r,2}(\Omega)$ is of bounded inf-norm. We need to prove that any $h_2 \in \mS_{2}(\Omega)$ is of bounded inf-norm beforehand. Using boundedness condition listed in equation \ref{assp:boundedness_modified} implies:
\begin{align*}
\|h_2\|_{\infty} &= \|\frac{1}{2}V(|u|^2-|u^{\ast}|^2)-f(u-u^{\ast})\|_{\infty}\\
&\leq \frac{1}{2}V_{\text{max}}\Big(\|u\|_{\infty}^2+ \|u^{\ast}\|_{\infty}^2\Big) + \|f\|_{\infty}\Big(\|u\|_{\infty}+\|u^{\ast}\|_{\infty}\Big)\\
&\leq \frac{1}{2}V_{\text{max}} \times 2C^2 + 2C^2 = (V_{\text{max}} + 2)C^2.
\end{align*}
By taking $M_2 := (V_{\text{max}} + 2)C^2$, we then have $\|h_2\|_{\infty} \leq M_2$ for all $h_2 \in \mS_{2}(\Omega)$. Note that the denominator of $\tilde{h}_2$ can be lower bounded by $|\bE[h_1] + \bE[h_2]+r| \geq r > 0$. Combining these two inequalities help us upper bound the inf-norm $\|\tilde{h}_2\|_{\infty} = \sup_{x \in \Omega}|\tilde{h}_2(x)|$ as follows:
\begin{align*}
\|\tilde{h}_2\|_{\infty} = \frac{\|\bE[h_2]-h_2\|_{\infty}}{|\bE[h_1] + \bE[h_2] + r|} \leq \frac{2\|h_2\|_{\infty}}{r} \leq \frac{2M_2}{r} =: \beta_2.    
\end{align*}
Also, it is easy to check that for any $\tilde{h}_2 \in \tilde{\mS}_{r,2}(\Omega)$, we have
\begin{align*}
\bE[\tilde{h}_2] =  \frac{\bE[h_2]-\bE[h_2]}{\bE[h_1] + \bE[h_2]+r} = 0,
\end{align*}
\emph{i.e.} any function in the localized class $\tilde{\mS}_{r,2}(\Omega)$ is of zero mean.\\
Moreover, we take $\sigma_{2}^2 = \sup_{\tilde{h}_2 \in \tilde{\mS}_{r,2}(\Omega)}\bE[\tilde{h}_2^2]$ to be the upper bound on the second moment of functions in $\tilde{\mS}_{r,2}(\Omega)$. Now we have verified that any function $\tilde{h}_{2} \in \tilde{\mS}_{r,2}(\Omega)$ satisfies all the required conditions. By taking $\mu$ to be the uniform distribution on the domain $\Omega$ and applying Talagrand's Concentration inequality given in Lemma \ref{lem:Talagrand ineq}, we have:
\begin{equation}
\label{talagrand2}
\mathbb{P}_{x'}\left[\sup_{\tilde{h}_{2} \in \tilde{\mS}_{r,2}(\Omega)}\frac{1}{n}\sum_{j=1}^{n}\tilde{h}_{2}(x_j') \geq 2\sup_{\tilde{h}_{2} \in \tilde{\mS}_{r,2}(\Omega)}\mathbb{E}_{y'}\Big[\frac{1}{n}\sum_{j=1}^{n}\tilde{h}_{2}(y_j')\Big]+ \sqrt{\frac{2t\sigma_{2}^2}{n}}+\frac{2t\beta_{2}}{n}\right] \leq e^{-t}.
\end{equation}
By applying a union bound to the two inequalities derived in \ref{talagrand1} and \ref{talagrand2}, we can derive that with probability at least $1-2e^{-t}$, the inequality below holds:
\begin{align*}
\frac{1}{N}\sum_{i=1}^{N}\tilde{h}_{1}(x_i') + \frac{1}{n}\sum_{j=1}^{n}\tilde{h}(x_j) &\leq \sup_{\tilde{h}_{1} \in \tilde{\mS}_{r,1}(\Omega)}\frac{1}{N}\sum_{i=1}^{N}\tilde{h}_{1}(x_i) + \sup_{\tilde{h}_{2} \in \tilde{\mS}_{r,2}(\Omega)}\frac{1}{n}\sum_{j=1}^{n}\tilde{h}_{2}(x_j')\\  &\leq 2\sup_{\tilde{h}_{1} \in \tilde{\mS}_{r,1}(\Omega)}\mathbb{E}_{y}\Big[\frac{1}{N}\sum_{i=1}^{N}\tilde{h}_{1}(y_i)\Big]+ \sqrt{\frac{2t\sigma_1^2}{N}}+\frac{2t\beta_1}{N}\\
&+ 2\sup_{\tilde{h}_{2} \in \tilde{\mS}_{r,2}(\Omega)}\mathbb{E}_{y'}\Big[\frac{1}{n}\sum_{j=1}^{n}\tilde{h}_{2}(y_j')\Big]+ \sqrt{\frac{2t\sigma_{2}^2}{n}}+\frac{2t\beta_{2}}{n}\\
&\leq \frac{16\phi(r)}{r} + \sqrt{\frac{2t}{n}}(\sigma_1 + \sigma_2)+\frac{2t(\beta_1 + \beta_2)}{n}.   
\end{align*}
By the definition of $\beta_1$ and $\beta_2$, we have that the term $\frac{2t(\beta_1 + \beta_2)}{n}$ can be upper bounded by:
\begin{align*}
\frac{2t(\beta_1 + \beta_2)}{n} = \frac{4t(M_1 + M_2)}{nr} \leq \frac{4(V_{\text{max}} + 3)C^2t}{nr}.    
\end{align*}
Now we will derive some upper bound on the sum $\sigma_1 + \sigma_2$. By definition we have that:
\begin{align*}
(\sigma_1 + \sigma_2)^2 \leq 2(\sigma_1^2+\sigma_2^2) &= 2\Big[\sup_{\tilde{h}_1 \in \tilde{\mS}_{r,1}(\Omega)}\bE[\tilde{h}_1^2] + \sup_{\tilde{h}_2 \in \tilde{\mS}_{r,2}(\Omega)}\bE[\tilde{h}_2^2]\Big]\\
&= 2\Big[\sup_{h \in \mS(\Omega)}\frac{\bE[h_1^2] - \bE[h_1]^2}{|\bE[h_1] + \bE[h_2]+r|^2} + \sup_{h \in \mS(\Omega)}\frac{\bE[h_2^2] - \bE[h_2]^2}{|\bE[h_1] + \bE[h_2]+r|^2}\Big]\\
&\leq 4\sup_{h \in \mS(\Omega)}\frac{\bE[h_1^2] + \bE[h_2^2]}{|\bE[h_1] + \bE[h_2]+r|^2}.
\end{align*}
Now it suffices to derive an upper bound of $\frac{\bE[h_1^2] + \bE[h_2^2]}{|\bE[h_1] + \bE[h_2]+r|^2}$ for any $h \in \mS(\Omega)$. The existence of such an upper bound is guaranteed because of the regularity results of the PDE. We aim to show that there exist some constants $\alpha,
\alpha' > 0$, such that for any $h \in \mS(\Omega)$, the following inequality holds:
\begin{equation}\label{modified cond: talagrand strong convex}
\alpha(\mathbb{E}[h_1^2] +\bE[h_2^2]) \leq \|u-u^\ast\|_{H^1(\Omega)}^2 \leq \alpha'(\bE[h_1]+\bE[h_2]).     
\end{equation}
The RHS of the inequality follows from strong convexity of the DRM objective function proved in Theorem \ref{thm: PDE regularity_drm}:
\begin{align*}
\bE[h_1]+\bE[h_2] = \mE(u) - \mE(u^\ast) \geq \frac{\min\{1, V_{\text{min}}\}}{4}\|u-u^\ast\|_{H^1(\Omega)}^2.
\end{align*}
The LHS of the inequality follows from boundedness condition listed in equation \ref{assp:boundedness_modified} and the QM-AM inequality:
\begin{align*}
\mathbb{E}[h_1^2] + \bE[h_2^2] &= \int_{\Omega}\frac{1}{4}\Big(\|\nabla u\|^2-\|\nabla u^{\ast}\|^2\Big)^2dx+ \int_{\Omega}\left[\frac{1}{2}V(|u|^2-|u^{\ast}|^2)-f(u-u^{\ast})\right]^2 dx\\
&\leq \frac{1}{4}\int_{\Omega}\Big(\|\nabla u\|^2-\|\nabla u^{\ast}\|^2\Big)^2dx + \frac{1}{2}\int_{\Omega}V^2(|u|^2-|u^{\ast}|^2)^2dx + 2\int_{\Omega}f^2(u-u^{\ast})^2dx\\
&\leq \frac{1}{4}\int_{\Omega}\Big|\|\nabla u\| - \|\nabla u^{\ast}\|\Big|^2(\|\nabla u\| + \|\nabla u^{\ast}\|)^2dx + \frac{1}{2}V_{\text{max}}^2\int_{\Omega}\Big||u| - |u^{\ast}|\Big|^2(|u| + |u^{\ast}|)^2dx\\
&+2C^2\int_{\Omega}(u-u^{\ast})^2dx \leq C^2\int_{\Omega}\|\nabla u - \nabla u^{\ast}\|^2dx + 2C^2(1+V_{\text{max}}^2)\int_{\Omega}|u-u^{\ast}|^2dx\\
&\leq 2C^2(1+V_{\text{max}}^2)\|u-u^{\ast}\|_{H^1(\Omega)}^2.
\end{align*}
By picking $\alpha' = \frac{4}{\min\{1, V_{\text{min}}\}}$ and $\alpha = \frac{1}{2C^2(1+V_{\text{max}}^2)}$, we have finished proving inequality \ref{modified cond: talagrand strong convex}. Then we can can upper bound the term $\frac{\bE[h_1^2] + \bE[h_2^2]}{|\bE[h_1] + \bE[h_2]+r|^2}$ as:
\begin{align*}
\frac{\bE[h_1^2] + \bE[h_2^2]}{|\bE[h_1] + \bE[h_2]+r|^2}\leq \frac{\frac{\alpha'}{\alpha}\Big(\bE[h_1] + \bE[h_2]\Big)}{2r\Big(\bE[h_1] + \bE[h_2]\Big)} \leq \frac{\alpha'}{2\alpha r}.    
\end{align*}
Combining the bounds derived above helps us upper bound the term $\sqrt{\frac{2t}{n}}(\sigma_1 +\sigma_2)$ as below:
\begin{align*}
\sqrt{\frac{2t}{n}}(\sigma_1 +\sigma_2) \leq \sqrt{\frac{8t}{n}}\sqrt{\sup_{h \in \mS(\Omega)}\frac{\bE[h_1^2] + \bE[h_2^2]}{|\bE[h_1] + \bE[h_2]+r|^2}} \leq \sqrt{\frac{4\alpha' t}{n\alpha r}}   
\end{align*}
Thus, using the two upper bounds on $\sqrt{\frac{2t}{n}}(\sigma_1 +\sigma_2)$ and $\frac{2t(\beta_1 + \beta_2)}{n}$, we have
\begin{align*}
\frac{1}{N}\sum_{i=1}^{N}\tilde{h}_{1}(x_i') + \frac{1}{n}\sum_{j=1}^{n}\tilde{h}(x_j) &\leq \frac{16\phi(r)}{r} + \sqrt{\frac{2t}{n}}(\sigma_1 + \sigma_2)+\frac{2t(\beta_1 + \beta_2)}{n}\\
&\leq \frac{16\phi(r)}{r} + \sqrt{\frac{4\alpha' t}{n \alpha r}} + \frac{4(V_{\text{max}} + 3)C^2t}{nr} = \psi(r)
\end{align*}

Let's pick the critical radius $r_0$ to be:
\begin{equation}
\label{thresholdradius_mdrm}
\begin{aligned}
r_0 = \max\{2^{14}r^{\ast}, \frac{24Mt}{n}, \frac{144\alpha' t}{\alpha n}\}. 
\end{aligned}    
\end{equation}
Note that concavity of the function $\phi$ implies that $\phi(r) \leq r$ for any $r \geq r^{\ast}$. Combining this with the first inequality listed in \ref{peelingcond_mdrm} yields:
\begin{align*}
\frac{16\phi(r_0)}{r_0} &\leq \frac{2^{11} \phi(\frac{r_0}{2^{14}})}{2^{14}\frac{r_0}{2^{14}}} = \frac{1}{8} \times \frac{\phi(\frac{r_0}{2^{14}})}{\frac{r_)}{2^{14}}} \leq \frac{1}{8}.
\end{align*}
On the other hand, applying equation \ref{thresholdradius_mdrm} yields:
\begin{align*}
\sqrt{\frac{4\alpha' t}{n \alpha r_0}} &\leq \sqrt{\frac{4\alpha' t}{n \alpha}\frac{\alpha n}{144\alpha' t}} = \frac{1}{6},\\
\frac{4(V_{\text{max}} + 3)C^2t}{nr_0} &\leq \frac{4(V_{\text{max}} + 3)C^2t}{n} \times \frac{n}{24(V_{\text{max}} + 3)C^2t} = \frac{1}{6}.
\end{align*}
Summing the three inequalities above implies:
\begin{align*}
\psi(r_0) = \frac{16\phi(r_0)}{r_0} + \sqrt{\frac{4\alpha' t}{n \alpha r_0}} + \frac{4(V_{\text{max}} + 3)C^2t}{nr_0} \leq \frac{1}{8}+\frac{1}{6}+\frac{1}{6} < \frac{1}{2}.
\end{align*}
By picking $r=r_0$, we can further deduce that for any function $u \in \mF(\Omega)$, the following inequality holds with probability $1-e^{-t}$: 
\begin{align*}
\frac{\mE(u) - \mE(u^\ast)- \mE_{n}(u) + \mE_{n}(u^\ast)}{\mE(u) - \mE(u^\ast)+r_0} = \frac{1}{n}\sum_{i=1}^{n}\tilde{h}(x_i) &\leq \psi(r_0) < \frac{1}{2}.
\end{align*}
Multiplying the denominator on both sides indicates:
\begin{align*}
\Delta \mE_{\text{gen}} = \mE(u) - \mE(u^\ast)- \mE_{n}(u) + \mE_{n}(u^\ast) \leq \frac{1}{2}\Big[\mE(u) - \mE(u^\ast)\Big] + \frac{1}{2}r_0= \frac{1}{2}\Delta \mE^{(n)} + \frac{1}{2}r_0.
\end{align*}
Substituting the upper bound above into the decomposition $\Delta \mE^{(n)} \leq \Delta \mE_{\text{gen}} + 2\Delta \mE_{\text{app}}+\frac{4t}{n}$ yields that with probability $1-3e^{-\min\{1,C_q\}t}$, we have:
\begin{align*}
\Delta \mE^{(n)} \leq \Delta \mE_{\text{gen}} + \frac{3}{2}\Delta \mE_{\text{app}} + \frac{t}{2n} \leq \frac{1}{2}\Delta \mE^{(n)} + \frac{1}{2}r_0 + \frac{3}{2}\Delta \mE_{\text{app}}+\frac{t}{2n}.    
\end{align*}
Simplifying the inequality above yields that with probability $1-3e^{-\min\{1,C_q\}t}$, we have:
\begin{align*}
\Delta \mE^{(n)} &\leq r_0 + 3\Delta \mE_{\text{app}} + \frac{t}{n} = 3\inf_{u_{\mF} \in \mF(\Omega)}\Big(\mE(u_{\mF}) - \mE(u^{\star})\Big) + \max\{2^{14}r^{\ast}, 24M\frac{t}{n}, \frac{36\alpha'}{\alpha}\frac{t}{n}\} + \frac{t}{n}\\
&\lesssim \inf_{u_{\mF} \in \mF(\Omega)}\Big(\mE(u_{\mF}) - \mE(u^{\star})\Big)+ \max\Big\{r^*,\frac{t}{n}\Big\}
\end{align*}
Moreover, using strong convexity of the DRM objective function proved in Theorem \ref{thm: PDE regularity_drm} implies:
\begin{align*}
\Delta \mE^{(n)} = \mE(\hat u_{\text{MDRM}})-\mE(u^\ast) \gtrsim\|\hat u_{\text{MDRM}}-u^\ast\|_{H^1(\Omega)}^2     
\end{align*}
Combining the two bounds above yields that with probability $1-3e^{-\min\{1,C_q\}t}$, we have:
\begin{align*}
\|\hat u_{\text{MDRM}}-u^\ast\|_{H^1(\Omega)}^2  \lesssim  \inf_{u_{\mF} \in \mF(\Omega)}\Big(\mE(u_{\mF}) - \mE(u^{\star})\Big)+ \max\Big\{r^*,\frac{t}{n}\Big\}  
\end{align*}
\end{proof}

\section{Intuition Behind the Sub-optimality of the Unmodified Deep Ritz Methods}
\label{appendix:subopt}

In this section, we aim to discuss the intuition behind the sub-optimality of the unmodified DRM via using the truncation Fourier basis. To simplify the notation, in this section we consider the following simplest Poisson equation $\Delta u = f$ on the hypercube with zero Dirichlet boundary condition. To illustrate the necessity of the modification we made, we consider the difference between the following two estimators
\begin{itemize}
    \item \textbf{Estimator 1.} We use the truncated Fourier basis estimator to learn the right hand side function $f$ and then we invert the PDE exactly to get the estimated $u$. 
    \item \textbf{Estimator 2.} We plug in a parametrization of the truncated fourier basis into the empirical DRM objective 
\end{itemize}

We would like to point out that \emph{estimator 1 isn't build for computational consideration}. Instead, we use it to consider the statistical limit of our sampled data. We first show that the estimator 1 can achieve the minimax optimal estimation error.

\paragraph{Error Of Estimator 1} Firstly, we show that if one wants to learn the function $u$ in $H_0^1$ norm, one need to learn the right hand side function $f$ in $H_0^{-1}$ norm. The $H_0^{-1}$ norm is defined as the dual norm of the $H^{1}$ norm, \emph{i.e.} $\|u\|_{H_0^{-1}}=\max_{\|v\|_{H_0^1}\le 1} \left<u,v\right>$. Once we assume we have an estimate $\hat f$ of $f$ in $H_0^{-1}$, we can have an estimate of $u$ via $\hat u:=\left(\Delta\right)^{-1}\hat f$, whose distance to $u$ in the $H^1$ norm satisfies:
\begin{align*}
\| \nabla u-\nabla\hat u\|_{H_0^1} &=    \max_{\|v\|_{H_0^1}\le 1} \left<\nabla u-\nabla\hat u,\nabla v\right>\\
&= \max_{\|v\|_{H_0^1}\le 1} \left<\Delta u-\Delta\hat u,v\right>\\
&=\max_{\|v\|_{H_0^1}\le 1} \left<f-\hat f,v\right>=\|f-\hat f\|_{H_{-1}}.
\end{align*}

Estimator 1 using the truncated fourier estimator to estimate the right hand side function $f$. Suppose we can access a random sample of observed data as $\{x_i,f(x_i)\}_{i=1}^n$, then the Fourier coefficient $f_z:=\left<u,\phi_z\right>$ can be estimated as $\hat f_z:=\frac{1}{n}\sum_{i=1}^n f(x_i)\phi_z(x_i) $. To bound the estimation error of $\hat f:=\sum_{\|z\|_\infty\le Z} \hat  f_z\phi_z$ in $H_0^{-1}$, we first apply the bias-variance decomposition:

\begin{align*}
\mathbb{E}\|\hat f-f\|_{H_0^{-1}}^2\le \|\mathbb{E} \hat f-\hat f\|_{H_0^{-1}}^2 +\mathbb{E}\|f-\mathbb{E}\hat f\|_{H_0^{-1}}^2
\end{align*}

We first bound the bias term $\|\mathbb{E} \hat f-f\|_{H^{-1}}^2$. Given $\mathbb{E} \hat f=\sum_{\|z\|_\infty\le Z}f_z\phi_z$, we have that for a truncation set $Z$ of the from $\mathcal{Z}:=\{ z\in\mathbb{N}^d| \|z\|_\infty\le Z \}$, the bias term can be controlled by:

\begin{align*}
\|\sum_{\|z\|_\infty>Z} f_z\phi_z\|_{H^{-1}}^2\le  C \sum_{\|z\|_\infty>Z} f_z^2z^{-2}\le \|z\|^{-2(s-1)}\|f\|_{H_{\alpha-2}}^2
\end{align*}

Next we estimate the variance of the estimator by decomposing the variance into the following sum:
\begin{align*}
\mathbb{E}\|f-\hat f\|_{H_{-1}}^2\le \mathbb{E}\sum_{\|z\|_\infty\le Z} (\hat f_z-f_z)^2\|\phi_z\|_{H_{-1}}^2\le \sum_{\|z\|_\infty\le Z}|z|^{-1}\text{Var}(\hat f_z). 
\end{align*}
Finally we achieve a $Z^{-2(s-1)}+\frac{Z^{d-2}}{n}$ upper bound for estimator 1. With optimal selection of $Z$, we can achieve the min-max optimal convergence rate $n^{-\frac{2s-2}{d+2s-4}}$.

\paragraph{Difference Between Estimator 1 and Estimator 2} Next we aim to understand the Deep Ritz Method objective function via plugging in a truncated Fourier series estimator. We consider an estimator of the form $u=\sum \hat u_z\phi_z(x)$, which lies in the space of truncated fourier series. Then the empirical DRM objective function can be expressed as 
\begin{equation}
\label{eq:drmfourier}
    \frac{1}{2n}\sum_{i=1}^n \left(\sum_z \hat u_z\nabla \phi_z(x_i)\right)^2+\sum_z \hat u_z\phi_z(x_i)f(x_i).
\end{equation}

We observe that (\ref{eq:drmfourier}) is a quadratic formula with respect to the Fourier coefficients $\bm{u}:=(u_z)_{\|z\|_\infty\le Z}$. Thus, we can rewrite it as the following matrix form
\begin{equation}
    \label{eq:drmfouriermatrix}
    \min \frac{1}{2}\bm{u}^\top \hat A \bm{u}+\bm{u}^\top \hat f\text{, where } \hat A=\left(\frac{1}{n}\sum_{i=1}^n\nabla \phi_i (x_i)\nabla \phi_j(x_i)\right)_{\|i\|_\infty\le Z,\|j\|_\infty\le Z}.
\end{equation}

Based on the matrix formulation \ref{eq:drmfouriermatrix}, we can compare the solution given by the two estimators
\begin{itemize}
    \item \textbf{Estimator 1:} The Fourier coefficients of the solution of Estimator 1 are
    \begin{equation}
    \hat {\bm{u}}_1= \text{diag}
    \left(\|z\|^2\right)_{\|z\|_\infty\le Z}^{-1}\hat f.
    \end{equation}
    \item \textbf{Estimator 2:} The Fourier coefficients of the solution of Estimator 2 are 
    \begin{equation}
        \hat{\bm{u}}_2 = \hat A^{-1} \hat f.
    \end{equation}
\end{itemize}
Note that $\mathbb{E} \hat A=\left(\|z\|^2\right)_{\|z\|_\infty\le Z}$. Thus, we can further introduce another variance from the sampling of $A$. By directly estimating $\hat{\bm{u}}_1-\hat{\bm{u}}_2$, we will show that this term will be larger than the final convergence rate. Notice that
\begin{equation}
    \|\hat{\bm{u}}_1-\hat{\bm{u}}_2\|_{H^1}^2 = f^\top \left((\mathbb{E}\hat A)^{-1}-\hat A^{-1}\right)^\top \text{diag}
    \left(\|z\|^2\right)_{\|z\|_\infty\le Z} \left((\mathbb{E}\hat A)^{-1}-\hat A^{-1}\right) f
\end{equation}

Next we aim to bound $\left((\mathbb{E}\hat A)^{-1}-\hat A^{-1}\right)$. We first use the Matrix Bernstein Inequality\cite{tropp2015introduction} to bound the $H^1$ distance between $\hat{\bm{u}}_1$ and $\hat{\bm{u}}_2$. According to the Matrix Bernstein Inequality, we have that with probability $1-e^{-t}$, the following inequality holds
\begin{equation}
    \left\|\left((\mathbb{E}\hat A)-\hat A\right)\right\|_{\bm{H}}\le\sqrt{\frac{Z^d}{n}}+\frac{t}{n},
\end{equation} 

where $\|\cdot\|_{\bm{H}}$ is the matrix operator norm respect to the vector $\|\cdot\|_{\bm{H}}$ defined as $\|z\|_{\bm{H}}^2 = z^\top\text{diag}
    \left(\|z\|^2\right)_{\|z\|_\infty\le Z}^{-1} z$. Note that
\begin{equation}
\left(I+(\mathbb{E}\hat A)^{-1}\left(\hat A-(\mathbb{E}\hat A)\right) \right)\left((\mathbb{E}\hat A)^{-1}-\hat A^{-1}\right) = (\mathbb{E}\hat A)^{-1}\left(\hat A-(\mathbb{E}\hat A)\right)(\mathbb{E}\hat A)^{-1}
\end{equation}

When $n$ is large enough, we know that $\frac{1}{2}I\leqslant  I+(\mathbb{E}\hat A)^{-1}\left(\hat A-(\mathbb{E}\hat A)\right)\leqslant   I$ with high probability. Thus the term $\|\hat u_1-\hat u_2\|_{H^1}^2$ is at the scale of $ \left\|\left((\mathbb{E}\hat A)-\hat A\right)\right\|_{\bm{H}}^2\approx\frac{Z^d}{n}$, which is of the same magnitude as what we get from the empirical process approach in our main proof. It is also larger than $\frac{Z^{d-2}}{n}$, which is the magnitude of the variance term for $\hat u_1$. Therefore, here we conjecture that the our bound for DRM itself is tight and leads to the sub-optimal convergence rate.

\section{Preliminaries on Tools for Lower Bounds}
\label{appendix:fano}
In this section, we repeat the standard tools we use to establish the lower bound. The main tool we use is the Fano's inequailty and the Varshamov-Gilber Lemma.

\begin{lemma}[Fano's methods] Assume that $V$ is a unifrom random variable over set $\mathcal{V}$, then for any markov chain $V\rightarrow X\rightarrow \hat V$, we always have
$$
\mathcal{P}(\hat V\not = V)\ge 1-\frac{I(V;X)+\log 2}{\log(|\mathcal{V}|)}
$$

\end{lemma}

\begin{lemma}[Varshamov-Gillbert Lemma,\cite{tsybakov2008introduction} Theorem 2.9] Let $D\ge 8$. There exists a subset $\mathcal{V}=\{\tau^{(0)},\cdots,\tau^{(2^{D/8})}\}$ of $D-$dimensional hypercube $\mathcal{H}^D=\{0,1\}^D$ such that $\tau^{(0)}=(0,0,\cdots,0)$ and the $\ell_1$ distance between every two elements is larger than $\frac{D}{8}$
$$
\sum_{l=1}^D \|\tau^{(j)}-\tau^{(k)}\|_{\ell_1}\ge \frac{D}{8}\text{, for all }0\le j,k \le 2^{D/8}
$$

\end{lemma}

\end{document}